\documentclass[12pt]{amsart}
\usepackage[margin=1in]{geometry}
\usepackage[utf8]{inputenc}
\usepackage{amsfonts,amsmath,amsthm,amssymb}
\usepackage[unicode]{hyperref}
\usepackage{listings}
\usepackage{lastpage}
\usepackage{tikz}
\usepackage[justification=centering]{caption}
\usetikzlibrary{cd}
\usepackage[mathscr]{euscript}
\usepackage{xcolor}
\usepackage{graphicx, subcaption}
\usepackage{mathabx}

\graphicspath{{./pics/}}	

\newcommand{\C}{{\mathbb C}}
\newcommand{\I}{{\mathbb I}}
\newcommand{\D}{{\mathbb D}}
\renewcommand{\H}{{\mathbb H}}
\newcommand{\Z}{{\mathbb Z}}
\newcommand{\R}{{\mathbb R}}
\newcommand{\Rose}{\mathcal{R}}
\newcommand{\N}{{\mathbb N}}
\renewcommand{\D}{{\mathbb D}}
\renewcommand{\S}{\mathbb{S}}
\renewcommand{\P}{\mathcal{P}}
\newcommand{\Homeo}{\operatorname{Homeo}}

\newcommand{\id}{\operatorname{id}}
\newcommand{\Rg}{\operatorname{Rg}}
\newcommand{\Dom}{\operatorname{Dom}}

\newcommand{\Teich}{\operatorname{Teich}}
\newcommand{\inter}{\operatorname{int}}

\renewcommand{\Re}{\operatorname{Re}}

\newcommand{\alphal}{\alpha{\uparrow}}
\newcommand{\gammal}{\gamma{\uparrow}}
\newcommand{\deltal}{\delta{\uparrow}}
\renewcommand{\mod}{\mathrm{mod}}

\newcommand{\hide}[1]{}

\newcounter{main}

\theoremstyle{plain}
        \newtheorem{theorem}{Theorem}[section]
        \newtheorem*{theorem*}{Theorem}
        \newtheorem*{conj*}{Conjecture}
        \newtheorem{lemma}[theorem]{Lemma}
        \newtheorem{corollary}[theorem]{Corollary}
        \newtheorem{proposition}[theorem]{Proposition}
        \newtheorem{maintheorem}[main]{Main Theorem}        

\theoremstyle{definition}
        \newtheorem{definition}[theorem]{Definition}
        \newtheorem*{definition*}{Definition}

\theoremstyle{remark}
        \newtheorem{remark}[theorem]{Remark}

        \newtheorem{example}[theorem]{Example}

        \newtheorem*{example*}{Example}
        \newtheorem*{examples*}{Examples}        
        \newtheorem*{claim}{Claim}
        \newtheorem*{claim1}{Claim 1}
        \newtheorem*{claim2}{Claim 2}
        \newtheorem*{claim3}{Claim 3}

\newenvironment{subproof}[1][\proofname]{%
  \begin{proof}[#1]%
}{%
  \end{proof}%
}

\title[Dynamical approximations of postsingularly finite entire maps]{Dynamical approximations of postsingularly finite entire maps}

\author{Malavika Mukundan}
\address {Dept. of Mathematics, 530 Church Street, University of Michigan Ann Arbor, MI, 48109}
\email{malavim@umich.edu}

\author{Nikolai Prochorov}
\address {Aix--Marseille Univ., Institut de Math\'{e}matiques de Marseille, 
13003 Marseille, France}
\email{nikolai.prochorov@etu.univ-amu.fr}

\author{Bernhard Reinke}
\address{Max Planck Institute for Mathematics in the Sciences, 
Inselstrasse 22, Leipzig, 04107}
\email{bernhard.reinke@mis.mpg.de}

\date{\today}


\begin{document}

\begin{abstract}
    We prove that every postsingularly finite entire map $g$ can be approximated by a sequence of postcritically finite complex polynomials $(g_n)$ such that their postsingular dynamics $g|P_g$ and $g_n|P_{g_n}$ are conjugate for every $n \in \N$. To establish this result, we introduce the notion of combinatorial convergence for sequences of entire Thurston maps defined on the topological plane $\R^2$ and having the same marked set $A$. We prove that if such a sequence $(f_n)$ converges combinatorially to a Thurston map $f$, then the sequence of Thurston pullback maps $(\sigma_{f_n})$ converges to $\sigma_f$ locally uniformly on the Teichm\"{u}ller space $\Teich(\R^2, A)$.
\end{abstract}
\maketitle

\tableofcontents

\newpage
\section{Introduction}

The family of complex polynomials has some of the most well-understood dynamics among all entire maps. At the same time, every entire map $g$ can be approximated by a sequence of polynomials $(g_n)$ that are easier to handle compared with $g$. \textit{A priori}, we cannot expect such approximations to be helpful in understanding the dynamics of $g$;  even if $g$ and $g_n$ are close from the numerical point of view, their dynamical properties can still be drastically different. For instance, it is unclear to what extent Taylor approximations preserve dynamics. This raises the question: can we produce ``dynamically meaningful'' approximations of entire maps by sequences of complex polynomials? This problem is especially natural  for families of maps that share several crucial dynamical properties with polynomials, such as entire maps of \textit{finite type} (i.e., maps with finitely many \textit{singular values}).

The problem of dynamical approximation was first studied for the exponential family $\{z\mapsto \lambda \exp(z), \lambda \in \C^*\}$ in \cite{Hubbard_et_al}. The authors considered approximations provided by unicritical polynomials of the form $\lambda(1 + z/n)^n$, and showed that hyperbolic components in the parameter spaces of these polynomials of degree $n$ converge to those in the exponential family as the degree $n$ tends to infinity. They also established that certain \textit{hairs} in the parameter plane for the exponential family are limits of corresponding \textit{external rays} for the polynomial ones. In other words, these results show that the classical approximation of $\lambda\exp(z)$ by the sequence $(\lambda(1+z/n)^n)$ is meaningful from the dynamical point of view. Further developments on this topic include \cite{dyn_approx}, where the author introduced a ``dynamical'' metric on the space of all non-constant non-linear entire maps and showed that convergence in this metric implies kernel convergence of hyperbolic components in the corresponding parameter spaces.

In this paper, we consider \textit{postsingularly finite} (or, shortly, \textit{psf}) entire maps, i.e., maps of finite type for which all singular values are (pre-)periodic. Psf polynomials are also called \textit{postcritically finite} (\textit{pcf} in short) since all their singular values are critical values. While postsingularly finite entire maps might seem a rather special subset of the space of all entire maps, they play an important role in modern holomorphic dynamics owing to their utility in studying the structure of the class of finite type functions. For instance, the complicated structure of the Mandelbrot set can be described in terms of pcf polynomials \cite{DH_Orsay} (see also \cite{fibers}). Moreover, in the rational setting, McMullen conjectured the existence of a strong dynamical similarity between a dense subset of the space of all rational maps with connected \textit{Julia sets} and postcritically finite rational maps (see \cite[Conjecture 1.1, Section 4]{McM_renorm}).

One of the crucial properties of postcritically finite polynomials is that they admit a complete description in purely combinatorial terms, for instance, by \textit{Hubbard trees} or \textit{external angles} (see \cite{Poirier_Thesis}, \cite{Poirier}). Extending such combinatorial descriptions to the family of psf entire maps is an active area of research (see \cite{class_of_exp}, \cite{hubbard_trees_for_exp} for the exponential family, and \cite{Pfrang_thesis}, \cite{dreadlock} for the general case). However, many  questions still remain open. Dynamically meaningful polynomial approximations for postsingularly finite entire maps would provide intuition and potentially chalk out a pathway for studying the combinatorics of such maps.

One of the reasons for the successful development of the theory of postcritically finite polynomials is Thurston theory. 
This theory is based on the idea that holomorphic maps are too rigid, and one should abstract from  the underlying complex structure and consider the more general setup of \textit{branched self-coverings} of the 2-sphere $\S^2$ or \textit{topologically holomorphic maps} defined on the plane $\R^2$ (see Definition \ref{def: topologically holomorphic} for further details). When such a map is postsingularly finite, it is called a \textit{Thurston map} (for a precise version see Definition \ref{def: thurston map}). An important application of Thurston theory is to generate holomorphic psf (or pcf) maps on $\C$ defined, for instance, by combinatorial objects such as \textit{spiders}; see \cite[Chapter~10]{Hubbard_Book_2}.

The core of this theory is a method of determining  whether a given Thurston map is \textit{realized} by a postsingularly finite holomorphic map (see Definition \ref{def: thurston equivalence}). This involves the dynamics of an operator called the \textit{Thurston pullback map}, defined on a suitably chosen Teichm\"{u}ller space (see Sections~\ref{subsec: quasiconformal maps and teichmuller spaces} and ~\ref{subsec: thurston pullback map}). When a Thurston map is not realised holomorphically, it is said to be \textit{obstructed}. 
Crucially, it is known when a \textit{polynomial} Thurston map is realized by a postcritically finite complex polynomial. The Levy-Berstein theorem \cite[Theorem 10.3.8]{Hubbard_Book_2}, and Thurston's characterization theorem \cite{DH_Th_char} (which applies in the more general setting of pcf branched self-covers of $\S^2$) provide a complete answer to this question. 

Thurston theory lays out the relationship between the topological properties of a map, its dynamics, and its geometry in terms of the existence of a holomorphic realization. The last few decades have seen the development of this theory in the transcendental setting, but several  unsolved questions remain. 
The biggest of  these is the nature of obstructions for a \textit{transcendental Thurston map} (in other words, necessary and sufficient conditions for such a map to be obstructed). Obstructions are understood in special classes such as  \textit{exponential Thurston maps} (see  \cite{HSS}), and have recently been studied for \textit{structurally finite Thurston maps} in \cite{Shemyakov_Thesis}, but the general case remains wide open. 

This discussion raises, in this topological setting, the same approximation question  mentioned earlier for entire maps: 
can we view entire Thurston maps as dynamically meaningful limits of polynomial Thurston maps? As a natural extension, we can ask approximation-type questions for other constructions in Thurston theory such as Thurston pullback maps. This article contributes to transcendental Thurston theory by providing a positive answer to the above questions. It broadens our understanding of the relationships between transcendental and polynomial Thurston maps. Furthermore, it translates these results in the Thurston theory regime back to the holomorphic one, using them to construct  meaningful approximations of postsingularly finite entire maps on the dynamical plane. Concurrent work by the first author \cite{mukundan2023dynamical} uses spiders to prove similar results for the special case of postsingularly finite exponential maps $z \mapsto \lambda \exp(z)$  and gives further information on the approximating polynomials in terms of their location in the parameter space of unicritical polynomials of a given degree.

\subsection{Main results.} One of the main theorems of this paper shows that all postsingularly finite entire maps can be viewed as dynamically meaningful limits of postcritical finite complex polynomials. 

\begin{maintheorem}\label{mainthm: main theorem A}
    Let $g$ be an arbitrary postsingularly finite entire map. Then there exists a sequence of postcritically finite polynomials $(g_n)$ converging locally uniformly to $g$ such that $g$ and $g_n$ have the same postsingular portrait for all $n \in \N$.
\end{maintheorem}

See Definition~\ref{def: dynamical portrait} for the precise notion of a \textit{postsingular portrait}. In particular, Main Theorem \ref{mainthm: main theorem A} implies that for each $n\in \N$, there is a map $\varphi_n\colon P_g \to P_{g_n}$ (where $P_g$ and $P_{g_n}$ are the corresponding \textit{postsingular sets}), that conjugates $g|P_g$ and $g_n|P_{g_n}$ and converges to $\id_{P_g}$ as $n$ tends to infinity. Note that the behavior of singular values under iteration is one of the most important dynamical features of an entire map. This behavior, in fact, controls the global dynamics of an entire map of finite type \cite{entire_maps}. Roughly speaking, in the setting of Main Theorem \ref{mainthm: main theorem A}, the ``dynamical cores'' of the polynomials $g_n$ look more and more like the ``dynamical core'' of the limiting map $g$ as $n$ tends to infinity. This supplies a contrast with Taylor approximations $(h_n)$ for $g$, where it is possible that the number of critical values of~$h_n$ tends to infinity as $n \to \infty$. This results in the sets $P_{h_n}$ having a structure very different from $P_g$.

The combinatorics of the pcf polynomials $g_n$ from Main Theorem \ref{mainthm: main theorem A} also approach the combinatorics of $g$ in the following sense:  let $\gamma \subset \C \setminus P_g$ be a simple closed curve and $\widetilde{\gamma}$ be a connected component of $g^{-1}(\gamma)$. 
If $\widetilde{\gamma}$ is also a simple closed curve, then, as it is shown in Corollary \ref{corr: convergence of different lifts}, for any sufficiently small $\varepsilon > 0$ and sufficiently large $n$, there exists a unique simple closed curve $\widetilde{\gamma}_n \subset g_n^{-1}(\gamma)$ that lies in the $\varepsilon$-neighbourhood of $\widetilde{\gamma}$. In particular, for sufficiently large $n$, the curves $\widetilde{\gamma}_n$ and $\widetilde{\gamma}$ are free-homotopic relative $P_{g_n} \cup P_g$ and, moreover, $\deg(g|\widetilde{\gamma}) = \deg(g_n|\widetilde{\gamma}_n)$. If $\widetilde{\gamma}$ is not a simple closed curve (i.e., $\deg(g|\widetilde{\gamma})$ is infinite), there exists a unique (at least for $n$ large enough) simple closed curve $\widetilde{\gamma}_n \subset g_n^{-1}(\gamma)$ so that $d(z, \widetilde{\gamma}_n) \to 0$ for every $z \in \widetilde{\gamma}$ and, moreover, $\deg(g_n|\widetilde{\gamma}_n) \to \infty$ as $n$ tends to~$\infty$. For a more detailed version see Corollary \ref{corr: convergence of different lifts}.
The above remarks provide a summary of the various points of view on the approximating maps we  construct, along with the implications for their dynamics.

The main engine of the proof of Main Theorem \ref{mainthm: main theorem A} is Thurston theory. As it is common in the study of postsingularly finite maps, we start by establishing a desired result in the topological context, i.e., for Thurston maps. 
But talking about pointwise or locally uniform convergence for Thurston maps is not meaningful since they are usually considered up to isotopy (see Definition \ref{def: isotopy of thurston maps}) relative to their postsingular sets. To sidestep this issue, we introduce the notion of \textit{combinatorial convergence}.

Let $f_n\colon \R^2 \to \R^2, n \in \N$ and $f\colon \R^2 \to \R^2$ be Thurston maps with the same postsingular set~$A$ for all $n \in \N$. Further assume that there exist points $b, t \in \R^2 \setminus A$ such that $f(b) = f_n(b) = t$ for all $n \in \N$. Let $\gamma \subset \R^2 \setminus A$ be a loop based at $t$, and let $\gammal(f, {b})$ and $\gammal({f_n},{b})$ respectively denote the $f$- and $f_n$-lifts of $\gamma$ beginning at $b$. We say that the sequence $(f_n)$ \textit{converges combinatorially} to $f$ if for every such loop $\gamma$, the following conditions are satisfied for all $n$ sufficiently large:

\begin{itemize}
    \item the closing behavior of  $\gammal({f},{b})$ 
(i.e., whether it is a loop or a non-closed path) is the same as that of $\gammal({f_n},{b})$, and 
    \item if $\gammal({f},{b})$ is a loop, then it is path-homotopic in $\R^2 \setminus A$ to $\gammal({f_n},{b})$.
\end{itemize}

See Definition \ref{def: combinatorial convergence} for a more general and precise formulation. In particular, if $(f_n)$ converges combinatorially to $f$, then the functions $f|A$ and $f_n|A$ coincide for all $n$ large enough.

We investigate more profound properties of this type of convergence and establish a topological version of Main Theorem \ref{mainthm: main theorem A}. 

\begin{proposition}\label{prop: prop from intro}
    Suppose that $f\colon \R^2 \to \R^2$ is an entire Thurston map, then there exists a sequence $f_n\colon \R^2 \to \R^2, n \in \N$ of polynomial Thurston maps converging combinatorially to $f$.   
\end{proposition}

Let $f\colon \R^2 \to \R^2$ be a Thurston map. Associated with $f$ is a holomorphic function $\sigma_f:~\Teich(\R^2, P_f) \to \Teich(\R^2, P_f)$ called the Thurston pullback map (or the $\sigma$-map), where  $\Teich(\R^2, P_f)$ is a complex manifold called the \textit{Teichm\"{u}ller space of $\R^2$ with the marked set $P_f$} or the Teichm\"{u}ller space \textit{modelled on} $\R^2 \setminus P_f$ (see Sections \ref{subsec: quasiconformal maps and teichmuller spaces}, \ref{subsec: thurston pullback map} for further details). Pullback maps provide a powerful tool to study the associated Thurston maps. In particular, $\sigma_f$ has a (unique)  fixed point $\tau \in \Teich(\R^2, P_f)$ if and only if $f$ is realized by a postsingularly finite entire map (see Proposition \ref{prop: fixed point of sigma}). Our strategy in proving Main Theorem~\ref{mainthm: main theorem A} is to first solve the approximation problem for Thurston pullback maps. In a sense, we translate the original question to the Teichm\"{u}ller space setting. The following theorem forms the second main result of this paper and illustrates the effect of combinatorial convergence on the behavior of the corresponding $\sigma$-maps.

\begin{maintheorem}\label{mainthm: main theorem B}
    Let $f_n\colon \R^2 \to \R^2, n \in \N$ and $f\colon \R^2 \to \R^2$ be Thurston maps with the same postsingular set $A$ for all $n \in \N$. If the sequence $(f_n)$ converges combinatorially to $f$, then the sequence $(\sigma_{f_n})$ converges locally uniformly on $\Teich(\R^2, A)$ to $\sigma_f$.
\end{maintheorem}
It is easy to see that Main Theorem~\ref{mainthm: main theorem A} follows from the points below:
\begin{enumerate}
    \item \label{it: first step of the proof} Main Theorem~\ref{mainthm: main theorem B}, taken together with Proposition~\ref{prop: prop from intro}, tells us that for any Thurston map $f$ with $P_f = A$, there exists a sequence of polynomial Thurston maps $(f_n)$ with $P_{f_n} = A$ that converges combinatorially to $f$ so that the sequence $(\sigma_{f_n})$ converges locally uniformly on $\Teich(\R^2, A)$ to $\sigma_f$.
    \item \label{it: second step of the proof}  If $f$ is realised by a postsingularly finite entire map $g$, we use the weak contraction properties of $\sigma$-maps (see Proposition~\ref{prop: sigma is contracting}) and further analysis to show that for sufficiently large $n$, the map $f_n$ is realised by a postcritically finite complex polynomial $g_n$ such that
    \begin{itemize}
        \item $(g_n)$ converges  to $g$ locally uniformly on $\C$, and
        \item $g_n$ has the same dynamical portrait as $g$, in particular, $g_n |P_{g_n}$ is conjugate to~$g|P_g$.
    \end{itemize}
   This statement is proved in Theorem~\ref{thm: dynamical approximations}.
\end{enumerate}

Proposition~\ref{prop: prop from intro} and Main Theorem \ref{mainthm: main theorem B} allow us to think of entire Thurston maps and their $\sigma$-maps as limits of polynomial Thurston maps and the corresponding $\sigma$-maps. These results form the first direct link between polynomial and entire Thurston theory, and present a new viewpoint for studying other problems in the latter setting. In particular, Main Theorem~\ref{mainthm: main theorem B} as a standalone result relates to the realizability problem for entire Thurston maps. For $f$ and $f_n$'s as in point (\ref{it: first step of the proof}) above, we have shown that if $f$ is realised, then so are the $f_n$'s for $n$ large enough (see Corollary \ref{corr: fixed points of approximating maps}). Conversely, if the $f_n$'s are realised for sufficiently large $n$, then $\sigma_{f_n}$ has a unique fixed point $\tau_n \in \Teich(\R^2, A)$ for all $n$ large enough; in this case $f$ is obstructed if and only if the sequence $(\tau_n)$ leaves every compact set of $\Teich(\R^2, A)$ as $n\rightarrow \infty$. By understanding the ways this divergence might occur, we may hope to unravel information about the obstructions for $f$.  We may also attempt to see if the results in this paper hold for transcendental meromorphic Thurston maps, if we wish to approximate them by rational ones.

\subsection{Organization of the paper.} Our paper is organized as follows.  Section \ref{subsec: notation and basic concepts} is an introduction to basic definitions and our notation. In Section~\ref{subsec: planar embedded graphs}, we discuss planar embedded graphs, and in Section~\ref{subsec: topologically holomorphic maps}, we introduce topologically holomorphic maps and discuss their basic properties. Section \ref{subsec: entire thurston maps} provides the necessary background in Thurston theory, and Section \ref{subsec: quasiconformal maps and teichmuller spaces} discusses quasiconformal maps and Teichm\"{u}ller spaces. We define Thurston pullback maps in Section \ref{subsec: thurston pullback map}, and provide a framework for defining topologically holomorphic maps using planar embedded graphs in Section \ref{subsec: admissible quadruples}.

Section \ref{sec: approximations of entire thurston maps} is an investigation of different notions of convergence for sequences of entire Thurston maps. In Section \ref{subsec: combinatorial convergence}, we introduce the idea of combinatorial convergence and study its  properties. We define \textit{topological convergence} in Section~\ref{subsec: topological convergence} and show that the above two notions of convergence are  equivalent. In Section \ref{subsec: polynomial approximations}, we prove that every entire Thurston map can be approximated by a sequence of polynomial Thurston maps and provide tools for constructing such approximations. Finally, in Section \ref{subsec: convergence of sigma-maps}, we establish Main Theorem~\ref{mainthm: main theorem B}, i.e., show that combinatorial convergence of entire Thurston maps implies convergence of the corresponding $\sigma$-maps in Teichm\"{u}ller space.

In Section \ref{sec: approximations of postsingularly finite entire maps}, we upgrade the topological results of Section \ref{sec: approximations of entire thurston maps} to the holomorphic setting. In Section \ref{subsec: topological properties of locally uniform convergence}, we provide a way of analyzing the convergence of a sequence of entire maps of finite type based on their combinatorial and topological properties. Then we prove Main Theorem \ref{mainthm: main theorem A} and discuss its implications in Section \ref{subsec: dynamical approximations}.

\textbf{Acknowledgments.} 
This material is based upon work supported by the National Science Foundation under Grant No.\ DMS-1928930, while the first and the second authors were in residence at the Mathematical Sciences Research Institute (now Simons Laufer Mathematical Sciences Institute) in Berkeley, California, during the Spring 2022 semester. The second and the third authors were  partially supported by the ERC advanced grant ``HOLOGRAM''. The third author was also partially supported by the ERC advanced grant ``EMERGENCE''.

We would like to thank Dierk Schleicher for proposing this topic and his continued guidance, John Hubbard for introducing the problem of dynamical approximations, and Sarah Koch for providing insightful comments and supporting us through this process.  We are also grateful to Walter Bergweiler, Kostiantyn Drach, Alexandre Eremenko, N\'{u}ria Fagella, Lukas Geyer, Kevin Pilgrim, Lasse Rempe, and Dylan Thurston for inspiring discussions. 

\section{Background}\label{sec: background}

\subsection{Notation and basic concepts}\label{subsec: notation and basic concepts}
The sets of positive integers, integers, real and complex numbers are denoted by $\N$, $\Z$, $\R$, and $\C$, respectively. We use the notation $\I:=[0,1]$ for the closed unit interval on the real line, $\D :=\{z \in \C\colon |z|<1\}$ for the
open unit disk in~$\C$, $\D^* := \D \setminus \{0\}$ for the punctured unit disk, $\H := \{z \in \C: \Re(z) < 0\}$ for the left half-plane, and $\widehat{\C}:=\C\cup\{\infty\}$ for the Riemann sphere. The open disk of radius~$r>0$ centered at~$z_0 \in \C$ is denoted by $\D(z_0, r) := \{z \in \C:|z - z_0| < r\}$, and, for simplicity, $\D_r$ is the disk of radius~$r$ centered at 0. We denote the 2-dimensional plane and the 2-dimensional sphere by $\R^2$ and~$\S^2$, respectively. In this article, we treat them as purely topological objects. In particular, our convention is to write $f\colon \C \to \C$ or $f \colon \widehat{\C} \to \widehat{\C}$ to indicate that $f$ is holomorphic, and $f\colon \R^2 \to \R^2$ or $f \colon \S^2 \to \S^2$ if $f$ is continuous, but not necessarily holomorphic or smooth (cf. Definition \ref{def: topologically holomorphic}).

The cardinality of a set $X$ is denoted by $|X|$ and the identity map on $X$ by $\id_X$. If $f\colon X\to Y$ is a map and $U\subset X$, then $f|U$ stands for the restriction of $f$ to $U$. If $X$ is a topological space and $U\subset X$, then $\overline U$ denotes the closure, $\inter(U)$ the
interior, and $\partial U$ the boundary of $U$ in $X$. By $(f_n)$, we denote a sequence of maps $f_n$ (similarly, numbers, points, sets, groups, etc.) indexed by the set $\N$.

A subset $U \subset \R^2$ is called an \textit{open} or \textit{closed Jordan region} (in $\R^2$) if there exists an injective continuous map $\eta \colon \overline{\D} \to \R^2$ such that $U = \eta(\D)$ or $U = \eta(\overline{\D})$, respectively.

We assume that every topological surface $X$ is oriented. We say that $\mathcal{A}$ is a \textit{complex structure} on $X$ if $\mathcal{A}$ is a \textit{maximal} (under inclusion) \textit{complex atlas} on $X$, i.e., a collection of complex charts compatible with the orientation on $X$ and covering $X$, where any two of them are holomorphically compatible with each other.
We write $(X, \mathcal{A})$ to emphasize that $\mathcal{A}$ is a complex structure on $X$, and we say that $f\colon (X, \mathcal{A}) \to (Y, \mathcal{B})$ is holomorphic to emphasize that $f$ is holomorphic with respect to the complex structures $\mathcal{A}$ and $\mathcal{B}$ on the topological surfaces $X$ and $Y$, respectively. 


If $X$ is a topological space, then a \textit{path} $\alpha$ in $X$ is a continuous map $\alpha \colon \I \to X$. Points $x = \alpha(0)$ and $y = \alpha(1)$ are called \textit{endpoints} of the path $\alpha$ and we say that $\alpha$ \textit{joins} $x$ with~$y$. The interior of the path $\alpha$ is the set $\alpha((0, 1)) =: \inter(\alpha)$. The path $\alpha$ is called a \textit{loop} if $\alpha(0) = \alpha(1)$, otherwise we say that $\alpha$ is a \textit{non-closed path}. We say that the path $\alpha$ \textit{starts} at $x \in X$ if $\alpha(0) = x$. When $\alpha$ is a loop, we also say that $\alpha$ is \textit{based at} $x$. 
A non-closed path $\alpha$ is called \textit{simple} if it has no self-intersections (i.e., $\alpha\colon \I \to X$ is injective). A loop $\alpha$ is called a \textit{simple} if it has no self-intersections except at endpoints (i.e., $\alpha|(0, 1)$ is injective). Loop $\alpha$ is called \textit{constant} if the map $\alpha \colon \I \to X$ is constant. Sometimes we are going to conflate paths and their images. For instance, we write $\alpha \subset Y$, where $\alpha \colon \I \to X$ is a path and $Y \subset X$, to indicate that $\alpha(\I) \subset Y$.
We say that $\gamma \subset X$ is a \textit{simple closed curve} if $\gamma = \alpha(\I)$ for some simple loop $\alpha$ in $X$. 

If $\alpha$ and $\beta$ are two paths in $X$, then we denote by $\alpha \cdot \beta$ their \textit{concatenation}, i.e., the path that first traverses $\alpha$ and then $\beta$. By $\overline{\alpha}$ we denote the path in $X$ such that $\overline{\alpha}(t) = \alpha(1 - t)$ for all $t \in \I$. For $n \in \Z$, we define $\alpha^n$ to be a constant loop based at $\alpha(0)$ if $n = 0$, a concatenation of $\alpha$ with itself $n$ times, if $n > 0$, and concatenation of $\overline{\alpha}$ with itself $-n$ times, if $n < 0$.


Two paths $\alpha\colon \I \to X$ and $\beta \colon \I \to X$ are called \textit{path-homotopic} (or simply \textit{homotopic}) if there exists a continuous map $H\colon \I \times \I \to X$ called a \textit{homotopy} so that $H(t, 0) = \alpha(t)$, $H(t, 1) = \beta(t)$ for all $t \in \I$, and $H(0, s) = \alpha(0) = \beta(0)$, $H(1, s) = \alpha(1) = \beta(1)$ for all $s \in \I$. 
Let $A$ be a subset of $X$. We say that two paths $\alpha$ and $\beta$ are \textit{path-homotopic relative $A$} (abbreviated as ``$\alpha$ and $\beta$ are homotopic rel.\ $A$'' and denoted $\alpha \sim \beta$ rel.\ $A$) if $\alpha \subset X \setminus A$ and $\beta \subset X \setminus A$ are homotopic in $X \setminus A$, i.e.,  $H(s, t) \in X \setminus A$ for all $(s, t) \in \I \times \I$, where $H$ is the corresponding homotopy. 


If $X$ is a topological space and $x \in X$, then by $\pi_1(X, x)$ we denote the \textit{fundamental group} of $X$ at $x$, i.e., the set of all homotopy equivalence classes of loops in $X$ based at~$x$ endowed with the operation of path concatenation. If $f\colon X \to Y$ is a continuous map between topological spaces $X$ and $Y$ such that $f(x) = y$, then the group homomorphism $f_*\colon \pi_1(X, x) \to \pi_1(Y, y)$ is defined as $f_*([\alpha]) = [f \circ \alpha]$ for any loop $\alpha \subset X$ based at $x$, where~$[\cdot]$ denotes a homotopy equivalence class. Finally, we say that path $\alpha \subset Y$ \textit{lifts under~$f$} to a path $\beta \subset X$, if $\alpha = f \circ \beta$, and $\beta$ is called the \textit{$f$-lift} (or \textit{lift under $f$}) of $\alpha$. If $f\colon X \to Y$ is a covering map and $f(x) = y$, then every path $\alpha$ starting at $y$ has a unique $f$-lift starting at $x$, which is  denoted by $\alphal(f, x)$. By definition, the path $\alphal(f,x)$ is a loop if and only if $[\alpha] \in f_*\pi_1(X, x)$.

Suppose that $X$ and $Y$ are topological spaces. We say that homeomorphisms $\varphi \colon X \to Y$ and $\psi \colon X \to Y$ are \textit{isotopic} if there exists a continuous map $H\colon X \times \I \to Y$ called an \textit{isotopy} such that $H(x, 0) = \varphi(x)$ and $H(x, 1) = \psi(x)$ for all $x \in X$, and $H(\cdot, t)\colon X \to Y$ is a homeomorphism for every $t \in \I$. We say that $\varphi$ and $\psi$ are \textit{isotopic rel.\ $A$}, where $A \subset X$, if additionally $H(x, t) = x$ for all $(x, t) \in A \times \I$.

Let $X$ be a topological surface. We denote by $\Homeo^+(X)$ and $\Homeo^+(X, A)$ the group of all orientation-preserving self-homeomorphisms of $X$ and the group of all orientation-preserving self-homeomorphisms of $X$ fixing $A$ pointwise, respectively. We use the notation $\Homeo^+_0(X, A)$ for the subgroup of $\Homeo^+(X, A)$ consisting of all homeomorphisms isotopic rel.\ $A$ to $\id_X$.

We use the following notion of convergence, in the sense of \cite{Chabauty}, of sequences of subgroups of a given group endowed with the discrete topology.

\begin{definition}\label{def: convergence of groups}
    Let $G$ be a group and $(H_n)$ be a sequence of its subgroups. We say that the sequence $(H_n)$ \textit{converges} to a subgroup $H$ of $G$ and write $\lim\limits_{n \to \infty} H_n = H$ if for every $g \in G$, there exists $N = N(g)$ so that for each $n \geq N$, $g \in H_n$ if and only if $g \in H$.

\end{definition}

\subsection{Planar embedded graphs} \label{subsec: planar embedded graphs} We refer the reader to \cite{DiestelGraph} for classical background in graph theory. A \emph{planar embedded graph} is a pair ${G=(V, E)}$, where

\begin{enumerate}
    \item $V$ is a discrete (in particular, countable) set of points in $\R^2$, and

    \item $E$ is a set of simple paths and simple loops (viewed as subsets of $\R^2$, i.e., as images of the corresponding maps) such that their endpoints belong to $V$, their interiors are pairwise disjoint and lie in $\R^2 \setminus V$, and every compact set $K \subset \R^2$ intersects finitely many elements of $E$.
\end{enumerate}

The sets $V$ and $E$ are called the \emph{vertex set} and \emph{edge set} of $G$, respectively. Note that our notion of a planar embedded graph allows \emph{multi-edges}, i.e., distinct edges that connect the same pair of vertices, and \emph{loop-edges}, i.e., edges that connect a vertex to itself. 

In the following, suppose $G = (V, E)$ is a planar embedded graph. We say that $G$ is \textit{finite} if $V$ and $E$ are finite sets. The \emph{degree} of a vertex $v$ in $G$, denoted by $\deg_G(v)$, is the number $n_1 + 2n_2$, where $n_1$ and $n_2$ are the numbers of simple paths and simple loops in $E$ incident to $v$, respectively. We say that $G$ is $k$-\textit{regular} if $\deg_G(v) = k$ for every $v \in V$. A \emph{subgraph} of $G$ is a planar embedded graph $G' = (V', E')$ such that $V' \subset V$ and $E' \subset E$. 

The subset $\mathcal{G}:=V \cup \bigcup_{e \in E} e$ of $\R^2$ is called the \emph{realization} of $G$. A \emph{face} of the graph $G$ is a connected component of $\R^2 \setminus \mathcal{G}$. The set of all faces of $G$ is denoted by $F(G)$. Note that due to the definition of a planar embedded graph, the set $\mathcal{G}$ is closed in $\R^2$, and thus, every face $F$ of $G$ is open. If $F$ is a face of $G$, then we denote by $\partial F$ the subgraph of $G$ forming the (topological) boundary of $F$ in $\R^2$. 


The graph $G$ is called connected if its realization $\mathcal{G}$ is connected (or, equivalently, path-connected). It follows that $G$ is connected if and only if each face of $G$ is simply connected. 
We will often conflate a planar embedded graph $G$ with its realization~$\mathcal{G}$.




Let $G_1 = (V_1, E_1)$ and 
$G_2 = (V_2, E_2)$ be two planar embedded graphs. We say that $G_1$ is \emph{isomorphic} to $G_2$ if there exists a homeomorphism $\varphi \in \Homeo^+(\R^2)$ that maps vertices and edges of $G_1$ into vertices and edges of $G_2$, that is, $\varphi(V_1) = V_2$ and $\varphi(E_1) = E_2$. In this case, we call $\varphi$ an \emph{isomorphism} between $G_1$ and $G_2$. If $\varphi$ is isotopic rel.\ $A$ to $\id_{\R^2}$ for some set $A\subset \R^2$, we say that $G_1$ is \textit{isotopic} rel.\ $A$ to $G_2$.

Suppose that $f\colon U \to W$ is a covering map, where $U$ and $W$ are open subsets of $\R^2$. Additionally, assume that the realization $\mathcal{G}$ of a graph $G = (V, E)$ lies in $U$. Then we can define a new graph $f^{-1}(G)$, called the \textit{preimage} of $G$ under $f$, where $V(f^{-1}(G)) := f^{-1}(V(G))$ and the set $E(f^{-1}(G))$ consists of all $f$-lifts of the edges of $G$.

Examples of planar embedded graphs are \textit{cycles}, \textit{chains} and \textit{rose graphs}. We say that $G=(V, E)$ is a cycle if $G$ is finite, connected, and 2-regular. The graph $G$ is a finite chain if it consists of a single vertex with no edges, or it is connected and has exactly two vertices of degree one and all other vertices have degree two. In the case when $G$ is a finite chain, all vertices of $G$ of degree less than 2 are called \textit{endpoints} of $G$. If $G$ is infinite, connected, and 2-regular, then it is called an infinite chain. Finally, if $|V| = 1$ and for every edge $e$ of $G$, the bounded connected component of $\R^2 \setminus e$ does not intersect the realization $\mathcal{G}$ of $G$, then we say that $G$ is a rose graph.  If $\Rose$ is a rose graph, then we call its unique vertex the \textit{center} of $\Rose$ and its edges the \textit{petals} of the rose graph $\Rose$. Given a finite set $A \subset \R^2$ and a rose graph $\Rose$, we say that $\Rose$ \textit{surrounds} $A$ if every bounded face of $\Rose$ contains a unique point of $A$, and every point of $A$ is contained in some bounded face of $\Rose$.

We say that graph $G' = (V', E')$ is the result of \textit{subdivision of an edge} $e \in E(G)$ of the planar embedded graph $G = (V, E)$ if $G'$ is obtained from $G$ by adding a new vertex in the interior of $e$. More precisely, $V' = V \cup \{v\}$ and $E = (E \setminus \{e\}) \cup \{e_1, e_2\}$, where $v \in \inter(e)$, and $e_1, e_2$ are the closures of the connected components of $e \setminus \{v\}$.  In particular, subdivision of an edge does not change the realization of the graph, and the result is uniquely defined up to an isotopy rel.\ $V$.

Let $G = (V, E)$ be a planar embedded graph and $e \in E$ be one of its edges. We say that a continuous map $\alpha \colon \I \to e$ is a \textit{parametrization} of $e$ if $\alpha|(0,1)$ is bijective onto~$\inter(e)$. We say that two parametrizations $\alpha_1$ and $\alpha_2$ of the same edge $e$ are equivalent if $\alpha_1^{-1} \circ \alpha_2 |(0, 1) \colon (0, 1) \to (0, 1)$ is an increasing function. One can easily see that for every edge $e \in E(G)$ there are two distinct equivalence classes of parametrizations. We call each of these equivalence classes a \textit{direction} of the edge $e$. We say that the graph $G$ is \textit{directed} if each of its edges is endowed with a unique direction (called the  \textit{forward} direction). A choice of forward directions for all edges of a graph $G$ is also called an \textit{orientation} of~$G$. Directions that are omitted from the orientation of $G$ are called \textit{backward}. In a similar way, we introduce notions of forward and backward parametrizations of the edges of the directed graph $G$. 
If $v$ is a vertex of $G$ and $e$ is an edge incident to $v$, then there is a natural way to call $e$ \textit{incoming} or \textit{outgoing} at $v$ unless $e$ is a loop-edge.

Suppose that $\Phi\colon G_1 \to G_2$ is a continuous map, where $G_1$ is a planar embedded graph and $G_2$ is a directed planar embedded graph, such that $\Phi^{-1}(V(G_2)) = V(G_1)$ and $\Phi|\inter(e)$ is injective for each edge $e \in E(G_1)$. The map $\Phi$ and the orientation of the graph $G_2$ naturally induce an orientation of the graph $G_1$. Indeed, we choose a forward direction for $e \in E(G_1)$ so that if $\alpha$ is a parametrization of $e$, then $\alpha$ is forward if and only if $\Phi \circ \alpha$ is a forward parametrization of the edge $\Phi(e) \in E(G_2)$.
Similarly, there is a natural way to define forward directions for a subgraph and graph obtained by a subdivision of an edge of a directed planar embedded graph. 

We say that a directed planar embedded graph $G = (V, E)$ is \textit{unilaterally connected} if for every pair of vertices $u, v \in V$, there exists a path $\alpha$ with endpoints at $u$ and $v$ that is obtained by concatenation of forward parametrizations of edges of $G$. 


We say that a rose graph $\Rose$ is \textit{counterclockwise directed} if for each edge $e \in E(\Rose)$ any forward parametrization $\alpha$ orients $e$ counterclockwise if viewed from the interior of the bounded connected component of $\R^2 \setminus e$. Graphs at the top in Figures \ref{fig: nondyn_quadruple for cosine}, \ref{fig: nondyn_quadruple or eg2}, \ref{fig: dyn_pair for cosine}, and \ref{fig: dyn_pair for multi_error} provide examples of counterclockwise directed rose graphs with two or three petals (in particular, arrows on their petals indicate forward directions of the corresponding edges). In a similar way we can define a \textit{counterclockwise directed cycle}.







\subsection{Topologically holomorphic maps}\label{subsec: topologically holomorphic maps}

\begin{definition}\label{def: topologically holomorphic}
     Let $X$ and $Y$ be topological surfaces. Then a map $f\colon X \to Y$ is called \textit{topologically holomorphic} if for every $x \in X$ there exists $d \in \N$, an open neighbourhood $U$ of~$x$, and two orientation-preserving homeomorphisms $\psi\colon U \to \D$ and $\varphi\colon f(U) \to \D$ such that $\psi(x) = \varphi(f(x)) = 0$ and $(\varphi \circ f \circ \psi^{-1})(z) = z^d$ for all $z \in \D$.

     In particular, the integer $\deg(f,x) := d$ is independent of $U, \psi$ and $\varphi$ and is called the \textit{local degree} of $f$ at $x$.         
\end{definition}


\begin{remark}
    The notion of a topologically holomorphic map is close in spirit to the notion of \textit{branched covering map}. However, these two notions are generally not the same (cf. \cite[Definition A.7]{THEBook}).
\end{remark}
\begin{definition}
    Let $X$ and $Y$ be topological spaces. A map $f: X\rightarrow Y$ is called \textit{discrete} if $f^{-1}(y)$ is a discrete subset of $X$ for every $y \in Y$.
\end{definition}

Theorem \cite[Theorem 1.3, Remark 1.4]{Stoilow} (see also \cite{Stoilow_original_paper}, \cite{Stoilow_original_book}) imply that we can give an alternative definition of a topologically holomorphic map.


\begin{proposition}\label{prop: alternative definition}
    Let $X$ and $Y$ be topological surfaces. Then a map $f\colon X \to Y$ is topologically holomorphic if and only if it is continuous, open, discrete, and for every $x \in X$ such that $f$ is locally injective at $x$, there exists an open neighbourhood $U$ of $x$ for which $f|U\colon U \to f(U)$ is an orientation-preserving homeomorphism.
\end{proposition}

It is known that we can pull back complex structures under topologically holomorphic maps (see \cite[Sto\"{i}low's factorisation theorem]{Stoilow} and \cite[Lemma A.12]{THEBook}) as the following proposition states. 

\begin{proposition} \label{prop: pulling back complex structures}
    Let $f\colon X\to Y$ be topologically holomorphic. Then for every complex structure $\mathcal{A}$ on $Y$ there exists a unique complex structure $f^*\mathcal{A}$ on $X$ such that the map $f\colon (X, f^* \mathcal{A}) \to (Y, \mathcal{A})$ is holomorphic.
    
    In particular, for any orientation-preserving homeomorphism $\varphi\colon Y \to S_Y$, where $S_Y$ is a Riemann surface, there exists a Riemann surface $S_X$ and an orientation-preserving homeomorphism $\psi\colon X \to S_X$ such that $\varphi \circ f \circ \psi^{-1}\colon S_X \to S_Y$ is holomorphic.
\end{proposition}

In particular, the second part of Proposition \ref{prop: pulling back complex structures} provides the third equivalent definition of a topologically holomorphic map.

Several classical notions of one-dimensional complex analysis can be extended to the setting of topologically holomorphic maps.

\begin{definition}\label{def: values}
    Let $X$ and $Y$ be topological surfaces and $f\colon X \to Y$ be a  topologically holomorphic map. 
    \begin{enumerate}
        \item A point $y \in Y$ is  a \textit{regular value of $f$} if there exists a neighbourhood $U$ of $y$ such that $f^{-1}(U) = \bigsqcup_{i \in I} U_i$ and $f|U_i\colon U_i \to U$ is a homeomorphism for each $i \in I$. If such a neighborhood $U$ does not exist, we call $y$ a \textit{singular value of $f$}. The set of all singular values of $f$ is denoted by $S_f$. If $S_f$ is finite, then we say that $f$ is a map \textit{of finite type} or $f$ belongs to the \textit{Speiser class} $\mathcal{S}$. 
    
        \item We say that $x \in X$ is a \textit{critical point} of $f$ if $\deg(f, x) > 1$. A point $y \in Y$ is called a \textit{critical value} of $f$ if there exists a critical point $x \in X$ such that $f(x) = y$.
        
        \item We say that $y \in Y$ is an \textit{asymptotic value} of $f$ if there exists a continuous curve $\gamma\colon[0,\infty) \rightarrow X$ such that $\gamma$ leaves every compact set of $X$ and $f(\gamma(t)) \rightarrow y$ as~$t \rightarrow \infty$. 
    \end{enumerate}
\end{definition}

\begin{remark}
    \begin{enumerate}
        \item Critical and asymptotic values are always singular.

        \item If $f\colon X \to Y$ is topologically holomorphic, then the restriction $$f|X \setminus f^{-1}(S_f) \colon X \setminus f^{-1}(S_f) \to Y \setminus S_f$$ is a covering map. 
    \end{enumerate}
\end{remark}


        

Every locally non-constant holomorphic function between two Riemann surfaces is an obvious example of a topologically holomorphic map. Suppose that $f \colon X \to Y$ is topologically holomorphic, we say that $f$ is \textit{entire} if $X = Y = \R^2$, \textit{meromorphic} if $X = \R^2, Y = \S^2$ and if $X = Y = \S^2$, then $f$ is a branched covering map. Thus, topologically holomorphic maps serve as topological models for one-dimensional holomorphic functions, and the types of maps listed above are the underlying models for the corresponding families of holomorphic maps.


Our focus in this article is on entire topologically holomorphic maps. Interesting phenomena can occur when we pull back complex structures in this setting. Suppose that $\mathcal{A}$ is a complex structure on $\R^2$ such that $(\R^2, \mathcal{A})$ is biholomorphic to $\D$. It is evident that $(\R^2, f^*\mathcal{A})$ is biholomorphic to $\D$, since there is no non-constant map from $\C$ to $\D$. However, if $(\R^2, \mathcal{A})$ is biholomorphic to $\C$, then by the Uniformization Theorem,  $(\R^2, f^*\mathcal{A})$ is biholomorphic to either $\C$ or $\D$. This ambiguity in the biholomorphism class of $(\R^2, f^*\mathcal{A})$ is known as the  \textit{type problem} \cite{geometric_function_theory}.

\begin{definition}\label{def: types of maps}
    We say that a topologically holomorphic map $f\colon \R^2 \to \R^2$ has \textit{stable parabolic type}, if for every complex structure $\mathcal{A}$ on $\R^2$ such that $(\R^2, \mathcal{A})$ is biholomorphic to~$\C$, $(\R^2, f^*\mathcal{A})$ is biholomorphic to $\C$, as well. 
    
    If for every such complex structure $\mathcal{A}$, $(\R^2, f^*\mathcal{A})$ is biholomorphic to $\D$, we say that $f$ has \textit{stable hyperbolic type}.
\end{definition}
\begin{remark}\label{rmk:two_types}
    It is shown in \cite[pp.~3-4]{geometric_function_theory} 
 that every topologically holomorphic map of finite type has either stable parabolic or stable hyperbolic type (see also \cite{Teichmuller}).
\end{remark}
If $f \colon \R^2 \to \R^2$ is a topologically holomorphic map having stable parabolic type, then by Proposition \ref{prop: pulling back complex structures}, for any orientation-preserving homeomorphism $\varphi \colon \R^2 \to \C$, there exist another orientation-preserving homeomorphism $\psi \colon \R^2 \to \C$ and entire map $g$ such that $f = \varphi^{-1} \circ g \circ \psi$. Note that $f \in \mathcal{S}$ if and only if $g \in \mathcal{S}$.


In the rest of the paper, we consider only entire topologically holomorphic maps of finite type. By Remark~\ref{rmk:two_types}, the type problem in this class of maps has only has two possible answers, and we simply call such a map parabolic or hyperbolic.


We say that a topologically holomorphic map $f$ of finite type is a \textit{topological polynomial} of degree $d$ if the covering map $f|\R^2 \setminus f^{-1}(S_f)$ has a finite degree $d =: \deg(f)$. Otherwise, we say that $f$ is \textit{transcendental}. It is easy to see that any topological polynomial is parabolic.


The following propositions lay out the structure of the preimage of a neighborhood of an isolated singular value (or infinity) under an entire topologically holomorphic map.

\begin{proposition} \label{prop: preimages without infinity}

    Let $f\colon \R^2 \to \R^2$ be a topologically holomorphic map, $U \subset \R^2$ be a bounded simply connected domain, and $V$ be a connected component of $f^{-1}(U)$.
    If $U\cap S_f = \{y\}$, then $V$ is simply connected and exactly one of the following is true:

    \begin{enumerate}
        \item there exist orientation-preserving homeomorphisms $\varphi \colon U \to \D$ and $\psi\colon V \to \D$ and an integer $d\in \N$ such that $\psi(y) = 0$ and $(\varphi \circ (f|V) \circ \psi^{-1})(z) =  z^d$ for all $z \in \D$. In particular, $V\setminus f^{-1}(y)$ is an annulus and $f|V\setminus f^{-1}(y)\colon V\setminus f^{-1}(y) \to U\setminus \{y\}$ is a covering map of degree $d$. 
        Additionally, if $\partial U \cap S_f = \emptyset$, then $V$ is  bounded;
        \item there exist orientation-preserving homeomorphisms $\varphi\colon U \to \D$ and $\psi \colon V \to \H$ such that $\varphi(y) = 0$ and $(\varphi \circ (f|V) \circ \psi^{-1})(z) = \exp(z)$ for all $z \in \H$. In particular, $V$ is unbounded, and $f|V\colon V \to U\setminus \{y\}$ is a universal covering map. 
    \end{enumerate}

    
    If $U$ is a bounded simply connected domain such that $U \cap S_f = \emptyset$, then each connected component $V$ of $f^{-1}(U)$ is simply connected and $f|V \colon V \to U$ is an orientation-preserving homeomorphism. If additionally $\partial U \cap S_f = \emptyset$, then $V$ is bounded.

    \end{proposition}

    \begin{proof}
        This proposition follows from \cite[Theorem 5.10, 5.11]{Forster} in the case when $f$ is an entire holomorphic map. The general case follows easily by applying Proposition \ref{prop: pulling back complex structures}.  
    \end{proof}

    \begin{proposition} \label{prop: preimages with infinity}

    Let $f\colon \R^2 \to \R^2$ be a topologically holomorphic map, $U$ be an unbounded simply connected domain, and $V$ be a connected component of $f^{-1}(U)$. If $U \cap S_f = \emptyset$, then $V$ is an unbounded simply connected domain, and $f|V \colon V \to U$ is an orientation-preserving homeomorphism. If $U = \R^2 \setminus W$, where $W$ is a compact simply connected set containing $S_f$, then $V$ is an unbounded domain and one of the following is true:

    \begin{enumerate}
        \item if $f$ is a topological polynomial of degree $d$, then there exist orientation-preserving homeomorphisms $\varphi \colon U \to \D^*$ and $\psi\colon V \to \D^*$ such that $(\varphi \circ (f|V) \circ \psi^{-1})(z) =  z^d$ for all $z \in \D^*$. In particular, $V$ is an annulus, $\R^2 \setminus V$ is compact, and $f|V\colon V \to U$ is a covering map of degree $d$;

        \item if $f$ is transcendental, then there exist orientation-preserving homeomorphisms $\varphi\colon U \to \D^*$ and $\psi \colon V \to \H$ such that $(\varphi \circ (f|V) \circ \psi^{-1})(z) = \exp(z)$ for all $z \in \H$. In particular, $f|V\colon V \to U$ is a universal covering map. 
    \end{enumerate}

    \end{proposition}

    \begin{proof}
        The proof is similar to that of Proposition \ref{prop: preimages without infinity}.
    \end{proof}








Note that Proposition \ref{prop: preimages without infinity} implies that every singular value of a topologically holomorphic map $f\colon \R^2 \to \R^2$ of finite type is either critical or asymptotic.


We say that an unbounded simply connected domain $V \subset \R^2$ is an \textit{asymptotic tract} of~$f$ over $y \in S_f$ if there exists a bounded simply connected domain $U \subset \R^2$ such that $U \cap S_f = \overline{U} \cap S_f = \{y\}$ and $V$ is a connected component of $f^{-1}(U)$. Similarly, $V$ is an asymptotic tract of $f$ over $\infty$ if there exists a simply connected compact set $W$ such that $S_f \subset \inter(W)$ and $V$ is a connected component of $f^{-1}(\R^2 \setminus W)$. Note that if $V$ is a tract of $f$, then by Propositions \ref{prop: preimages without infinity} and \ref{prop: preimages with infinity}, 
 $f|V \colon V \to f(V)$ is a universal covering map.
 
Analogous to \cite[Propostion 2.3]{Rempe_Epstein} and proven in a similar way, we have the following isotopy lifting property for entire topologically holomorphic maps of finite type.

\begin{proposition}\label{prop: isotopy lifting property}
    Let $f \colon \R^2 \to \R^2$ and $\widehat{f}\colon \R^2 \to \R^2$ be topologically holomorphic maps of finite type with $\varphi_0 \circ f = \widehat{f} \circ \psi_0$ for some $\varphi_0, \psi_0 \in \Homeo^+(\R^2)$.

    Let  $A \subset \R^2$ be a finite set containing $S_{f}$,  and $\varphi_1 \in \Homeo^+(\R^2)$ be isotopic rel.\ $A$ to $\varphi_0$. Then $\varphi_1 \circ f = \widehat{f} \circ \psi_1$ for some $\psi_1 \in \Homeo^+(\R^2)$ isotopic rel.\ $f^{-1}(A)$ to $\psi_0$.
\end{proposition}

\subsection{Entire Thurston maps}\label{subsec: entire thurston maps}



Let $f\colon X \to X$ be a topologically holomorphic map, where $X$ is a topological surface. The \textit{postsingular set} of $f$ is defined as
$$
    P_f := \bigcup_{n \geq 0} f^{\circ n}(S_f),
$$
and its elements are called \textit{postsingular values}. We say that $f$ is \textit{postsingularly finite} (or, shortly, \textit{psf}) if $P_f$ is finite, i.e., $f \in \mathcal{S}$ and every singular value is (pre-)periodic under $f$.  Postsingularly finite topological polynomials are also called \textit{postcritically finite} (\textit{pcf} in short), and their postsingular values are called \textit{postcritical} since all their singular values are critical values.

Now we are ready to state one of the key definitions of this section.

\begin{definition}\label{def: thurston map}
    A non-injective topologically holomorphic map $f\colon \R^2 \to \R^2$ is called an \textit{entire Thurston map} if it is parabolic and postsingularly finite. 
    
    Given a finite set $A \subset \R^2$ with $P_f \subset A$ and $f(A) \subset A$, we call the pair $(f, A)$ a \textit{marked entire Thurston map} and $A$ its \textit{marked set}.
\end{definition}

We often treat marked Thurston maps the same as usual Thurston maps and use the notation $f\colon (\R^2, A) \righttoleftarrow$. If no marked set is mentioned, we assume it to be $P_f$. 

Similarly to Definition \ref{def: thurston map}, we can introduce the notion of a \textit{rational Thurston map} $f\colon \S^2 \to \S^2$ and a \textit{meromorphic Thurston map} $f\colon \R^2 \to \S^2$. In the literature, the phrase ``Thurston map'' usually refers to a rational Thurston map, but we use it as shorthand for an entire Thurston map instead. 
Natural examples of entire Thurston maps are provided by (holomorphic) postsingularly finite entire maps, particularly by postcritically finite complex polynomials.


\begin{definition}\label{def: thurston equivalence}
    We say that two Thurston maps $f\colon (\R^2, A) \righttoleftarrow$ and $\widehat{f}\colon (\R^2, B) \righttoleftarrow$ are \textit{Thurston} (or \textit{combinatorially}) equivalent, if there exist homeomorphisms $\varphi, \psi \in \Homeo^+(\R^2)$ such that $\varphi(A) = \psi(A) = B$, $\varphi$ and $\psi$ are isotopic rel.\ $A$, and $\varphi \circ f = \widehat{f} \circ \psi$.
\end{definition}

A Thurston map $f\colon (\R^2, A)\righttoleftarrow$ is said to be  \textit{realized} if it is combinatorially equivalent to a postsingularly finite entire map $g$. If $f$ is not realized, we say that it is \textit{obstructed}.





Frequently, Thurston maps we work with are defined uniquely only up to the following notions of equivalence (for instance, when we define Thurston maps using combinatorial constructions).



\begin{definition}\label{def: isotopy of thurston maps}
    Two Thurston maps $f_1\colon (\R^2, A) \righttoleftarrow$ and $f_2 \colon (\R^2, A) \righttoleftarrow$ are called \textit{isotopic} (rel.\ $A$) if there exists $\varphi \in \Homeo_0^+(\R^2, A)$ such that $f_1 = f_2 \circ \varphi$.
\end{definition}


The following proposition provides a simple classification of topologically holomorphic maps with only one singular value, as well as Thurston maps with a single postsingular value.

\begin{proposition}\label{prop: small postsingular set}
    Suppose that $f\colon \R^2 \to \R^2$ is a topologically holomorphic map such that $|S_f| = 1$. Then $f = \varphi^{-1} \circ g \circ \psi$, for some orientation-preserving homeomorphisms  $\varphi, \psi \colon \R^2 \to~\C$ and a unique $g \in \{z \mapsto z^d| d \geq 2\} \cup \{z \mapsto \exp(z)\}$.

    If $f\colon (\R^2, A) \righttoleftarrow$ is a Thurston map with $|A| = 1$, then $f$ is combinatorially equivalent to $z \mapsto z^d$ for some $d \geq 2$. In particular, $A$ consists of a unique fixed critical point of $f$.
\end{proposition}

\begin{proof}
    The first part essentially follows from Proposition \ref{prop: preimages with infinity}.


    Suppose now that $f\colon (\R^2,A) \righttoleftarrow$ is a Thurston map with $|A| = 1$. Thus, $f$ has a fixed singular value, so it cannot be of the form $\varphi^{-1} \circ \exp \circ \psi$ for some orientation-preserving homeomorphisms $\varphi \colon \R^2 \to \C$ and $\psi \colon \R^2 \to \C$. Therefore,  $(\varphi \circ f \circ \psi^{-1})(z)  = z^d$ for all $z \in \C$ and $d = \deg(f)$, where $A = \{\varphi(0)\} = \{\psi(0)\}$. In particular, $\varphi$ and $\psi$ are isotopic rel.\ $A$ and $f$ is combinatorially equivalent to $z \mapsto z^d$.
\end{proof}


The dynamics of a Thurston map on its marked set can also be represented visually, in a way that turns out to be useful in study. The following definition makes this precise.

\begin{definition} \label{def: dynamical portrait}
    Let $f\colon (\R^2, A) \righttoleftarrow$ be a Thurston map. The \textit{marked portrait (rel.\ $A$) of $f$} is a weighted directed abstract graph $\P_{f, A}$ such that the vertex set of $\P_{f, A}$ equals $A$, and for each vertex $v \in A$ there exists a unique directed edge from $v$ to $f(v)$ with weight $\deg(f, v)$. Additionally, among all vertices of $\P_{f, A}$, we label the ones that are singular values of $f$.

    We say that two Thurston maps $f\colon (\R^2, A)\righttoleftarrow$ and $\widehat{f}\colon (\R^2, B)\righttoleftarrow$ have \textit{the same marked portrait} if there exists a bijective map $\varphi\colon A \to B$ such that 
    
    \begin{itemize}
        \item $\varphi(S_f) = S_{\widehat{f}}$, and
        \item $\varphi$ is an \textit{isomorphism} between the weighted directed abstract graphs $\P_{f, A}$ and $\P_{\widehat{f}, B}$, i.e., there is an edge $e_{u,v}$ of $\P_{f, A}$ joining $u$ with $v$ if and only if there is an edge $\widehat{e}_{\varphi(u),\varphi(v)}$ of  $\P_{\widehat{f}, B}$ joining $\varphi(u)$ with $\varphi(v)$. The weights of $e_{u,v}$ and $\widehat{e}_{\varphi(u),\varphi(v)}$ coincide.
    \end{itemize}

\end{definition}
In the setting of Definition~\ref{def: dynamical portrait}, if $A = P_f$, then, for simplicity, we denote by $\P_f$ the marked portrait rel.\ $P_f$ of $f$  and call it the  \textit{postsingular portrait} of $f$.

If Thurston maps $f \colon (\R^2, A)\righttoleftarrow$ and $\widehat{f} \colon (\R^2, B)\righttoleftarrow$ have the same marked  portrait, then it is clear that $f|A$ and $\widehat{f}|B$ are conjugate dynamical systems. Combinatorial equivalence of $f$ and $\widehat{f}$ is a sufficient condition for them to have the same marked portrait.

\begin{example}\label{ex: portraits}
    Given below are two postsingularly finite entire maps that will be used as prototypical examples throughout this paper.

    \begin{enumerate}
        \item The map $G_1(z) = \pi \cos (z) / 2$ has no asymptotic values and two critical values $\pm \pi / 2$, with $P_{G_1} = \{0, -\pi/2, \pi/2\}$;
        
        \item The map $G_2(z) = \sqrt{\ln 2}(1-\exp {(z^2)})$ has a unique critical value $0$ and a unique asymptotic value $\sqrt{\ln 2}$, with $P_{G_2} = \{0, -\sqrt{\ln 2}, \sqrt{\ln 2}\}$.
    \end{enumerate}

   Figure \ref{fig: portraits} illustrates $\P_{G_1}$ and $\P_{G_2}$. Singular and non-singular vertices of the corresponding graphs are labeled by solid and hollow squares,  respectively.

\end{example}

\begin{figure}[t]
    \begin{subfigure}[b]{0.45\textwidth}
        \centering
        \includegraphics[width=0.6\textwidth]{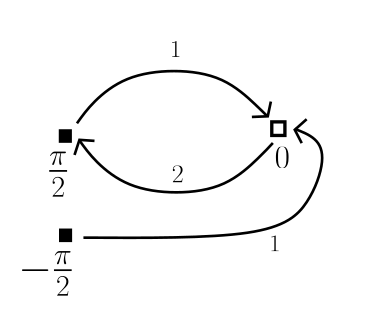}
        \caption{$G_1(z) = \pi \cos(z) / 2$.}
        \label{fig:portrait_cosine}
    \end{subfigure} 
        \begin{subfigure}[b]{0.45\textwidth}
        \centering
        \includegraphics[width=0.7\textwidth]{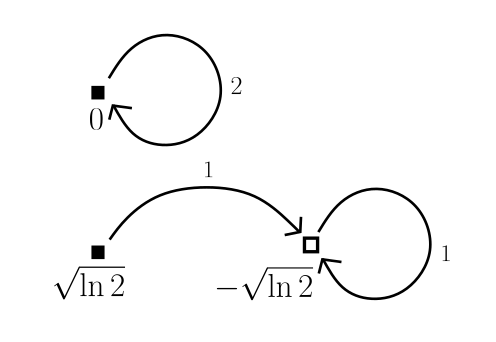}
        \caption{$G_2(z) = \sqrt{\ln 2} (1 - \exp(z^2))$.}
        \label{fig:portrait_multi_error}
    \end{subfigure}
    \caption{Examples of postsingular portraits}
    \label{fig: portraits} 
\end{figure}

\subsection{Quasincoformal maps and Teichm\"{u}ller spaces}\label{subsec: quasiconformal maps and teichmuller spaces}

We say that an orientation-preserving homeomorphism $\varphi\colon \C \to \C$ is \textit{$K$-quasiconformal} for $K\geq 1$ if for every annulus $V \subset \C$ of finite modulus, we have $$\mod(V)/K \leq \mod(\varphi(V)) \leq K \mod(V),$$ where $\mod(V)$ and $\mod(\varphi(V))$ are the moduli of the annuli $V$ and $\varphi(V)$, respectively (see \cite[Proposition 3.2.1]{Hubbard_Book_1} for the definition of modulus). 

The infimum of all values $K$ such that $\varphi$ is $K$-quasiconformal is called the \textit{dilatation} of $\varphi$, and it is denoted by $K(\varphi)$ (in particular, $\varphi$ is $K(\varphi)$-quasiconformal). For further properties and equivalent definitions of quasiconformal maps, see \cite{Ahlfors}, \cite{Hubbard_Book_1}, or \cite{branner_fagella_2014}.

\begin{proposition} \label{prop: qc homeos fixing three points}
    Let $\varphi_n\colon \C \to \C, n \in \N$ be a sequence of $K_n$-quasiconformal maps fixing two distinct points $a, b \in \C$. Suppose that $(K_n)$ converges to 1 as $n\to \infty$, then the sequence~$(\varphi_n)$ converges locally uniformly to $\mathrm{id}_{\C}$. 
\end{proposition}

\begin{proof}
    Let us consider an arbitrary subsequence $(\varphi_{n_k})$ of the sequence $(\varphi_n)$. Theorem \cite[Theorem 1.26]{branner_fagella_2014} implies that from $(\varphi_{n_k})$, we can extract a further subsequence converging locally uniformly to the quasiconformal limit $\varphi \colon \C \to \C$. Moreover, from the same theorem, one can easily derive that the dilatation of $\varphi$  equals 1. Thus, by Weyl's Lemma \cite[Theorem 1.14]{branner_fagella_2014}, we see that $\varphi$ is holomorphic; since it fixes two distinct points of $\C$, it coincides with  $\id_{\C}$. This argument applies to an arbitrary subsequence of $(\varphi_n)$. Therefore, we can conclude that the full sequence $(\varphi_n)$ converges locally uniformly (on $\C$) to $\id_{\C}$.
\end{proof}

Let $A \subset \R^2$ be a finite set. Then \textit{Teichm\"{u}ller space of $\R^2$ with the marked set $A$} (or Teichm\"{u}ller space \textit{modelled on} $\R^2 \setminus A$) is defined as
$$
    \Teich (\R^2,A) := \{\varphi | \varphi: \R^2 \rightarrow \C \text{ is an orientation-preserving homeomorphism}\} / \sim
$$
where $\varphi \sim \psi$ if there exists an affine map $M$ such that $\varphi $ is isotopic rel.\ $A$ to $M \circ \psi$. 

Note that $\Teich (\R^2,A)$ is essentially the same space as more commonly known Teichm\"{u}ller space of the topological sphere $\S^2 = \R^2 \cup \{\infty\}$ with the marked set $A \cup \{\infty\}$ (see, for instance, \cite[Section 1.1]{Buff}). However, since transcendental Thurston maps cannot be extended to $\S^2$, the point $\infty$ plays a special role. Hence, for every $\tau = [\varphi] \in \Teich(\S^2, A \cup \{\infty\})$, we make a choice of representative $\varphi$ so that $\varphi(\infty) = \infty$ leading to the definition of $\Teich(\R^2, A)$~above.

    
The space $\Teich(\R^2, A)$ has a natural structure of a complex $(|A| - 2)$-dimensional manifold \cite[Theorem 6.5.1]{Hubbard_Book_1}. It is also possible to define a complete metric, called the \textit{Teichm\"{u}ller metric}, on $\Teich(\R^2,A)$ by setting
$$
    d([\varphi_1],[\varphi_2]) := \inf_{\psi} \log K(\psi),
$$
where the infimum is taken over all quasiconformal homeomorphisms $\psi \colon \C \to \C$ such that $\varphi_1$ and $\psi \circ \varphi_2$ are isotopic rel.\ $A$ (see \cite[Section 6.4]{Hubbard_Book_1} for details).


\subsection{Thurston pullback map}\label{subsec: thurston pullback map}

In this section, we introduce the notion of Thurston pullback map for an (entire) Thurston map and provide its essential properties. This notion is classical for rational Thurston maps (e.g., \cite[Definition 10.6.1]{Hubbard_Book_1}), but it is less well-known in the transcendental setting (see \cite{HSS} for the case of \textit{exponential Thurston maps} and \cite[Section 5.1]{Pfrang_thesis}, \cite{Shemyakov_Thesis} for the general case).

\begin{proposition}\label{prop: def of sigma map}
    Let $f \colon (\R^2, A) \righttoleftarrow$ be a Thurston map and $\varphi \colon \R^2 \to \C$ be an orientation-preserving homeomorphism. Then there exists an orientation-preserving homeomorphism $\psi\colon \R^2 \to \C$ such that $g_\varphi := \varphi \circ f \circ \psi^{-1}\colon \C \to \C$ is an entire holomorphic map. In other words, the following diagram commutes
   $$
    \begin{tikzcd}
        (\R^2, A) \arrow[r,"\psi"] \arrow[d,"f"] & (\C, \psi(A)) \arrow[d,"g_\varphi := \varphi \circ f \circ \psi^{-1}"]\\
     (\R^2, A) \arrow[r,"\varphi"] & (\C, \varphi(A))\\
    \end{tikzcd}
    $$
    The homeomorphism $\psi$ is unique up to post-composition with an affine map. Different choices of $\varphi$ that represent the same point in  $\Teich(\R^2, A)$ yield maps $\psi$ that represent the same point in $\Teich(\R^2, A)$. 
    
    In other words, we have a well-defined map $\sigma_f \colon \Teich(\R^2, A) \to \Teich(\R^2, A)$ such that $\sigma_f([\varphi]) = [\psi]$, called the Thurston pullback map (or simply the $\sigma$-map) associated with $f$. As $\varphi$ ranges across all maps representing a single point in $\Teich(\R^2,A)$, the function $g_\varphi$ is uniquely defined up to pre- and post-composition with affine maps.
\end{proposition}

\begin{proof}
    The existence of a homeomorphism $\psi$ follows from Proposition \ref{prop: pulling back complex structures} and the Uniformization theorem. It is also clear that $\psi$ is unique up to post-composition with an affine map and that post-composing $\varphi$ with an affine map does not affect $\psi$.
    Due to Proposition \ref{prop: isotopy lifting property}, changing $\varphi$ by isotopy rel.\ $A$ does not change $[\psi]$. Thus, changing $\varphi$ within its equivalence class in $\Teich(\R^2,A)$ does not affect $[\psi]$, showing that $\sigma_f$ introduced as above is well-defined. These arguments also show that $g_\varphi$ is uniquely defined up to pre- and post-composition with affine maps.
\end{proof}
It is important to note that $\sigma_f$ is a holomorphic map with respect to the natural complex structure on $\Teich(\R^2, A)$ (for a proof, see, for instance, \cite[Section 1.3]{Buff}). Moreover, it is well-behaved with respect to the Teichm\"{u}ller metric, as the following proposition suggests. 

\begin{proposition}\label{prop: sigma is contracting}
    Let $f\colon (\R^2, A) \righttoleftarrow$ be a Thurston map, then

    \begin{enumerate}
        \item $\sigma_f$ is 1-Lipschitz, i.e., $d(\sigma_f(\tau_1), \sigma_f(\tau_2)) \leq d(\tau_1, \tau_2)$ for every $\tau_1, \tau_2 \in \Teich(\R^2, A)$;
    
        \item if $f$ is transcendental, then $\sigma_f$ is locally uniformly contracting, i.e., for any compact set $K \subset \Teich(\R^2, A)$ there exists $\varepsilon_K > 0$ such that $d(\sigma_f(\tau_1), \sigma_f(\tau_2))~\leq~(1~-~\varepsilon_K)~d(\tau_1, \tau_2)$ for every $\tau_1, \tau_2 \in K$;

        \item if $f$ is polynomial, then $\sigma^{\circ 2}_f$ is locally uniformly contracting.
    \end{enumerate}
    
\end{proposition}

\begin{proof}
    The first item follows from the fact that every holomorphic map on $\Teich(\R^2, A)$ is 1-Lipschitz \cite[Corollary 6.10.7]{Hubbard_Book_1}.
    The second item follows from \cite[Section 3.2]{HSS}
    (see also \cite[Chapter 5.1]{Pfrang_thesis} and \cite{Shemyakov_Thesis}). The last item follows from \cite[Corollary~10.7.8]{Hubbard_Book_2}.
\end{proof}    
    

\begin{remark}\label{rem: remark on sigma maps}
    If Thurston maps $f \colon (\R^2, A) \righttoleftarrow$ and $\widehat{f} \colon (\R^2, B) \righttoleftarrow$ are Thurston equivalent, then $\sigma_{f} \colon \Teich(\R^2, A) \to \Teich(\R^2, A)$ and $\sigma_{\widehat{f}} \colon \Teich (\R^2, B) \to \Teich(\R^2, B)$ are conjugate by a biholomorphism. In the special case where $A = B$ and $f$ is isotopic rel.\ $A$ to $\widehat{f}$, we have~$\sigma_{f} = \sigma_{\widehat{f}}$. 
\end{remark}

One of the most crucial properties of Thurston pullback maps is summarized by the following proposition, and follows from Definition \ref{def: thurston equivalence} and Proposition \ref{prop: def of sigma map} (also cf. \cite[Theorem 10.6.4]{Hubbard_Book_2} and \cite[Theorem 3.1]{HSS}).

\begin{proposition}\label{prop: fixed point of sigma}
    A Thurston map $f\colon (\R^2, A) \righttoleftarrow$ is realized if and only if the  Thurston pullback map $\sigma_f\colon \Teich(\R^2, A) \to \Teich(\R^2, A)$ has a fixed point $\tau \in \Teich(\R^2, A)$.
\end{proposition}

\begin{proof}
    Suppose that $f$ is realized by a postsingularly finite entire map $g\colon (\C, B) \righttoleftarrow$. Then by Definition~\ref{def: thurston equivalence}, there exist orientation-preserving homeomorphisms $\varphi, \psi \colon \R^2 \to \C$ such that $\varphi(A) = \psi(A) = B$, $\varphi$ and $\psi$ are isotopic rel. $A$, and $\varphi \circ f = g \circ \psi$. Clearly, $\tau = [\varphi] = [\psi] \in \Teich(\R^2, A)$ is a fixed point of $\sigma_f$.

    Now suppose that $\tau = [\varphi] \in \Teich(\R^2, A)$ is a fixed point of $\sigma_f$. Let $\psi\colon \R^2 \to \C$ be an orientation-preserving homeomorphism so that $g_\varphi = \varphi \circ f \circ \psi^{-1}$ is entire. Since $[\varphi] = [\psi]$, by post-composing $\psi$ with an affine map we can assume that $\varphi|A = \psi|A$, which in turn implies that $\varphi$ and $\psi$ are isotopic rel. $A$. Thus, $g_\varphi\colon (\C, \varphi(A)) \righttoleftarrow$ is a postsingularly finite entire map Thurston equivalent to $f$.
\end{proof}

\begin{remark}\label{rem: two postsingular values}
    Since $\Teich(\R^2, A)$ consists of one point for $|A| = 2$, Proposition \ref{prop: fixed point of sigma} immediately implies that every Thurston map with two postsingular values is realized.
\end{remark}

Finally, Propositions~\ref{prop: sigma is contracting} and \ref{prop: fixed point of sigma} lead to the following result called \textit{Thurston's rigidity} (cf. \cite[Corollary 10.7.8]{Hubbard_Book_1}).

\begin{proposition}\label{prop: rigidity}
    Let $f \colon (\R^2, A) \righttoleftarrow$ be a Thurston map, then
    \begin{enumerate}
        \item $\sigma_f$ can have at most one fixed point;

        \item \label{it: rigidity} if $g_1\colon (\C, A_1) \righttoleftarrow$ and $g_2\colon (\C, A_2) \righttoleftarrow$ are postsingularly finite entire maps Thurston equivalent to $f$, then $g_1$ and $g_2$ are conjugate by an affine map.
    \end{enumerate}
\end{proposition}


\subsection{Admissible quadruples}\label{subsec: admissible quadruples}

This section discusses a combinatorial construction of entire topologically holomorphic maps of finite type. More precisely, we introduce combinatorial models based on planar embedded graphs called \textit{admissible quadruples} (Definition \ref{def:  admissible quadruple}) and show that every such quadruple defines an entire topologically holomorphic map of finite type and vice versa (Propositions \ref{prop: quadruple by map} and \ref{prop: map by quadruple}). We also introduce a notion of equivalence for admissible quadruples and show that equivalent quadruples correspond to equivalent maps (see Definition \ref{def: equivalence of quadruples} and Proposition \ref{prop: equivalence of quadruples and maps} for details). Although the framework introduced in this section is non-dynamical, in Section~\ref{subsec: polynomial approximations} we will showcase its utility in exploring dynamical properties of Thurston maps.


The constructions in this section are very similar to \textit{line complexes} (see \cite[Section~2]{nevanlinna} and \cite[Section XI]{Goldberg}) and \textit{rose maps} used for defining topological polynomials in \cite{portraits}. We refrain from borrowing the above objects directly for the sake of continuity. Additionally, we provide short proofs for all the necessary statements, even though they are mostly folklore in the context of the objects mentioned above. 

Let $\Rose$ be a counterclockwise directed rose graph that surrounds a finite set $A \subset \R^2$ (for a definition, see Section~\ref{subsec: planar embedded graphs}) with $m$ petals $p_1, p_2, \dots, p_m$ written in the counterclockwise order around the center $t$ of $\Rose$. Denote by $P_j$ the face of $\Rose$ surrounded by $p_j$ for each $j = 1,2, \dots, m$, and by $P_\infty$ the unique unbounded face of $\Rose$.  
Label the points $A$ as $a_1,a_2,..., a_m$, such that $a_j \in P_j$ for each $j = 1,2, \dots, m$.

Let $\Gamma$ be a $2m$-regular connected planar embedded graph, and $\Phi: \Gamma \rightarrow \Rose$ be a covering map such that $\Phi(v) = t$ for all $v \in V(\Gamma)$ and $\Phi(e) \in E(\Rose)$ for all $e \in E(\Gamma)$. We view the graph ~$\Gamma$~as a directed graph since the map $\Phi$ naturally induces an orientation on $\Gamma$ (see Section~\ref{subsec: planar embedded graphs}).


Suppose that $m \geq 2$. Consider an arbitrary face $F \in F(\Gamma)$. We label $F$ by $P_j$ if $\Phi(E(\partial F)) = \{p_j\}$, and if no such $j$ exists (i.e., $|\Phi(E(\partial F))| > 1$) we label $F$ by $P_\infty$. We denote by $\Gamma^*$ the directed planar embedded graph obtained by subdividing each edge of $\Gamma$. 

\begin{definition}\label{def: admissible quadruple}
    Consider a quadruple $\Delta = (A, \Rose, \Gamma, \Phi)$, where $A,\Rose,\Gamma$, and $\Phi$ are as above. We say that $\Delta$ is an \textit{admissible quadruple} if $m = 1$ or if the following conditions are satisfied: 
    \begin{enumerate}
        \item every face $F \in F(\Gamma)$ labeled by $P_\infty$ is unbounded, and
        \item \label{it: admissiblity} for each $v \in V(\Gamma)$, the set of edges of $\Gamma^*$ incident to $v$ can be written in the counterclockwise order as $e_1 = e_{2m+1}, e_2, e_3,..,e_{2m}$ such that the following conditions are satisfied for each $j = 1, 2, \dots, m$:
        \begin{itemize}
        \item $e_{2j - 1}$ is outgoing at $v$ and $e_{2j}$ is incoming at $v$;
        
        
        \item the edges $e_{2j-1}$ and $e_{2j}$ belong to the boundary of the same face $F \in F(\Gamma)$ labelled by $P_j$;
        
        \item the edges $e_{2j}$ and $e_{2j + 1}$ belong to the boundary of the same face $F \in F(\Gamma)$ labelled by $P_\infty$.
        \end{itemize}
    \end{enumerate}

    The set $A$ is called the \textit{marked set} of $\Delta$.

\end{definition}

\begin{remark}\label{rem: boundary}
    Suppose that $\Delta = (A, \Rose, \Gamma, \Phi)$ is an admissible quadruple and $F$ is a face of $\Gamma$. If $F$ is not labelled by $P_\infty$, then clearly, $\partial F$ is a counterclockwise directed cycle if $F$ is bounded, and an infinite unilaterally connected chain, otherwise. If $F$ is labelled by $P_\infty$, then $\partial F$ is a unilaterally connected graph.
\end{remark}

Natural examples of admissible quadruples are provided by preimages of rose graphs under entire topologically holomorphic maps of finite type. More precisely, suppose that $f\colon~\R^2~\to~\R^2$ is topologically holomorphic and rose graph $\Rose$ surrounds the set $A$, where $S_f \subset A$. Denote by $\Delta(A, \Rose, f)$ the quadruple $(A, \Rose, f^{-1}(\Rose), \Phi_{\Rose, f})$, where $\Phi_{\Rose, f}(x) = f(x)$ for each $x \in f^{-1}(\Rose)$. 

\begin{proposition}\label{prop: quadruple by map}
    Let $f$, $A$, and $\Rose$ be as above. Then $\Delta(A, \Rose, f)$ is an admissible quadruple. Moreover, if $m \geq 2$, then for each face $F$ of $f^{-1}(\Rose)$ the following properties hold:
    \begin{enumerate}
        \item if $F$ is bounded and labeled by $P_j$ for some $j \in \{1,2,...,m\}$, then $f(F) = P_j$ and $|F \cap f^{-1}(A)| = 1$. If $|V(\partial F)| = 1$, then $F$ does not contain any critical points of $f$ and $f|F$ is injective; otherwise, $F$ contains a unique critical point $x_F$ with $\deg(f, x_F) = |V(\partial F)|$;
        

        \item if $F$ is unbounded and labeled by $P_j$ for some $j \in \{1,2,...,m\}$, then $f(F) = P_j \setminus \{a_j\}$ and $F\cap f^{-1}(A) = \emptyset$. In particular, $a_j \in S_f$ and $F$ is an asymptotic tract of $f$~over~$a_j$;

        
        \item if $F$ is labeled by $P_\infty$, then $f(F) = P_\infty$ and $F \cap f^{-1}(A) = \emptyset$. In particular, $f$ restricts to a universal covering map from $F$ to $P_\infty$.
        
    \end{enumerate}
\end{proposition}

\begin{proof}
    First, we show that $f^{-1}(\Rose)$ is connected. Consider any two distinct vertices $u$ and~$v$ of $f^{-1}(\Rose)$. There exists a path $\alpha \colon \I \to \R^2 \setminus f^{-1}(A)$ joining $u$ and $v$. Note that~$f \circ \alpha$ is a loop in $\R^2 \setminus A$ based at $t$. Assuming that each petal $p_j$ is parameterized by a loop $\alpha_j\colon \I \to~p_j$ based at $t$, it follows that $f \circ \alpha$ is homotopic rel.\ $A$ to a loop $\gamma = \gamma_1\cdot\gamma_2\cdot\dots\cdot\gamma_k$, where $\gamma_{l} \in~\{\alpha_1, \alpha_2, \dots, \alpha_m, \overline{\alpha_1}, \overline{\alpha_2}, \dots, \overline{\alpha_m}\}$ for each $l \in \{1,2,\dots,k\}$. By the homotopy lifting property, the path $\alpha$ is homotopic rel. $f^{-1}(A)$ to the $f$-lift $\widetilde{\gamma}$ of $\gamma$ starting at $u$. In particular, $\widetilde{\gamma}$ joins $u$ with $v$.  Since $\widetilde{\gamma} \subset f^{-1}(\Rose)$, the vertices $u$ and $v$ belong to the same connected component of $f^{-1}(\Rose)$.

    Proposition \ref{prop: preimages with infinity} implies that every face of $f^{-1}(\Rose)$ labeled by $P_\infty$ is unbounded. Finally, the regularity of $f^{-1}(\Rose)$ and condition (\ref{it: admissiblity}) follow from the fact that $f$ is locally injective and orientation-preserving at every $v \in V(f^{-1}(\Rose))$. Thus, the quadruple $\Delta(A, \Rose, f)$ is admissible.

    The rest of the statement easily follows from Propositions \ref{prop: preimages without infinity} and \ref{prop: preimages with infinity}.
\end{proof}

\begin{remark}\label{rem: quadruples with one marked point}
    Let $\Delta = (A, \Rose, \Gamma, \Phi)$ be an admissible quadruple with $m = |E(\Rose)| = 1$. In this case, $\Gamma$ is a counterclockwise directed cycle or an infinite unilaterally connected chain. It is easy to see that there exists a map $f$ such that $\Delta(A,\Rose,f) = \Delta$ satisfying $f = \varphi^{-1} \circ g \circ \psi$ for some orientation-preserving homeomorphisms $\varphi, \psi \colon \R^2 \to \C$, where $g(z) = z^d$ if $d = |V(\Gamma)|<\infty$, and $g(z) = \exp(z)$ if $\Gamma$ is infinite. In particular, using Proposition \ref{prop: small postsingular set}, we can formulate a statement close in spirit to Proposition \ref{prop: quadruple by map} for the case $m = 1$.
\end{remark}

Now assume that $f \colon \R^2 \to \R^2$ is an arbitrary topologically holomorphic map such that $S_f \subset A$ and $\Delta(A, \Rose, f) = \Delta$. 
Since $\Rose$ is a deformation retract of $\R^2 \setminus A$ and $\Gamma$ is a deformation retract of $\R^2 \setminus f^{-1}(A)$ as Proposition \ref{prop: quadruple by map} and Remark \ref{rem: quadruples with one marked point} suggest, for any $v \in V(\Gamma)$, we see that $f_*\pi_1(\R^2 \setminus f^{-1}(A), v) = \Phi_*\pi_1(\Gamma,v)$, and 
\begin{align*}
    \Phi_* \pi_1 (\Gamma, v) = \{[\Phi(\delta_1) \cdot \Phi(\delta_2) \cdot \dots \cdot \Phi(\delta_l)]:& \text{ each }\delta_i \text{ is a path parameterizing an edge of $\Gamma$,}\\ &\text{ and }\delta_1 \cdot \delta_2 \cdot \dots \cdot \delta_l \text{ is a loop in $\Gamma$ based at $v$}
    \}.
\end{align*}
Now suppose that $\gamma$ is a loop based at the center of the rose graph $\Rose$ such that $\gamma$ is homotopic rel.\ $A$ to $\gamma_1 \cdot \gamma_2 \cdot \dots \cdot \gamma_k$, where $\gamma_j \in \{\alpha_1, \alpha_2, \dots, \alpha_m, \overline{\alpha_1}, \overline{\alpha_2}, \dots, \overline{\alpha_m}\}$ and $\alpha_j$ is a parametrization of $p_j$ with $\alpha_j(0) = t$ for each~$j = 1, 2, \dots, m$. If we know  $\Phi$, we can reconstruct $f$-lift of $\gamma$ starting at any vertex $v \in V(\Gamma)$ up to homotopy rel.\ $f^{-1}(A)$.

The following result is the converse to Proposition \ref{prop: quadruple by map}:

\begin{proposition}\label{prop: map by quadruple}
    Let $\Delta = (A, \Rose, \Gamma, \Phi)$ be an admissible quadruple. Then there exists a topologically holomorphic map $f\colon \R^2 \to \R^2$ of finite type such that $S_f \subset A$ and $\Delta(A, \Rose, f)~=~\Delta$.

    
\end{proposition}

\begin{proof}

    When $|A| = 1$, the desired result follows from Remark \ref{rem: quadruples with one marked point}. Thus, we assume $|A| \geq 2$ and give an outline of the construction. 

    First, we define $f$ on $\Gamma$ simply by setting $f|\Gamma := \Phi$.  Choose orientation-preserving homeomorphisms $\varphi_{P_j}\colon \D \to P_j$ and $\varphi_{P_\infty} \colon \H \to P_\infty$. Similarly, since each face $F$ of $\Gamma$ is simply connected, choose an orientation-preserving homeomorphism $\psi_F : X \rightarrow F$ where $X = \D$ if $F$ is bounded, and $X=\H$ if $F$ is unbounded. 
    
    Given a face $F$ with label $P \in F(\Rose)$, we define $f|F$ so that

    \begin{enumerate}
        \item if $F$ is bounded, we let $f|F := \varphi_P \circ g_d \circ \psi_F^{-1}$, where $d = |V(\partial F)|$ and $g_d(z) = z^d$, 
        \item if $F$ is unbounded, let $f|F := \varphi_P \circ \exp \circ \psi_F^{-1}$.
    \end{enumerate}

    Due to admissibility conditions and Remark \ref{rem: boundary}, the sets of homeomorphisms  $\{\varphi_P\}_{P \in F(\Rose)}$ and $\{\psi_F\}_{F \in F(\Gamma)}$ can be chosen so that the map $f$ we construct above is continuous. Finally, one can show that $f$ acts locally as a power map (in the sense of Definition \ref{def: topologically holomorphic}) and $S_f \subset A$. Thus, $f$ is topologically holomorphic, has finite type, and satisfies  $\Delta(A, \Rose, f) = \Delta$.
\end{proof}

\begin{definition}\label{def: equivalence of quadruples}
    Two admissible quadruples $\Delta_1 = (A, \Rose_1, \Gamma_1, \Phi_1)$ and $\Delta_2 = (A, \Rose_2, \Gamma_2, \Phi_2)$ are said to be \textit{equivalent} if there exist $\theta \in \Homeo_0^+(\R^2, A)$ and $\varphi \in \Homeo^+(\R^2, A)$ such that 
\begin{enumerate}
    \item $\theta$ is an isomorphism between $\Rose_1$ and $\Rose_2$; 

    \item $\varphi$ is an isomorphism between $\Gamma_1$ and $\Gamma_2$; 

    \item $\theta \circ \Phi_1 = \Phi_2 \circ \varphi$.
\end{enumerate}
\end{definition}

It is easy to see that Definition \ref{def: equivalence of quadruples} provides an equivalence relation on the set of all admissible quadruples with a fixed marked set.

\begin{proposition}\label{prop: equivalence of quadruples and maps}
    Let $A \subset \R^2$ be finite, and $\Rose_1$, $\Rose_2$ be rose graphs surrounding $A$ such that $\Rose_1$ is isotopic rel.\ $A$ to $\Rose_2$.
    Let $f_1 \colon \R^2 \to \R^2$ and $f_2\colon \R^2 \to \R^2$ be topologically holomorphic maps such that $S_{f_1} \subset A$ and $S_{f_2} \subset A$. Then $\Delta(A, \Rose_1, f_1)$ and $\Delta(A, \Rose_2, f_2)$ are equivalent if and only if there exists a homeomorphism $\psi \in \Homeo^+(\R^2)$ such that $f_1 = f_2 \circ \psi$. 
    
    Moreover, if $f_1$ and $f_2$ are holomorphic, and $\Delta(A, \Rose_1, f_1)$ and $\Delta(A, \Rose_2, f_2)$ are equivalent, then the map $\psi$ is an affine transformation.
\end{proposition}

\begin{proof}

	Suppose that the admissible quadruples $\Delta(A, \Rose_1, f_1)$ and $\Delta(A, \Rose_2, f_2)$ are equivalent. Due to Proposition \ref{prop: isotopy lifting property} we can assume that $\Rose_1 = \Rose_2 = \Rose$ and that
    the equivalence between $\Delta(A, \Rose_1, f_1)$ and $\Delta(A, \Rose_2, f_2)$ is provided by $\theta = \id_{\R^2}$ and an orientation-preserving homeomorphism $\varphi$. By pre-composing $f_2$ with $\varphi$ we can further assume $\theta = \varphi = \id_{\R^2}$, i.e., $\Delta(A, f_1, \Rose_1) = \Delta(A, f_2, \Rose_2) = (A, \Rose, \Gamma, \Phi)$. Then by the previous discussions, we have that 
    $$
        (f_1)_* \pi_1(\R^2 \setminus f_1^{-1}(A), v) = \Phi_* \pi_1(\Gamma, v) = (f_2)_* \pi_1(\R^2 \setminus f_2^{-1}(A), v)
    $$
    for every $v \in V(\Gamma)$. By the classical theory of covering maps, there exists an orientation-preserving homeomorphism $\psi\colon \R^2 \setminus f_1^{-1}(A) \to \R^2 \setminus f_2^{-1}(A)$ such that $f_1 = f_2 \circ \psi$ on $\R^2 \setminus f_1^{-1}(A)$. Since $f_1^{-1}(A)$ is a discrete subset of $\R^2$, we can extend $\psi$ to $\R^2$, still satisfying $f_1 = f_2 \circ \psi$.

Conversely, let us suppose there exists $\psi \in \Homeo^+(\R^2)$ such that $f_1 = f_2 \circ \psi$. By our assumptions on $\Rose_1$ and $\Rose_2$, we can find  $\theta \in \Homeo_0^+(\R^2, A)$ such that $\theta(\Rose_1) = \Rose_2$. Then by Proposition \ref{prop: isotopy lifting property}, there exists $\varphi \in \Homeo^+(\R^2)$ such that $\theta \circ f_1 = f_2 \circ \varphi$, and the rest easily follows.
    

  If $f_1$ and $f_2$ are holomorphic, any homeomorphism $\psi$ satisfying $f_1 = f_2 \circ \psi$ is holomorphic, and  therefore affine.
\end{proof}   



We say that a topologically holomorphic map $f \colon \R^2 \to \R^2$ \textit{realizes} an admissible quadruple $\Delta = (A, \Rose, \Gamma, \Phi)$ or, equivalently, $\Delta$ \textit{defines} $f$ if $S_f \subset A$ and $\Delta(A, \Rose, f)$ is equivalent to $\Delta$. In particular, Propositions \ref{prop: map by quadruple} and \ref{prop: equivalence of quadruples and maps} imply that every admissible quadruple $\Delta$ defines an entire topologically holomorphic map $f$ of finite type, which is unique up to pre-composition by an orientation-preserving homeomorphism.

Note that an admissible quadruple $\Delta = (A, \Rose, \Gamma, \Phi)$ is a combinatorial object even though~$\Phi$ is a continuous map. In fact, to define the map $\Phi$ uniquely (up to a certain notion of equivalence introduced below), it is sufficient to know the images $\Phi(e), e \in E(\Gamma)$ and the orientation of the graph $\Gamma$ induced by $\Phi$. Indeed, suppose that $\Psi \colon \Gamma \to \Rose$ is a covering map such that $\Psi(v) = t$ for each vertex $v \in V(\Gamma)$, $\Phi(e) = \Psi(e)$ for each edge $e \in E(\Gamma)$, and the orientations of $\Gamma$ induced by the maps $\Phi$ and $\Psi$ coincide. In this case, it is clear that there exists a homeomorphism $\varphi \colon \Gamma \to \Gamma$ isotopic rel.\ $V(\Gamma)$ to $\id_\Gamma$ such that $\Psi = \Phi \circ \varphi$. In particular, the orientation of the graph $\Gamma$ and the images of its edges under the map $\Phi$ uniquely define the equivalence class of the admissible quadruple $\Delta$.

The language of admissible quadruples provides a convenient way of thinking about entire topologically holomorphic maps of finite type, which we demonstrate in the following example.

\begin{example}\label{ex: non dyn quadruples}
    
	Let $A = \{-1, 1\}$ be the set represented by solid black squares at the top of Figure \ref{fig: nondyn_quadruple for cosine}. Denote by $\Rose_1$ and $\Gamma_1$ the planar embedded graphs shown at the top and bottom of Figure \ref{fig: nondyn_quadruple for cosine}, respectively. The map $\Phi_1\colon \Gamma_1 \to \Rose_1$ is a covering that maps each edge of $\Gamma_1$ to the unique edge of $\Rose_1$ of the same color. Arrows on the edges of the graphs $\Rose_1$ and~$\Gamma_1$ indicate the orientations of the corresponding graphs. It is straightforward to check that $\Delta_1 = (A, \Rose_1, \Gamma_1, \Phi_1)$ is an admissible quadruple, which is realized by the map $g_1(z) = \cos(z)$. Thus, any entire (holomorphic) map realizing $\Delta_1$ has the form $\cos(a z + b)$ for some constants $a, b \in \C$ with $a \neq 0$.


	
	Figure \ref{fig: nondyn_quadruple or eg2} is analogous to Figure \ref{fig: nondyn_quadruple for cosine}, and provides another example $\Delta_2 = (A, \Rose_2, \Gamma_2, \Phi_2)$ of an admissible quadruple. This quadruple can be shown to be realized by the map $g_2(z) = 2 \exp(z^2) - 1$. In particular, the planar embedded graph $\Gamma_2$ has two unbounded faces that correspond to the asymptotic tracts of $g_2$ over $w_0 = 1$, two unbounded faces that correspond to the asymptotic tracts of $g_2$ over $\infty$, and the only face with the boundary consisting of two edges, corresponding to the unique critical point of $g_2$ (see Proposition \ref{prop: quadruple by map}).
	
	\end{example}

\begin{figure}[h]
    \centering
    \includegraphics[scale=0.5]{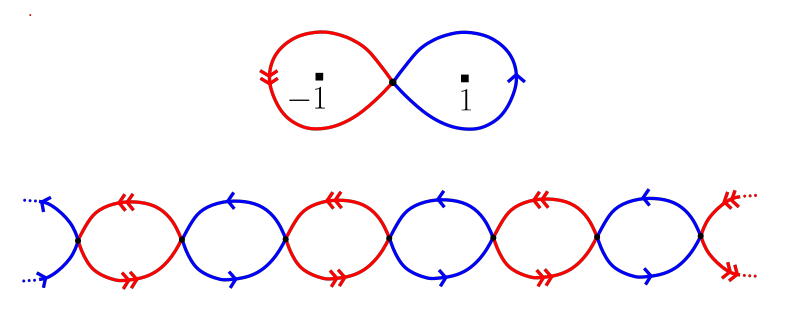}
\caption{Admissible quadruple realized by $g_1(z) = \cos(z)$.}
\label{fig: nondyn_quadruple for cosine}
\end{figure}



\begin{figure}[h]
    \centering
    \includegraphics[scale=0.5]{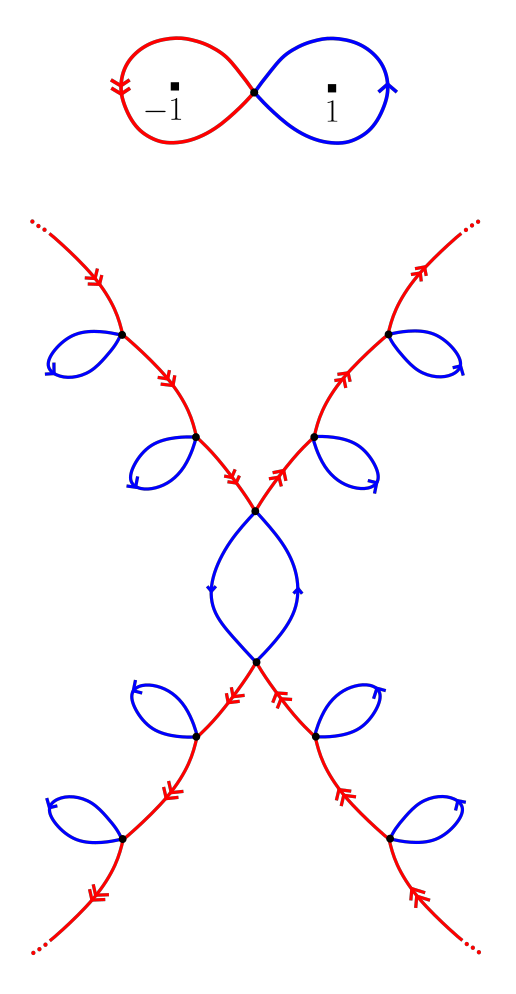}
\caption{Admissible quadruple realized by $g_2(z) = 2\exp(z^2) - 1$.}
\label{fig: nondyn_quadruple or eg2}
\end{figure}

\begin{definition}
    Let $\Delta$ be an admissible quadruple. If every topologically holomorphic map $f\colon \R^2 \to \R^2$ realizing $\Delta$ is parabolic, then we say that $\Delta$ is \textit{parabolic}. If every such map $f$ is hyperbolic, then $\Delta$ is called \textit{hyperbolic}.
\end{definition}

Proposition \ref{prop: equivalence of quadruples and maps} implies that every admissible quadruple is either parabolic or hyperbolic. Moreover, from Propositions \ref{prop: pulling back complex structures} and \ref{prop: equivalence of quadruples and maps}, it follows that every parabolic admissible quadruple $\Delta$ with a marked set $A$ defines an entire holomorphic map of finite type with $S_f \subset A$, which is unique up to pre-composition with an affine transformation.

\section{Approximations of entire Thurston maps}\label{sec: approximations of entire thurston maps}
 
In this section, we introduce and study different notions of convergence for sequences of entire Thurston maps having the same marked set $A$. In fact, we provide two different points of view (one  combinatorial, and one topological, see Definitions \ref{def: combinatorial convergence}~and~\ref{def: topological convergence}), and then show that these notions are equivalent (see Proposition \ref{prop: equivalence of convergences}). Afterward, we provide a construction that allows us to approximate, in an appropriate sense, an arbitrary transcendental Thurston map by a sequence of polynomial Thurston maps (see Proposition \ref{prop: polynomial approximations}). Finally, we study how this impacts the convergence of the corresponding Thurston pullback maps and establish Main Theorem \ref{mainthm: main theorem B}.

\subsection{Combinatorial convergence}\label{subsec: combinatorial convergence}

Suppose that $f\colon (\R^2, A)\righttoleftarrow$ is a Thurston map with $f(b) = t$ for some $b, t \in \R^2 \setminus A$. If $\gamma$ is a loop in $\R^2 \setminus A$ based at $t$, then we use the notation $\gammal(f, b)$ as a shorthand for the $f$-lift of $\gamma$ starting at $b$. 

Then for every loop $\gamma \subset \R^2 \setminus A$ based at $t$ we can record
\begin{enumerate}
    \item whether $\gamma$ lifts under $f$ to a loop based at $b$, i.e., whether $[\gamma] \in f_*\pi_1(\R^2 \setminus f^{-1}(A), b)$;

    \item if $[\gamma] \in f_*\pi_1(\R^2 \setminus f^{-1}(A), b)$, homotopy class rel.\ $A$ of $\gammal(f, b)$, i.e., $[\gammal(f, b)] \in \pi_1(\R^2\setminus A, b)$. 
\end{enumerate}

We now establish that this information is enough to define the Thurston map $f$ uniquely up to isotopy rel.\ $A$ (cf. \cite{Kameyama} and \cite[Theorem II.3]{Bartholid_Dudko} for rational Thurston maps).

\begin{proposition} \label{prop: combinatorial data defines map}
    Let $f_1\colon (\R^2, A)\righttoleftarrow$ and $f_2\colon (\R^2, A)\righttoleftarrow$ be Thurston maps such that $f_1(b_1) = t$ for some $b_1 \in \R^2$ and $t \in \R^2\setminus A$.
    Then $f_1$ and $f_2$ are isotopic rel.\ $A$ if and only if there exists $b_2 \in \R^2$ and a path $p \subset \R^2 \setminus A$ joining $b_1$ with $b_2$ such that $f_2(b_2) = t$ and for every loop $\gamma \subset \R^2 \setminus A$ based at $t$ the following holds:
    \begin{enumerate}
        \item \label{it: equal groups} $\gamma$ lifts under $f_1$ to a loop based at $b_1$ if and only if it lifts under $f_2$ to a loop based at $b_2$, in other words,    $(f_1)_*\pi_1(\R^2 \setminus f_1^{-1}(A), b_1) = (f_2)_*\pi_1(\R^2 \setminus f_2^{-1}(A), b_2)$;
        \item \label{it: equal lifts} if $[\gamma] \in (f_1)_*\pi_1(\R^2 \setminus f_1^{-1}(A), b_1)$, then loops $\gammal(f_1, b_1)$ and $p\cdot\gammal(f_2, b_2)\cdot \overline{p}$ are homotopic rel.\ $A$.
    \end{enumerate}
\end{proposition}

\noindent Before proving the above, we formulate a simple observation.


\begin{lemma}\label{lemm: useful prop}
    Suppose $A \subset \R^2$ is finite, and $\varphi$ is an orientation-preserving homeomorphism such that $\varphi(b) = t$ for some $b, t \in \R^2 \setminus A$. Then $\varphi$ is isotopic rel.\ $A$ to $\id_{\R^2}$ if and only if there exists a path $p \subset \R^2\setminus A$ joining $b$ with $t$ such that every loop $\gamma \subset \R^2 \setminus A$ based at $b$ is homotopic rel.\ $A$ to $p \cdot (\varphi \circ \gamma) \cdot \overline{p}$.
\end{lemma}

\begin{proof} 
    Suppose that $\varphi$ is isotopic rel.\ $A$ to $\id_{\R^2}$ via an isotopy $(\varphi_s)_{s \in \I}$, where $\varphi_0 = \id_{\R^2}$ and $\varphi_1 = \varphi$. Define $p: \I \to \R^2 \setminus A$ as $p(s) = \varphi_s(b)$ for each $s \in \I$. Then it is evident that any loop $\gamma \subset \R^2 \setminus A$ based at $b$ and $p \cdot (\varphi \circ \gamma) \cdot \overline{p}$ are homotopic rel.\ $A$ via a homotopy $H_{s} := p_s \cdot (\varphi_s \circ \gamma) \cdot \overline{p_s}$, where $p_s$ is a subpath of $p$ joining $b$ with $\varphi_s(b)$ for every $s \in \I$.

    Now suppose that there exists a path $p \subset \R^2\setminus A$ joining $b$ with $t$ such that every loop $\gamma \subset \R^2 \setminus A$ based at $b$ is homotopic rel.\ $A$ to $p \cdot (\varphi \circ \gamma) \cdot \overline{p} \subset \R^2 \setminus A$. Taking $\gamma$ to be a loop separating a unique point $a \in A$ from the other points of $A$ and using continuity of $\varphi$, we can show that $\varphi(a) = a$ for each $a \in A$. The rest easily follows: for instance, it is a particular case of the much stronger Dehn-Neilsen-Baer theorem \cite[Theorem 8.1, 8.8]{farb}.
\end{proof}

\begin{proof}[Proof of Proposition \ref{prop: combinatorial data defines map}]

    
    Suppose that conditions (\ref{it: equal groups}) and (\ref{it: equal lifts}) are satisfied under some choice of $b_2$ and $p$. Since the restrictions $$f_1|\R^2 \setminus f_1^{-1}(A)\colon \R^2 \setminus f_1^{-1}(A) \to \R^2 \setminus A \text{ and } f_2|\R^2 \setminus f_2^{-1}(A)\colon \R^2 \setminus f_2^{-1}(A) \to \R^2 \setminus A$$ are both covering maps, condition (\ref{it: equal groups}) implies that there exists an orientation-preserving homeomorphism $\varphi\colon \R^2 \setminus f_1^{-1}(A)\to \R^2 \setminus f_2^{-1}(A)$ such that $f_1 = f_2 \circ \varphi$ on $\R^2 \setminus f_1^{-1}(A)$ and $\varphi(b_1) = b_2$. Since $f_1^{-1}(A)$ is discrete, we can extend $\varphi$ uniquely to a homeomorphism of $\R^2$ such that $f_1 = f_2 \circ \varphi$.
    
    Let $\gamma$ be an arbitrary loop in $\R^2 \setminus f_1^{-1}(A)$ based at $b_1$. Then $\varphi \circ \gamma$ is a $f_2$-lift (based at~$b_2$) of the loop $f_1 \circ \gamma$. In particular, due to condition (\ref{it: equal lifts}), the loops $\gamma$ and $p \cdot (\varphi \circ \gamma) \cdot \overline{p}$ are homotopic rel.\ ~$A$. Now Lemma~\ref{lemm: useful prop} implies that $\varphi \in \Homeo_0^+(\R^2, A)$.




    Now suppose that $f_1$ and $f_2$ are isotopic rel.\ $A$, i.e., there exists $\varphi \in \Homeo_0^+(\R^2, A)$ such that $f_1 = f_2 \circ \varphi$. Let us choose $b_2$ to be equal $\varphi(b_1)$. Condition (\ref{it: equal groups}) is then obviously satisfied, and condition (\ref{it: equal lifts}) follows from Lemma~\ref{lemm: useful prop}.
\end{proof}

Now we are ready to formulate one of the key definitions of this section.

\begin{definition}\label{def: combinatorial convergence}
    Let $f_n \colon (\R^2, A) \righttoleftarrow, n \in \N$ be a sequence of Thurston maps. We say that~$(f_n)$ \textit{converges combinatorially} to a Thurston map $f\colon (\R^2, A) \righttoleftarrow$ if there exist $b_n, b, t \in \R^2 \setminus A$, where $f_n(b_n) = f(b) = t$ for every $n \in \N$, and a sequence $(p_n)$ of paths in $\R^2 \setminus A$ joining $b$ with~$b_n$ such that for every loop $\gamma \subset \R^2\setminus A$ based at $t$, there exists $N = N(\gamma)$ so that for all $n \geq N$, the following conditions hold:
    \begin{enumerate}
        \item \label{it: converging groups} $\gamma$ lifts under $f_n$ to a loop starting at $b_n$ if and only if $\gamma$ lifts under $f$ to a loop starting at $b$. In other words, $$\lim\limits_{n \to \infty} (f_n)_*\pi_1(\R^2\setminus f_n^{-1}(A), b_n) = f_*\pi_1(\R^2 \setminus f^{-1}(A), b);$$

        \item \label{it: converging lifts} if $[\gamma] \in f_*\pi_1(\R^2 \setminus f^{-1}(A), b)$, then the loops $\gammal(f, b)$ and $p_n \cdot \gammal(f_n, b_n) \cdot \overline{p_n}$ are homotopic rel.\ $A$.
    \end{enumerate}
\end{definition}



\begin{proposition}\label{prop: comb convergence to poly}
    Let $f_n \colon (\R^2, A) \righttoleftarrow, n \in \N$ be a sequence of Thurston maps converging combinatorially to a polynomial Thurston map $f\colon (\R^2, A) \righttoleftarrow$. Then $f_n$ is isotopic rel.\ $A$ to~$f$ for all sufficiently large $n$.
\end{proposition}

\begin{proof}

    Suppose that $(f_n)$ converges combinatorially to $f$ with respect to some choice of $(b_n), b, t$ and $(p_n)$.
    Since $f$ is a topological polynomial, then $f^{-1}(A)$ is finite and, in particular, the group $f_* \pi_1(\R^2 \setminus f^{-1}(A), b)$ is finitely generated. Thus, if $n$ is large enough, $(f_n)_*\pi_1(\R^2 \setminus f_n^{-1}(A), b_n)$ contains $f_* \pi_1(\R^2 \setminus f^{-1}(A), b)$ as a subgroup. From the classical theory of covering maps, it follows that $f = f_n \circ \varphi_n$ on $\R^2 \setminus f^{-1}(A)$, where $\varphi_n \colon \R^2 \setminus f^{-1}(A) \to \R^2 \setminus f_n^{-1}(A)$ is a covering map sending $b$ to $b_n$. In particular, the topological degrees of the maps $f_n$ and $\varphi_n$ are finite and do not exceed $\deg(f)$ for all sufficiently large $n$.

    Let $\alpha$ be a simple loop in $\R^2$ based at $t$ such that the unique bounded component of $\R^2 \setminus \alpha$ contains the set $A$. Proposition \ref{prop: preimages with infinity} implies that $\alpha^k, k \in \mathbb{Z}$ lifts to a loop based at $b$ under the topological polynomial $f$ if and only if $k$ is divisible by $\deg(f)$; a similar statement holds for the topological polynomials $f_n$, as well. Thus, condition (\ref{it: converging groups}) implies that $\deg(f) = \deg(f_n)$ for all sufficiently large $n$. Hence, the covering map $\varphi_n \colon \R^2 \setminus f^{-1}(A) \to \R^2 \setminus f_n^{-1}(A)$ is an orientation-preserving homeomorphism and, therefore, can be extended to  $\R^2$ so that $f = f_n \circ \varphi_n$. 
    Finally, similarly to the proof of Proposition \ref{prop: combinatorial data defines map}, one can show that that $\varphi_n$ is isotopic rel.\ ~$A$ to $\id_{\R^2}$ using condition(\ref{it: converging lifts}) and Lemma~\ref{lemm: useful prop}.
\end{proof}


The next proposition shows that the notion of combinatorial convergence is independent of choices of representatives within the isotopy classes of considered Thurston maps.

\begin{proposition}\label{prop: independence from isotopy}
    Let $f_n \colon (\R^2, A) \righttoleftarrow, n \in \N$ be a sequence of Thurston maps converging combinatorially to a Thurston map $f\colon (\R^2, A) \righttoleftarrow$. Suppose that Thurston maps $\widehat{f}_n \colon (\R^2, A) \righttoleftarrow$ and $\widehat{f} \colon (\R^2, A) \righttoleftarrow$ are isotopic rel.\ $A$ to $f_n$ and $f$, respectively, for every $n \in \N$. Then the sequence $(\widehat{f}_n)$ converges combinatorially to $\widehat{f}$. 
\end{proposition}

\begin{proof}
    
    First of all, we note that if $\widehat{f}_n = f_n \circ \varphi^{-1}$ and $\widehat{f} = f \circ \varphi^{-1}$, where $\varphi \in \Homeo_0^+(\R^2, A)$, then the statement is obviously true with respect to $\widehat{t} := t, \widehat{b} := \varphi(b), \widehat{b}_n := \varphi(b_n)$, and $\widehat{p}_n :=~\varphi \circ p_n$. This observation allows us to reduce the verification to the case when $\widehat{f}_n = f_n \circ \varphi_n^{-1}$, where $\varphi_n \in \Homeo_0^+(\R^2, A)$, and $\widehat{f} = f$.

    One can easily see that condition (\ref{it: converging groups}) is satisfied for the sequence $(\widehat{f}_n)$ and the map $\widehat{f}$ if we take $\widehat{b} := b$, $\widehat{t} = t$, and $\widehat{b}_n := \varphi_n(b_n)$ for every $n \in \N$.

    Lemma~\ref{lemm: useful prop} implies that for every $n \in \N$ there exists a path $\ell_n \subset \R^2 \setminus A$ joining $b_n$ with $\widehat{b}_n = \varphi_n(b_n)$ such that any loop $\gamma \subset \R^2 \setminus A$ based at $b_n$ is homotopic rel.\ ~$A$ to $\ell_n \cdot (\varphi_n \circ \gamma) \cdot \overline{\ell_n}$. Now let us take $\widehat{p}_n := p_n \cdot \ell_n$ and show that $(\widehat{f}_n)$ converges combinatorially to $\widehat{f}$. Indeed, let $\delta \subset \R^2 \setminus A$ be a loop based at $\widehat{t}$ such that $[\delta]~\in~\widehat{f}_*\pi_1(\R^2 \setminus \widehat{f}^{-1}(A), b)$, then for sufficiently large $n$ we have 
    \begin{align*}
        \deltal(\widehat{f}, \widehat{b}) \sim p_n \cdot \deltal(f_n, b_n) \cdot \overline{p_n} \text{ rel.\ } A,\\
        \deltal(f_n, b_n) \sim \ell_n \cdot \deltal(\widehat{f}_n, \widehat{b}_n) \cdot \overline{\ell_n} \text{ rel.\ } A.
    \end{align*}
    In other words, $\deltal(\widehat{f}, \widehat{b})$ is homotopic rel.\ $A$ to $\widehat{p}_n \cdot \deltal(f_n, b_n) \cdot \overline{\widehat{p}_n}$ and this completes the proof. 
\end{proof}


\subsection{Topological convergence}\label{subsec: topological convergence}

The following definition provides another way to think of convergence of Thurston maps.
\begin{definition}\label{def: topological convergence}
    Let $f_n\colon (\R^2, A) \righttoleftarrow, n \in \N$ be a sequence of Thurston maps. We say that $(f_n)$ \textit{converges topologically} to a Thurston map $f\colon (\R^2, A) \righttoleftarrow$ if for every bounded set $D \subset \R^2$ and for all sufficiently large $n$, the inequality $f_n|D = f|D$ holds.
\end{definition}
Unlike for  combinatorial convergence, note that the limit in Definition \ref{def: topological convergence} is uniquely defined. 
The following proposition shows the relationship between Definitions \ref{def: combinatorial convergence} and~\ref{def: topological convergence}.

\begin{proposition}\label{prop: equivalence of convergences}
    Let $f_n\colon (\R^2, A) \righttoleftarrow, n \in \N$ be a sequence of Thurston maps. Then $(f_n)$ converges combinatorially to a Thurston map $f\colon (\R^2, A) \righttoleftarrow$ if and only if there exists a sequence of Thurston maps $\widehat{f}_n \colon (\R^2, A)\righttoleftarrow$ converging topologically to $f$, where $\widehat{f}_n$ is isotopic rel.\ $A$ to $f_n$.
\end{proposition}

In other words, the two notions of convergence - combinatorial and topological, are, in a certain sense, equivalent.
Before we prove Proposition \ref{prop: equivalence of convergences}, we obtain the following result of a non-dynamical nature.

\begin{proposition}\label{prop: partial lifting}
    Suppose that $f_n\colon \R^2 \to \R^2$, $n \in \N$ and $f\colon \R^2 \to \R^2$ are topologically holomorphic maps of finite type such that $S_{f_n} \subset A$ and $S_f \subset A$ for every $n \in \N$ and some finite set $A$. Let $b, b_n \in \R^2$ and $t \in \R^2 \setminus A$ be points such that $f_n(b_n) = f(b) = t$ for every~$n~\in~\N$. 
    
    Further assume that for every loop $\gamma \subset \R^2 \setminus A$ based at $t$, for all sufficiently large $n$, $\gamma$ lifts under $f_n$ to a loop based at $b_n$ if and only if it lifts under $f$ to a loop based at $b$. In other words, 
    $$
        \lim_{n \to \infty}(f_n)_* \pi_1(\R^2 \setminus f_n^{-1}(A), b_n) = f_*\pi_1(\R^2\setminus f^{-1}(A), b).
    $$   
    Then, given any bounded set $D \subset \R^2$ containing $b$, for all sufficiently large $n$, there exists a continuous injective map $\varphi_n \colon D \to \R^2$ such that $f|D = f_n \circ \varphi_n$ and $\varphi_n(b) = b_n$.   
\end{proposition}

\begin{proof}
   Without loss of generality we can assume that $D$ is an open Jordan region such that $b \in D$ and $\partial D \cap f^{-1}(A) = \emptyset$. Choose a \textit{spider} $S$ that consists of $|A|$ disjoint continuous simple curves $L_a \colon [0, + \infty) \to \R^2$, $a \in A$ such that $L_a$ joins $a$ with $\infty$, i.e., $L_a(0) = a$ and $\lim_{t \to +\infty} L_a(t) = \infty$ for every $a \in A$.
   



    \begin{claim1}
        Let $K \subset \R^2$ be a bounded set, then there are only finitely many connected components of $\R^2 \setminus f^{-1}(S)$ intersecting $K$.
    \end{claim1}

    \begin{subproof}[Proof of Claim 1]
        Suppose that there exist distinct connected components $E_1, E_2, \dots, E_n, \dots$ of $\R^2 \setminus f^{-1}(S)$ intersecting $K$. Pick an arbitrary point $x_n \in E_n \cap K$ for each $n \in \N$. We may assume without loss of generality that $(x_n)$ converges to $x \in \overline{K}$. Thus, any neighbourhood of $x$ intersects infinitely many connected components of $\R^2 \setminus f^{-1}(S)$, which leads to a contradiction since $f$ is a topologically holomorphic map.
    \end{subproof}



    \begin{claim2}
        There exists a bounded domain $D' \subset \R^2$ such that $D \subset D'$ and for every connected component $F$ of $\R^2 \setminus f^{-1}(S)$, the set $D' \cap F$ is connected (possibly empty).
    \end{claim2}

    \begin{subproof}[Proof of Claim 2]
        We will only give an outline of the proof, leaving some straightforward details to the reader. For each $a \in A^* := A \cup \{\infty\}$ we choose a set $U_a$ such that
        \begin{enumerate}
            \item \label{it: closed balls} $U_a$ is a closed Jordan region containing $a$ for each $a \in A$ and $U_\infty = \R^2 \setminus V_\infty$, where $V_\infty$ is an open Jordan region containing $A$; 


            \item \label{it: disjoint} $U_{a_1} \cap U_{a_2} = \emptyset$ for distinct $a_1, a_2 \in A^*$;

            \item \label{it: connected intersection} $U_a \cap S$ is connected for each $a \in A$ and $S \cup \bigcup_{a \in A^*} U_a$ is simply connected;



            \item \label{it: well placed} for each $a \in A^*$, every connected component of $f^{-1}(U_a)$ is either contained in $D$ or disjoint from it.


        \end{enumerate}

        The last condition can be satisfied if we take $U_a, a \in A^*$ disjoint from $f(\partial D) \subset \R^2 \setminus A$.

        



        Let $U := \bigcup_{a \in A^*}U_a$ and $W_F := F \setminus f^{-1}(U)$ for each connected component $F$ of $\R^2 \setminus f^{-1}(S)$. It is easy to see that $F$ can be written as the following disjoint union
        \begin{align*}
             F = F\setminus f^{-1}(U) \sqcup (F \cap f^{-1}(U)) = W_F \sqcup \bigsqcup_{a \in A^*} (F \cap f^{-1}(U_a)).
        \end{align*}
       By Proposition \ref{prop: preimages with infinity}, the map $f|F\colon F \to \R^2 \setminus S$ is a homeomorphism. Therefore, conditions~(\ref{it: closed balls}) and (\ref{it: connected intersection}) imply that the set $F \cap f^{-1}(U_a) = (f|F)^{-1}(U_a \setminus S)$ is connected for each $a \in A$, so it is either contained in $D$ or disjoint from it by condition (\ref{it: well placed}). Thus,
\begin{align*}
            (D \cup W_F) \cap F = W_F \sqcup \bigsqcup_{a \in A_F} (F \cap f^{-1}(U_a)) = (f|F)^{-1}\left(\R^2 \setminus \left(S \cup \bigcup_{a \in A^* \setminus A_F}U_a\right)\right),
\end{align*}       
where $A_F := \{a \in A^*: F \cap f^{-1}(U_a) \subset D\}$. Due to condition (\ref{it: connected intersection}), the set $\R^2 \setminus (S \cup \bigcup_{a \in A^* \setminus A_F}U_a)$ is a domain. Therefore, from Proposition \ref{prop: preimages with infinity}, it follows that the set $(D \cup W_F) \cap F$ is connected.




       Thus, $D' := D \cup \bigcup_{F: F \cap D \neq \emptyset} W_F$ is an open set such that $F \cap D'$ is connected for each connected component $F$ of $\R^2 \setminus f^{-1}(S)$. The set $D'$ is connected since $W_F \cap D \neq \emptyset$ if and only if $F \cap D \neq \emptyset$. At the same time, $D'$ is bounded since each $W_F$ is bounded and there are finitely many connected components of $\R^2 \setminus f^{-1}(S)$ intersecting $D$ due to Claim 1.
        
        
    \end{subproof}

    Due to the claims above, there exists a positive integer $m$ such that any $x \in D' \setminus f^{-1}(A)$ can be joined with the point $b$ by a simple path contained in $D' \setminus f^{-1}(A)$ and intersecting $f^{-1}(S)$ at most $m$ times. Denote by $P_{S, k}$ the set of elements of $\pi_1(\R^2 \setminus A, t)$ that can be represented by a loop $\alpha\colon \I \to  \R^2 \setminus A$ based at $t$ such that $|\alpha^{-1}(S)| \leq k$.  It is clear that $P_{S, k}$ is a finite subset of $\pi_1(\R^2 \setminus A, t)$.
    

    Now we construct a continuous map $\varphi_n\colon D' \to \R^2$ and prove its injectivity as required.
    Let~$D''$ be an open Jordan region containing $D'$. Since the group $f_*\pi_1(D'' \setminus f^{-1}(A), b)$ is finitely generated, by our initial assumptions, it is a subgroup of $(f_n)_* \pi_1(\R^2 \setminus f_n^{-1}(A), b_n)$ for all $n$ sufficiently large. Therefore, there exists a continuous map $\varphi_n \colon D'' \setminus f^{-1}(A) \to \R^2 \setminus f_n^{-1}(A)$ such that $f = f_n \circ \varphi_n$ on $D'' \setminus f^{-1}(A)$ and $\varphi_n(b) = b_n$. 

    We prove that the map $\varphi_n$ is injective on $D'\setminus f^{-1}(A)$.
    Suppose that it is not true for some~$n$, i.e., there exist two distinct points $x_n, y_n \in D' \setminus f^{-1}(A)$ such that $\varphi_n(x_n) = \varphi_n(y_n)$. We join~$b$ with $x_n$ and $y_n$ by two simple paths $\alpha_n$ and $\beta_n$ in $D' \setminus f^{-1}(A)$, respectively. Additionally, assume that both paths intersect $f^{-1}(S)$ at most $m$ times.
    It is clear that the loop $\gamma_n := f \circ \alpha_n\cdot \overline{f \circ \beta_n}$ based at $t$ does not lift under $f$ to a loop based at $b$ but it lifts under $f_n$ to the loop $\varphi_n \circ \alpha_n \cdot \overline{\varphi_n \circ \beta_n}$ based at $b_n$. 
Thus,
    \begin{align*}
        [\gamma_n] \in G_n:= P_{S, 2m} \cap \Big((f_n)_*\pi_1(\R^2 \setminus f_n^{-1}(A), b_n) \hspace{5pt}\setminus \hspace{5pt} f_* \pi(\R^2 \setminus f^{-1}(A), b)\Big).
    \end{align*}
    
     Since  $P_{S, 2m}$  is a finite set, if $\varphi_n$ fails to be injective for infinitely many $n$, then there exists $[\gamma] \in \pi_1(\R^2 \setminus A, t)$ such that $[\gamma] \in G_{n}$ for infinitely many $n$. This, however, is not possible: since $[\gamma] \not \in f_*\pi_1(\R^2\setminus f^{-1}(A), b)$, then for all $n$ sufficiently large,  $[\gamma]\not\in (f_n)_*\pi_1(\R^2 \setminus f_n^{-1}(A),t)$.
    
    Finally, since $f^{-1}(A)$ and $f_n^{-1}(A)$ are discrete subsets of $\R^2$, we can extend $\varphi_n$ to a continuous injective map defined on $D'$ satisfying $f|D' = f_n \circ \varphi_n$ and $\varphi_n(b) = b_n$ for all $n$ large enough.
    \end{proof}
\begin{proof}[Proof of Proposition \ref{prop: equivalence of convergences}]

If the sequence $(\widehat{f}_n)$ converges topologically to $f$, then it does so combinatorially with respect to $t, b, (b_n), p, (p_n)$, where $t \in \R^2 \setminus A$ and $b \in f^{-1}(t)$ are arbitrary, $b_n := b$, and $p$, $p_n$ are the constant loops based at $b$ for every $n \in \N$. Hence, Proposition~\ref{prop: independence from isotopy} implies that the sequence $(f_n)$ converges combinatorially to $f$.

Now suppose that the sequence $(f_n)$ converges combinatorially to $f$ with respect to some choice of $t, b, (b_n), p$ and $(p_n)$. Let $(D_m)$ be an exhaustion of $\R^2$ by closed Jordan regions containing $A$ and $b$ in their interiors for each $m \in \N$. Then for $n$ sufficiently large, there exists a maximal index $m = m(n)$ such that
\begin{enumerate}
    \item $f|D_m = f_n \circ \varphi_n$ for some continuous injective map $\varphi_n\colon D_m \to \R^2$ that satisfies $\varphi_n(b) = b_n$;

    \item for every loop $\gamma \subset \R^2 \setminus A$ based at $t$ such that $[\gamma] \in f_*\pi_1(D_m \setminus f^{-1}(A), b)$, loops $\gammal(f, b)$ and $p_n \cdot \gammal(f_n, b_n) \cdot \overline{p_n}$ are homotopic rel.\ $A$.
\end{enumerate}

These conditions hold for sufficiently large $n$ because of Definition \ref{def: combinatorial convergence}, Proposition \ref{prop: partial lifting} and the fact that $f_*\pi_1(D_m \setminus f^{-1}(A), b)$ is finitely generated. Moreover, it is easy to show that $m(n)$ converges to infinity as $n$ tends to infinity.

Now we can extend $\varphi_n$ to an orientation-preserving homeomorphism $\widehat{\varphi}_n$ of $\R^2$. In fact, by Alexander's trick, all such extensions of $\varphi_n$ are isotopic rel.\ $D_m$ to each other. Clearly, if $\gamma$ is a loop in $D_m \setminus f^{-1}(A)$ based at $b$, then $\varphi_n \circ \gamma$ is a $f_n$-lift (based at $b_n$) of $f \circ \gamma$.  In particular, $\gamma$ is homotopic rel.\ $A$ to $p_n \cdot (\widehat{\varphi}_n\circ \gamma) \cdot \overline{p_n}$ for every loop $\gamma \subset \R^2 \setminus A$ based at $b$. Thus, the homeomorphism $\widehat{\varphi}_n$ is isotopic rel.\ $A$ to $\id_{\R^2}$ by Lemma~\ref{lemm: useful prop}.



Now consider the sequence $(\widehat{f}_n)$ of Thurston maps, where $\widehat{f}_n := f_n \circ \widehat{\varphi}_n$ for all $n \in \N$. Clearly, $(\widehat{f}_n)$ converges topologically to $f$.
\end{proof}

\subsection{Polynomial approximations}\label{subsec: polynomial approximations}

In Section \ref{subsec: admissible quadruples}, we show how to define topologically holomorphic maps of finite type using the language of admissible quadruples. Now we apply this framework to study Thurston maps, as this point of view is helpful while working with the notion of combinatorial convergence. In particular, we establish criteria for combinatorial convergence of a given sequence of Thurston maps (Proposition \ref{prop: quadruples and convergence}) and show that every Thurston map can be approximated combinatorially by a sequence of polynomial Thurston maps (Proposition \ref{prop: polynomial approximations}).

\begin{definition} \label{def: dynamically admissible quadruple}
    An admissible quadruple $\Delta = (A, \Rose, \Gamma, \Phi)$ is called \textit{dynamically admissible} if $\Delta$ is parabolic, $\Gamma \cap A = \emptyset$, and every face $F$ of the graph $\Gamma$ contains at most one point of~$A$ if $F$ is bounded and no points of $A$, otherwise.
    
    
\end{definition}

Proposition \ref{prop: quadruple by map} implies that $\Delta(A, \Rose, f)$ is dynamically admissible if $f\colon (\R^2, A)\righttoleftarrow$ is a Thurston map. Also, analogous to Proposition \ref{prop: map by quadruple}, we can establish the following statement.

\begin{proposition}\label{prop: Thurston map by quadruple}
    Let $\Delta = (A, \Rose, \Gamma, \Phi)$ be a dynamically admissible quadruple. Then there exists a Thurston map $f\colon (\R^2, A) \righttoleftarrow$ such that $\Delta(A, \Rose, f) = \Delta$.
\end{proposition}

In a similar way we define a \textit{dynamical equivalence} relation on the set of all dynamically admissible quadruples with a fixed marked set. We say that two dynamically admissible quadruples $\Delta_1 = (A, \Rose_1, \Gamma_1, \Phi_1)$ and $\Delta_2 = (A, \Rose_2, \Gamma_2, \Phi_2)$ are dynamically equivalent if there exist $\theta, \varphi \in \Homeo_0^+(\R^2, A)$ such that 
\begin{enumerate}
    \item $\theta$ is an isomorphism between $\Rose_1$ and $\Rose_2$;

    \item $\varphi$ is an isomorphism between $\Gamma_1$ and $\Gamma_2$;

    \item $\theta \circ \Phi_1 = \Phi_2 \circ \varphi$.
\end{enumerate}

\noindent Finally, the following observation provides a dynamical analog of Proposition \ref{prop: equivalence of quadruples and maps}.

\begin{proposition}\label{prop: equivalent th maps and quadruples}
    Let $f_1 \colon (\R^2, A) \righttoleftarrow$ and $f_2 \colon (\R^2, A) \righttoleftarrow$ be Thurston maps, and $\Rose_1$, $\Rose_2$ be rose graphs surrounding $A$ such that $\Rose_1$ is isotopic rel.\ $A$ to $\Rose_2$.  Then $\Delta(A, \Rose_1, f_1)$ and $\Delta(A, \Rose_2, f_2)$ are dynamically equivalent if and only if $f_1$ and $f_2$ are isotopic rel.\ $A$.
\end{proposition}
\begin{figure}
    \centering
    \includegraphics[scale=0.5]{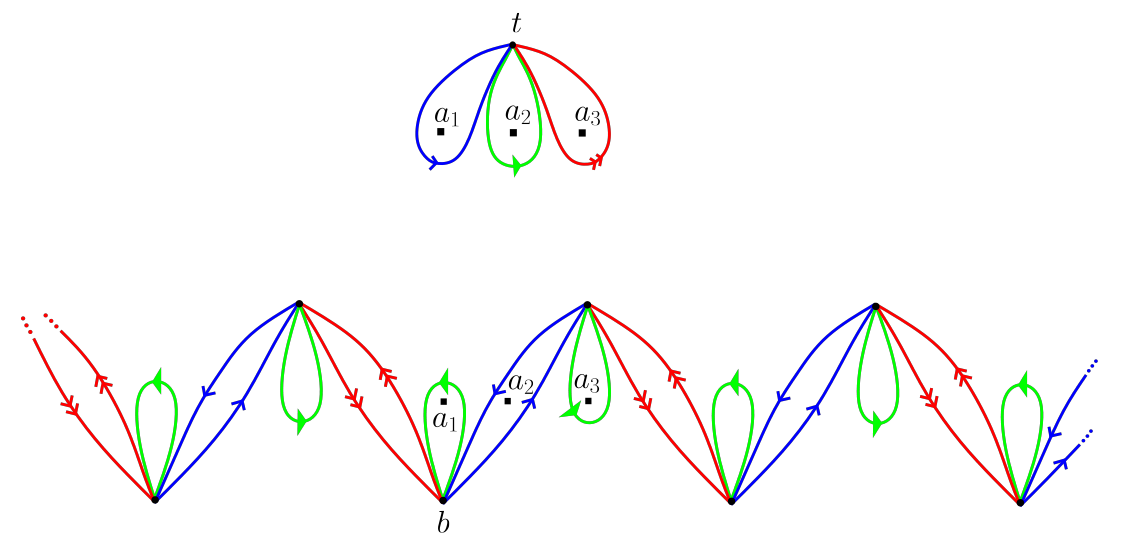}
    \label{fig:cosin_dyn}
\caption{Dynamically admissible quadruple realized by the psf entire map $G_1(z) = \pi \cos (z) / 2$, where $P_{G_1} = \{a_1, a_2, a_3\} = \{-\pi/2, 0, \pi/2\}$.}
\label{fig: dyn_pair for cosine}
\end{figure}
\begin{figure}
    \centering
    \includegraphics[scale=0.5]{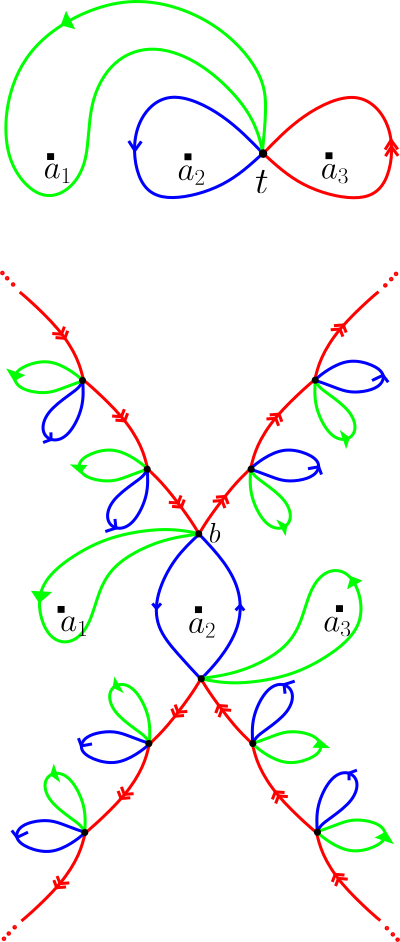}
\caption{Dynamically admissible quadruple realized by the psf entire map $G_2(z) = \sqrt{\ln 2}\big(1-\exp(z^2)\big)$, where\ $P_{G_2} = \{a_1, a_2, a_3\} = \{-\sqrt{\ln 2}, 0, \sqrt{\ln 2}\}$.}
\label{fig: dyn_pair for multi_error}
\end{figure}

\begin{proof}

  First suppose that $\Delta(A, \Rose_1, f_1)$ and $\Delta(A, \Rose_2, f_2)$ are dynamically equivalent via $\theta, \varphi \in \Homeo^+_0(\R^2, A)$. As in the proof of Proposition \ref{prop: equivalence of quadruples and maps}, we can assume that $\Rose_1 = \Rose_2 = \Rose$ and $\theta = \varphi = \id_{\R^2}$. The rest easily follows from Proposition \ref{prop: combinatorial data defines map} and the discussion of Section \ref{subsec: admissible quadruples}.

  Conversely, suppose that there exists $\psi \in \Homeo^+_0(\R^2)$ such that $f_1 = f_2 \circ \psi$. Choose $\theta \in \Homeo_0^+(\R^2, A)$ with $\theta(\Rose_1) = \Rose_2$. By Proposition \ref{prop: isotopy lifting property}, there exists $\varphi \in \Homeo^+_0(\R^2, A)$ such that $\theta \circ f_1 = f_2 \circ \varphi$. It directly implies that $\Delta(A,\Rose_1,f_1)$ and $\Delta(A,\Rose_2,f_2)$ are dynamically equivalent.
\end{proof} 


Given dynamically admissible quadruple $\Delta = (A, \Rose, \Gamma, \Phi)$, we say that a Thurston map $f \colon (\R^2, A) \righttoleftarrow$ \textit{realizes} $\Delta$ (or  equivalently, $\Delta$ \textit{defines} $f$), if $\Delta(A, \Rose, f)$ is dynamically equivalent to $\Delta$. Propositions \ref{prop: Thurston map by quadruple} and \ref{prop: equivalent th maps and quadruples} imply that every dynamically admissible $\Delta$ with a marked set $A$ defines a Thurston map $f\colon (\R^2, A) \righttoleftarrow$, unique up to isotopy rel.\ $A$.
\begin{example}\label{ex: dynamically admissible quadruples}
       Let $G_1 \colon (\C,P_{G_1}) \righttoleftarrow$ and $G_2 \colon (\C,P_{G_2}) \righttoleftarrow$ be the postsingularly finite entire maps defined in Example \ref{ex: portraits}. The graphs in Figures \ref{fig: dyn_pair for cosine} and \ref{fig: dyn_pair for multi_error} describe dynamically admissible quadruples realized by $G_1$ and $G_2$ respectively (compare with Figures \ref{fig: nondyn_quadruple for cosine} and \ref{fig: nondyn_quadruple or eg2} referenced in Example  \ref{ex: non dyn quadruples}). The solid black squares in these figures represent the postsingular values of $G_1$ and $G_2$, respectively. 
\end{example}

We shall use the framework of dynamically admissible quadruples to construct Thurston maps with required combinatorial convergence properties.


\begin{proposition}\label{prop: quadruples and convergence}
    Let $f_n \colon (\R^2, A) \righttoleftarrow, n \in \N$ and $f \colon (\R^2, A) \righttoleftarrow$ be Thurston maps, and~$\Rose$ be a rose graph that surrounds $A$.  
    Then the sequence $(f_n)$ converges combinatorially to~$f$ if and only if for every finite subgraph $K$ of~$f^{-1}(\Rose)$ and all sufficiently large~$n$, there exists
    a homeomorphism $\varphi_{K, n} \in \Homeo_0^+(\R^2, A)$ such that $\varphi_{K, n}(K)$ is a subgraph of ~$f_n^{-1}(\Rose)$ and $\Phi_{f,\Rose}|K = \Phi_{f_n,\Rose} \circ~\varphi_{K, n}|K$.
\end{proposition}

\begin{proof}
    Sufficiency easily follows from Definition \ref{def: combinatorial convergence}, Proposition \ref{prop: independence from isotopy}, and the discussion of Section \ref{subsec: admissible quadruples}. Necessity can be obtained by applying Proposition \ref{prop: equivalence of convergences}. 
\end{proof}

\begin{proposition} \label{prop: polynomial approximations}
    Let $f\colon (\R^2, A) \righttoleftarrow$ be an arbitrary Thurston map. Then there exists a sequence $f_n \colon (\R^2, A) \righttoleftarrow, n \in \N$ of polynomial Thurston maps converging combinatorially~to~$f$. 
    
\end{proposition}

\begin{proof}

    Due to Proposition \ref{prop: small postsingular set} and Remark \ref{rem: quadruples with one marked point}, the case when $|A| = 1$ is trivial and, therefore, we can assume that $|A| \geq 2$. Let us choose a rose graph $\Rose$ surrounding the set~$A$ and consider the dynamically admissible quadruple $\Delta(A, \Rose, f)  = (A, \Rose, \Gamma, \Phi)$.
    Next, we choose an arbitrary exhaustion of $\Gamma$ by finite connected subgraphs $K_n = (V_n, E_n)$. We shall construct a sequence of eventually dynamically admissible quadruples $\Delta_n = (A, \Rose, \Gamma_n, \Phi_n)$, where $\Gamma_n$ is obtained from $K_n$ by adding several new edges, and the maps $\Phi_n$ and $\Phi$ coincide on $K_n$. We describe this more precisely by defining $\Gamma_n$ and $\Phi_n$ algorithmically. For each $n \in \N$, initialize $\Gamma_n$ as  $K_n$, and $\Phi_n$ as $\Phi|K_n$. Let $F$ be an arbitrary face of $\Gamma$ labelled by a bounded face $P$ of $\Rose$ with the property that $\partial F$ intersects $K_n$ but it is not a proper subset of $K_n$.
\begin{figure}[h]
    \centering
    \includegraphics[scale=0.5]{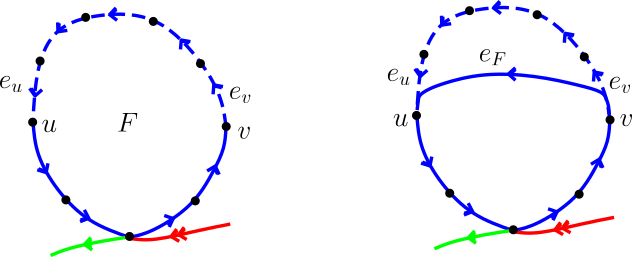}
    \caption{Constructing the directed edge $e_F$. The diagram on the left represents the face $F$ and the graph $K_n$, and the diagram on the right demonstrates the newly added edge $e_F$. Black dots and colored (solid and dashed) arcs represent the vertices and directed edges of $\Gamma$,  respectively. The solid blue arcs are edges of $K_n$, while the dashed ones are in $\Gamma \setminus K_n$. }
    \label{fig:new_edge}
\end{figure}

    \begin{claim}
        $C:=\partial F \cap K_n$ is a directed finite chain.
    \end{claim}
    \begin{subproof}[Proof of Claim]
        First we consider the case when $F$ is bounded: here, $\partial F$ is a counterclockwise directed cycle by Remark \ref{rem: boundary}. It suffices to show that $C$ (or equivalently $\partial F \setminus C$) is connected. Supposing the contrary, let $e_1$ and $e_2$ be two edges in disjoint components of~$\partial F \setminus C$. 
        It is easy to see that each edge $e \in E(\Gamma)$ is a boundary of exactly two faces and one of them is always unbounded. Therefore, for $i=1,2$, there exists a continuous curve $L_i\colon [0, +\infty) \to~\R^2$ that joins an interior point $x_i := L_i(0)$ of $e_i$ with $\infty$ (i.e., $\lim_{t \to +\infty} L_i(t) = \infty$) and intersects~$\Gamma$ only at the point $x_i$. Let $y_1$ and $y_2$ be vertices of $C$ lying in disjoint components of $\partial F \setminus \{\inter(e_1), \inter(e_2)\}$. Any path in~$\R^2$ joining $y_1$ to $y_2$ has to intersect $F \cup \{\inter(e_1), \inter(e_2)\}$ or $L_1 \cup L_2$. In particular, there is no path in $K_n$ joining $y_1$ to $y_2$. This forms a contradiction since $K_n$ is connected. The case when $F$ is unbounded is analogous.
    \end{subproof}
    

Let $u$ and $v$ be the endpoints of $C$ (it is possible that~$u = v$). Then there exist unique edges $e_u, e_v \in E(\partial F \setminus K_n)$ such that $e_u$ is incident to $u$, and $e_v$ is incident to $v$ (again, $e_u$ and $e_v$ might coincide). 
We shall construct an edge $e_F$ with endpoints at $v$ and $u$ (see Figure~\ref{fig:new_edge}) with $e_F \subset e_u \cup e_v \cup F$, so that $e_F$ coincides with the edges $e_u$ and $e_v$ in small neighborhoods of $u$ and $v$, respectively. In particular, $\inter(e_F)$ does not intersect $K_n$. We add $e_F$ to $\Gamma_n$ and define $\Phi_n|e_F$ so that $\Phi_n$ and $\Phi$ coincide on $e_F \cap e_u$ and $e_F \cap e_v$. Then we repeat the above procedure for all faces $F \in F(\Gamma)$ satisfying previously listed properties.



Now consider the sequence of constructed quadruples $\Delta_n = (A, \Rose, \Gamma_n, \Phi_n)$. Since $(K_n)$ is an exhaustion of $\Gamma$, there exists $N \in \N$ such that for each $n\geq N$, if $F \in F(\Gamma)$ is bounded and contains a point of $A$, then $\partial F \subset K_n$. For each $n\geq N$, it is straightforward to check that the conditions of   Definitions \ref{def: admissible quadruple} and \ref{def: dynamically admissible quadruple} are satisfied. In particular, $\Delta_n$ is parabolic since~$\Gamma_n$ is a finite graph. By Proposition \ref{prop: Thurston map by quadruple}, we can construct a polynomial Thurston map $f_n \colon (\R^2, A) \righttoleftarrow$ such that $\Delta(A, \Rose, f_n) = \Delta_n$ for each $n \geq N$. Thus, the sequence $(f_n)$ converges combinatorially to $f$ due to Proposition \ref{prop: quadruples and convergence}.
\end{proof}



\begin{remark}\label{rem: conv below}
Suppose that we are in the setting of the proof of Proposition \ref{prop: polynomial approximations}. Let $F$ be a face of $\Gamma_n$ labelled by a bounded face $P$ of $\mathcal{R}$ (with respect to the quadruple $\Delta_n$) for some $n \geq N$. Then there exists a unique face $F' \in F(\Gamma)$ having the same label $P$ (with respect to the quadruple $\Delta$) such that $F\subset F'$. Moreover, exactly one of the following is true: 

\begin{enumerate}
    \item $\deg(f_n|F) = \deg(f_n|F')$. Then Proposition \ref{prop: quadruple by map} implies that $F$ contains a (unique) critical point $z_n$ of $f_n$ if and only if $F'$ contains a (unique) critical point $z$ of $f$ and, moreover, $\deg(f, z) = \deg(f_n, z_n)$.

    \item $\deg(f_n|F) < \deg(f_n|F')$. If $F$ contains a critical point of $f_n$, then it is unique in $F$. In this case $F'$ is either an asymptotic tract of $f$ or it contains a unique critical point~$z$ of $f$ and, moreover, $\deg(f, z) > \deg(f_n, z_n)$.
\end{enumerate}


    


\end{remark}

With this discussion and the following example, we show that dynamically admissible quadruples provide a convenient way for constructing ``combinatorial" approximations and thinking about combinatorial convergence.
\begin{example}\label{ex: pcf approx}

Consider the postsingularly finite entire map $G_2$ from Example~\ref{ex: portraits}, realizing the dynamically admissible quadruple $\Delta_{G_2} = \Delta(P_{G_2}, \Rose, G_2) = (P_{G_2}, \Rose, \Gamma, \Phi)$ depicted in Figure~\ref{fig: dyn_pair for multi_error} along with a chosen point $b \in V(\Gamma)$. Define the graph $K_n \subset \Gamma$ to be a collection of all vertices and edges of $\Gamma$ accessible from $b$ via a path in $\Gamma$ intersecting interiors of at most $n$ edges. It is clear that the sequence $(K_n)$ is an exhaustion of $\Gamma$ by finite connected graphs. Starting with $\Delta_{G_2}$ and $(K_n)$ and applying the construction from the proof of Proposition~\ref{prop: polynomial approximations}, we obtain a sequence of polynomial Thurston maps $f_n \colon (\R^2, P_{G_2}) \righttoleftarrow, n \in \N$ converging combinatorially to $G_2\colon (\C, P_{G_2})\righttoleftarrow$, with each $f_n$ defined by a dynamically admissible quadruple $\Delta_n~=~(P_{G_2}, \Rose, \Gamma_n, \Phi_n)$. Figure \ref{fig: poly approx for multierror} illustrates 
the graphs $\Gamma_n$ and the maps~$\Phi_n$ for $n = 1, 2, 3$ (from left to right). As usual, $\Phi_n\colon \Gamma_n \to \Rose$ maps each edge of $\Gamma_n$ to the unique edge of $\Rose$ of the same color, and the set $P_{G_2}$ is represented by solid black squares.

However, combinatorial  approximations need not to be polynomial. Consider the sequence of infinite graphs $(\widehat{\Gamma}_n)$ depicted in Figure~\ref{fig: non poly approx}. For each $n$, we can similarly define a map $\widehat{\Phi}_n: \widehat{\Gamma}_n \to~\mathcal{R}$ such that $\widehat{\Delta}_n = (P_{G_2}, \mathcal{R}, \widehat{\Gamma}_n, \widehat{\Phi}_n)$ is a dynamically admissible quadruple (parabolicity of $\widehat{\Delta}_n$ follows from \cite[Theorem 4.1]{weiwei}). Then $\widehat{\Delta}_n$ determines a transcendental Thurston map $\widehat{f}_n$ such that the sequence $(\widehat{f}_n)$ converges combinatorially to $G_2: (\C, P_{G_2}) \righttoleftarrow$ as $n$ tends to $\infty$.

There is also no canonical choice for a sequence of polynomial Thurston maps converging combinatorially to a given transcendental Thurston map. For instance, the sequence of polynomial Thurston maps illustrated in Figure \ref{fig: poly approx for cosine} (with respect to the rose graph showed at the top of Figure \ref{fig: dyn_pair for cosine}) converges combinatorially the map $G_1\colon(\C, P_{G_1})\righttoleftarrow$ from Example~\ref{ex: portraits}, which realizes dynamically admissible quadruple $\Delta_{G_1}$ as in Figure~\ref{fig: dyn_pair for cosine}. However, the property of Remark \ref{rem: conv below} cannot be satisfied for these combinatorial approximations.



\end{example}

\begin{figure}[t]

\begin{subfigure}[b]{0.25\textwidth}
    \centering

    \includegraphics[scale=0.4]{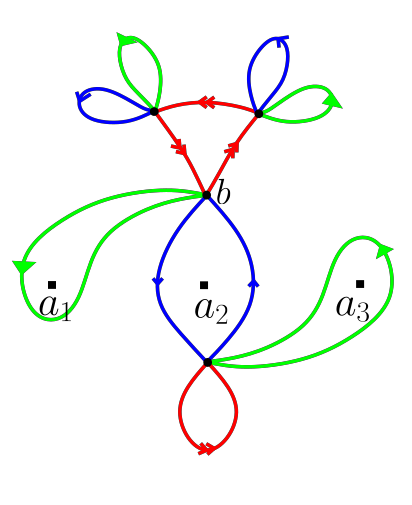}
\vspace{1.2cm}
\end{subfigure} \hspace{20pt}
\begin{subfigure}[b]{0.25\textwidth}
    \centering

    \includegraphics[scale=0.4]{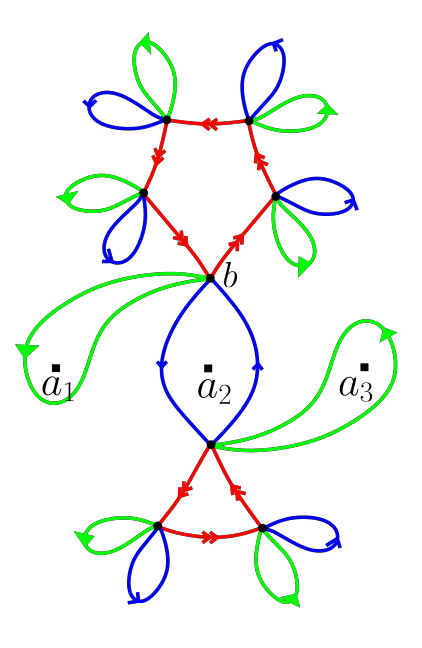}
\vspace{0.5cm}
\end{subfigure} \hspace{20pt}
\begin{subfigure}[b]{0.25\textwidth}
    \centering

    \includegraphics[scale=0.4]{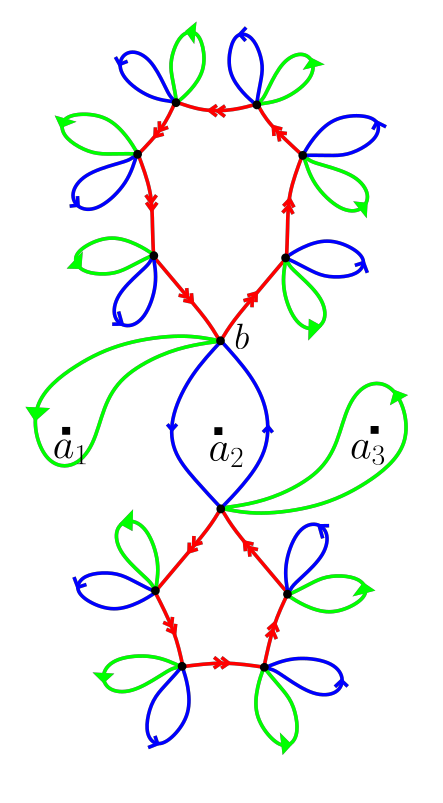}

\end{subfigure} 

\caption{Sequence of graphs that define polynomial Thurston maps converging combinatorially to $G_2\colon (\C, P_{G_2}) \righttoleftarrow$, where $G_2(z) = \sqrt{\ln 2}(1-\exp(z^2))$ and $P_{G_2} = \{a_1, a_2, a_3\} = \{-\sqrt{\ln 2}, 0, \sqrt{\ln 2}\}$.}
\label{fig: poly approx for multierror}  
\end{figure}

\begin{figure}

\begin{subfigure}[b]{0.25\textwidth}
    \includegraphics[scale=0.35]{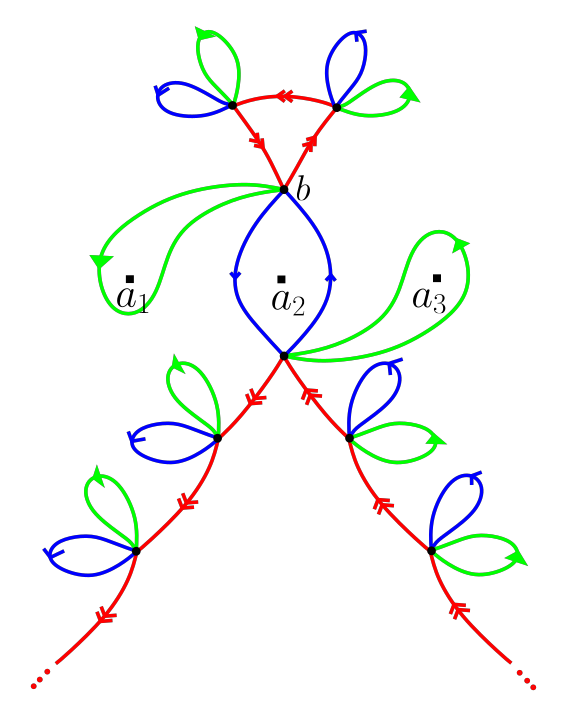} 
\end{subfigure}

\begin{subfigure}[b]{0.25\textwidth}
    \includegraphics[scale=0.35]{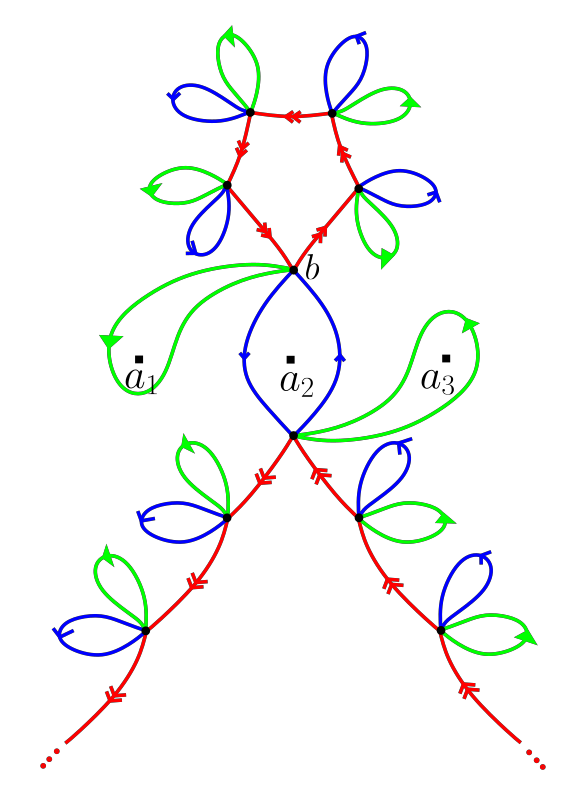} 
\end{subfigure}

\begin{subfigure}[b]{0.25\textwidth}
    \includegraphics[scale=0.35]{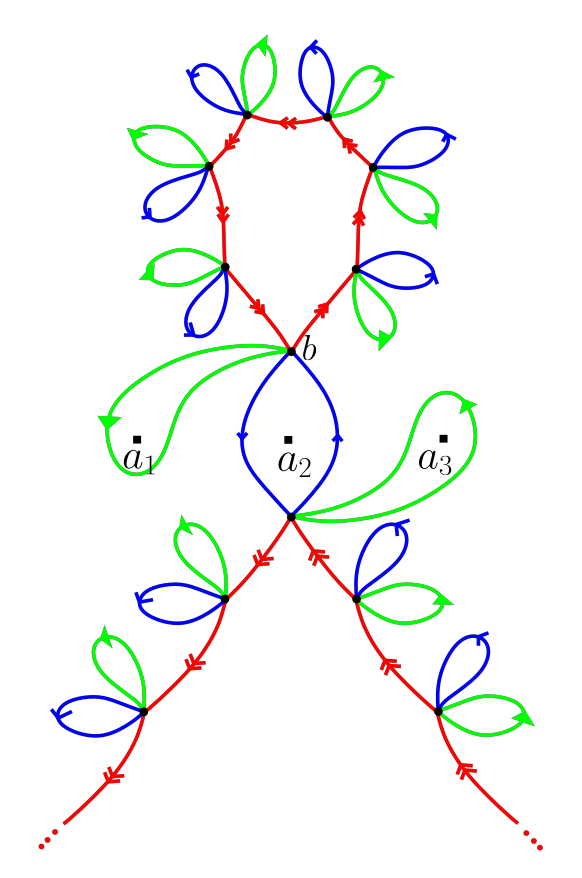} 
\end{subfigure}

\caption{Sequence of graphs that define transcendental Thurston maps converging combinatorially to $G_2\colon (\C, P_{G_2}) \righttoleftarrow$, where $G_2(z) = \sqrt{\ln 2}(1-\exp(z^2))$ and $P_{G_2} = \{a_1, a_2, a_3\} = \{-\sqrt{\ln 2}, 0, \sqrt{\ln 2}\}$.}
    \label{fig: non poly approx}    
\end{figure}

\begin{figure}[t]
    \centering
    \includegraphics[scale=0.5]{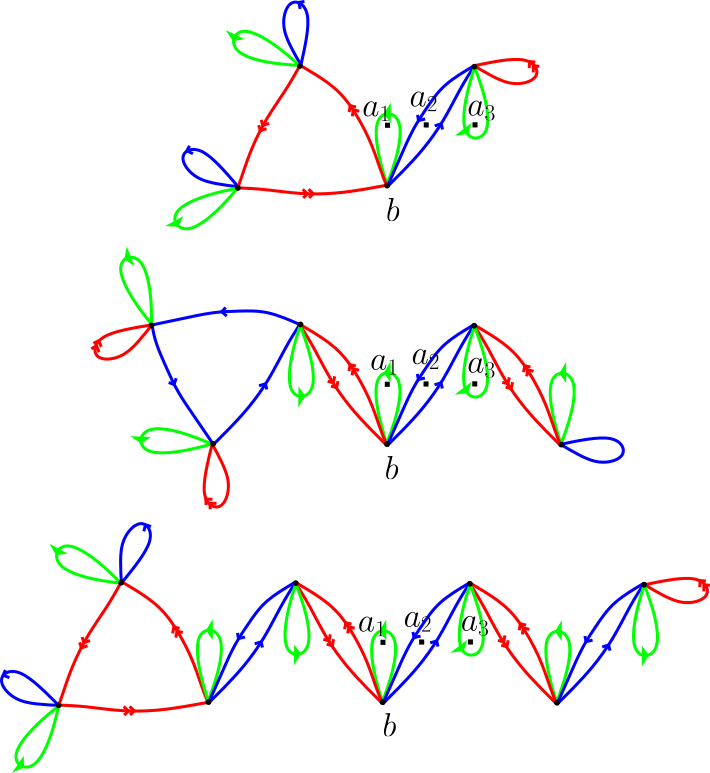}
    \caption{Sequence of graphs that define polynomial Thurston maps converging combinatorially to $G_1\colon (\C, P_{G_1}) \righttoleftarrow$, where $G_1(z) = \pi \cos(z) / 2$ and $P_{G_1} = \{a_1, a_2, a_3\} = \{0, -\pi/2, \pi/2\}$.}
    \label{fig: poly approx for cosine}
\end{figure}

\subsection{Convergence of \texorpdfstring{$\sigma$}{σ}-maps} \label{subsec: convergence of sigma-maps} 

We start this section by introducing the following result on quasiconformal maps \cite[Theorem 1]{krushkal_paper}.

\begin{theorem}\label{thm: qc extension}
    Let $\varphi \colon \D \to \C$ be an injective holomorphic map. Then for every $t \in \D^*$ the map $\varphi_t \colon \D \to \C$, where $\varphi_t(z) := \varphi(tz) / t$ for all $z \in \D$, admits a $K_t$-quasiconformal extension $\widehat{\varphi}_t\colon \C \to \C$, where $K_t = (1 + |t|) / (1 - |t|)$.
\end{theorem}

The following corollary provides a slight reformulation of Theorem \ref{thm: qc extension}.

\begin{corollary}\label{corr: qc extension}
    Let $\varphi \colon \D_R \to \C$ be an injective holomorphic map. Then for every $r, 0 < r < R$, the map $\psi_{r} := \varphi|\D_r  \colon \D_r \to \C$ admits a $K_{r/R}$-quasiconformal extension $\widehat{\psi}_{r}\colon \C \to \C$, where $K_{r / R} = (1 + r / R) / (1 - r / R)$.
\end{corollary}

\begin{proof}
    Applying Theorem \ref{thm: qc extension} with $t = r/R$ to the map $z \mapsto \varphi(Rz)$, where $z \in \D$, we see that the map $z \mapsto \frac{R}{r} \varphi(r z)$ defined on $\D$ admits a $K_{r/R}$-quasiconformal extension $\widehat{\varphi}_{r/R}\colon \C \to \C$. Define  $\widehat{\psi}_{r} \colon \C \to \C$ as follows:
    $$
        \widehat{\psi}_{r} (z) := \frac{r}{R}\widehat{\varphi}_{r/R}\left(\frac{z}{r}\right)
    $$

    Then one can see that $\widehat{\psi}_r|\D_r = \varphi|\D_r = \psi_r$, and $\widehat{\psi}_r$ is $K_{r/R}$-quasiconformal.    
\end{proof}

Theorem \ref{thm: qc extension} and Corollary \ref{corr: qc extension} allow us to establish the following upper bound on the distance between two points of $\Teich(\R^2, A)$ that satisfy certain properties.

\begin{proposition}\label{prop: bound}
   Let $\tau_1 = [\varphi_1], \tau_2 = [\varphi_2]$ be points in $\Teich(\R^2, A)$. Suppose that there exist positive real numbers $r$ and $R$ such that $0 < r < R$, $\varphi_1(A) \subset \D_r$, and $\varphi_2 \circ \varphi_1^{-1}$ is holomorphic on $\D_R$. Then 
   $$
        d(\tau_1, \tau_2) \leq \log \left(\frac{1 + r/R}{1 - r/R}\right).
   $$
\end{proposition}

\begin{proof}
    Let $\varphi := (\varphi_2 \circ \varphi_1^{-1})|\D_R \colon \D_R \to \C$ be a holomorphic injective map. By Corollary~\ref{corr: qc extension}, $\psi_r := \varphi|\D_r$ extends $K_{r/R}$-quasiconformally to a map $\widehat{\psi}_{r}$ on $\C$ \big(with $K_{r/R}  = \frac{1+r/R}{1-r/R}$\big). One can see that $\widehat{\psi}_{r} \circ \varphi_1$ and $\varphi_2$ coincide on $\varphi_1^{-1}(\D_r)$. In particular, these homeomorphisms are isotopic rel.\ $A$ by Alexander's trick.
    Therefore, the inequality
    $d(\tau_1, \tau_2) \leq \log K(\widehat{\psi}_{r})$ is satisfied. Since $\widehat{\psi}_{r}$ is $K_{r/R}$-quasiconformal, we are done.
\end{proof}

Now we are ready to establish Main Theorem \ref{mainthm: main theorem B}.

\begin{proof}[Proof of Main Theorem~\ref{mainthm: main theorem B}]\label{prf: Main Theorem B}
    


    Since $(f_n)$ converges combinatorially to $f$, by Proposition \ref{prop: equivalence of convergences} and Remark \ref{rem: remark on sigma maps}, we can assume without loss of generality that $(f_n)$ converges topologically~to~$f$.

    Given $\tau \in \Teich(\R^2, A)$, we first show that $\sigma_{f_n}(\tau) \to \sigma_f(\tau)$ as $n \to \infty$. Let $\tau = [\varphi]$, $\sigma_f(\tau) = [\psi]$, and $\sigma_{f_n}(\tau) = [\psi_n]$, where $\varphi, \psi, \psi_n \colon \R^2 \to \C$ are orientation-preserving homeomorphisms for every $n \in \N$. By Definition \ref{def: topological convergence}, given any bounded open set $D \subset \R^2$, the equality $f_n|D = f|D$ holds for sufficiently large $n$. In particular, $\psi_n \circ \psi^{-1}$ is holomorphic on $\psi(D)$. Then Proposition~\ref{prop: bound} implies that $d(\sigma_f(\tau), \sigma_{f_n}(\tau)) \to 0$ as $n \to \infty$.

    Now we show that the convergence $\sigma_{f_n} \to \sigma_f$ is in fact locally uniform on $\Teich(\R^2, A)$. Let $K \subset \Teich(\R^2, A)$ be a compact set. Choose arbitrary $\varepsilon > 0$ and cover $K$ by open balls $B_1, B_2, \dots, B_k$ each of radius $\varepsilon / 3$, with $B_j$ centered at $\mu_j \in \Teich(\R^2, A)$ for each $j = 1,2 \dots, k$. As we proved earlier, for sufficiently large $n$, $d(\sigma_{f_n}(\mu_j), \sigma_f(\mu_j)) \leq \varepsilon / 3$ for each $j  = 1,2,\dots, k$. 
   Hence, for all $n$ large enough and any $\mu \in K$, we have    
   $$
        d(\sigma_{f_n}(\mu), \sigma_f(\mu)) \leq d(\sigma_{f_n}(\mu), \sigma_{f_n}(\mu_j)) +  d(\sigma_{f_n}(\mu_j), \sigma_f(\mu_j)) + d(\sigma_f(\mu_j), \sigma_f(\mu))  \leq \varepsilon,
    $$
    where $j$ is chosen so that $d(\mu, \mu_j) \leq \varepsilon / 3$. We also use the fact that all $\sigma$-maps above are 1-Lipschitz (see Proposition \ref{prop: sigma is contracting}).    
\end{proof}

\begin{corollary}
    Let $f\colon (\R^2, A) \righttoleftarrow$ be a transcendental Thurston map. Then there exists a sequence of polynomial Thurston maps $f_n \colon (\R^2, A) \righttoleftarrow, n \in \N$ such that the sequence $(\sigma_{f_n})$ converges locally uniformly on $\Teich(\R^2, A)$ to $\sigma_f$.
\end{corollary}

\begin{proof}
    This follows directly from Proposition \ref{prop: polynomial approximations} and Main Theorem~\ref{mainthm: main theorem B}.
\end{proof}

\begin{corollary} \label{corr: fixed points of approximating maps}
Let $f\colon (\R^2, A) \righttoleftarrow$ be a realized Thurston map and $f_n \colon (\R^2, A) \righttoleftarrow, n \in \N$ be a sequence of Thurston maps converging combinatorially to $f$. Then $f_n$ is realized for sufficiently large $n$. Moreover, letting $\tau_n \in \Teich(\R^2, A)$ be a unique fixed point of $\sigma_{f_n}$, we have $\tau_n \to \tau$, where $\tau$ is the fixed point of $\sigma_f$.
\end{corollary}

\begin{proof}
    If $f$ is a polynomial Thurston map, by Remark~\ref{rem: remark on sigma maps}  and Proposition \ref{prop: comb convergence to poly}, the equality $\sigma_{f_n} = \sigma_f$ holds for  sufficiently large $n$. 
    
    Now suppose that $f$ is transcendental.
    Let $B \subset \Teich(\R^2, A)$ be a closed ball of radius $r > 0$ centered at $\tau$, the fixed point of $\sigma_f$. Since $\Teich(\R^2, A)$ is a manifold, $B$ is compact for sufficiently small $r$. In fact, it is compact for any $r > 0$ since $\Teich(\R^2, A)$ is a \textit{path metric space} in the sense of \cite[Definition 1.7]{gromov} (for the proof see \cite[Section 11.8]{farb} and \cite[Sections 1.8bis and 1.8bis+]{gromov}). By the metric version of Hopf-Rinow theorem \cite[page 8]{gromov}, every bounded closed set of $\Teich(\R^2, A)$ is compact.

    Proposition \ref{prop: sigma is contracting} implies existence of $\varepsilon_B > 0$ such that the inequality $d(\sigma_f(\mu_1), \sigma_f(\mu_2)) \leq (1 - \varepsilon_B)d(\mu_1, \mu_2)$ is satisfied for all $\mu_1, \mu_2 \in B$. Thus, for any $\mu \in B$ we have the following:
    \begin{equation}\label{eq: trinagle ineq}
        d(\sigma_{f_n}(\mu), \tau) \leq d(\sigma_{f_n}(\mu), \sigma_f(\mu)) + d(\sigma_f(\mu), \tau) \leq \varepsilon_n + (1 - \varepsilon_B)r,
    \end{equation}
    where we know by Main Theorem~\ref{mainthm: main theorem B} that $\varepsilon_n \to 0$ as $n\to \infty$. In other words, $\sigma_{f_n}(B) \subset B$ for all sufficiently large $n$. By Proposition \ref{prop: sigma is contracting}, $\sigma^{\circ 2}_{f_n}$ is uniformly contracting on $B$ and, therefore, by the Banach fixed point theorem, when $n$ is large enough, the map $\sigma_{f_n}$ has a fixed point~$\tau_n \in B$.

    Lastly, similar to inequality (\ref{eq: trinagle ineq}), we have
    $$
        d(\tau_n, \tau) \leq  d(\sigma_{f_n}(\tau_n), \sigma_f(\tau_n)) + d(\sigma_f(\tau_n), \sigma_f(\tau)) \leq \varepsilon_n + (1 - \varepsilon_B) d(\tau_n, \tau).
    $$
    This shows that the sequence $(\tau_n)$ converges to $\tau$. 
\end{proof}

\section{Approximations of postsingularly finite entire maps}\label{sec: approximations of postsingularly finite entire maps}

The main goal of this section is to prove Theorem \ref{thm: dynamical approximations}, which is a stronger version of Main Theorem \ref{mainthm: main theorem A}. To do this, we first investigate combinatorial and topological properties of locally uniform convergence of sequences of finite type maps. In particular, we develop techniques of combinatorial nature for finding the limit of such a sequence, where all maps have the same number of singular values (see Theorem \ref{thm: nevanlinna}). This result and the methods developed in Section \ref{sec: approximations of entire thurston maps} are the key ingredients for proving the desired theorem.




\subsection{Topological properties of locally uniform convergence}\label{subsec: topological properties of locally uniform convergence}


Let $f \colon X\to Y$ be a map, then we denote its domain of definition $X$ by $\Dom(f)$ and its image $f(X)$ by $\Rg(f)$. Also, we define the \textit{kernel} of a family $\{U_j\}_{j \in J}$ of subsets of $\C$ as the set of all points $z \in \C$ having an open neighborhood $V$ such that $V \subset U_j$ for all but finitely many $j \in J$. 

Suppose that $X \subset \C$ and $(g_n)$ is a sequence of maps, where $\Dom(g_n)$ and $\Rg(g_n)$ are subsets of $\C$ for all $n \in \N$. The sequence $(g_n)$ is said to converge locally uniformly on $X$ if for every $z \in X$, there exists a neighbourhood $U$ of $z$ such that $U \subset \Dom(g_n)$ for all $n \geq N = N(U)$ and $(g_n)_{n \geq N}$ converges uniformly on $U$. In particular, we require that $X$ is a subset of the kernel~of~$(\Dom(g_n))$.

For the rest of the paper, when we refer to a holomorphic covering map $g$, we always assume that $\Dom(g)$ and $\Rg(g)$ are open subsets of $\C$. The following proposition is based on \cite[Theorem 1]{Bargmann} and provides a convenient way of analyzing the convergence of holomorphic coverings.

\begin{proposition}\label{prop: bargmann}
    Let $(g_n)$ be a sequence of holomorphic covering maps. Suppose that $(g_n)$ is uniformly convergent on a neighbourhood of a point $z \in \C$ such that the limit function is not locally constant at $z$. Denote the connected component of $z$ in the kernel of $(\Dom(g_n))$ by $X$ and the connected component of $w := \lim_{n \to \infty} g_n(z)$ in the kernel of $(\Rg(g_n))$ by $Y$. Then $(g_n)$ converges locally uniformly on $X$ to a holomorphic covering map $g\colon X \to Y$.
\end{proposition}

The following result due to \cite[Theorem 1]{Kisaka} and \cite[Lemma 1]{Krauskopf_Kriete} explains the behaviour of singular values of finite type maps that converge locally uniformly.

\begin{proposition}\label{prop: kisaka}
    Let $g_n, n \in \N$ and $g$ be entire maps of finite type such that the sequence~$(g_n)$ converges locally uniformly to $g$. Suppose that $w$ is a singular value of $g$. Then for all $n$ sufficiently large, there exists a singular value $w_n$ of $g_n$ such that the sequence $(w_n)$ converges to $w$ as $n \to \infty$.
\end{proposition}

Let $g \colon X \to Y$ be a holomorphic map, where $X, Y \subset \C$ are open subsets. Suppose that $U \subset Y$ is an open set and $v \in X$ satisfies $g(v) \in U$. Then define $V$ as the connected component of $v$ in $g^{-1}(U)$. In the case when $g|V\colon V \to U$ is bijective, we denote its inverse by $\widetilde{g}_{U, v}\colon U \to V$, i.e., $\widetilde{g}_{U, v} = (g|V)^{-1}$.

Further assume that $g\colon X \to Y$ is a covering map and $U$ is an open Jordan region such that $\overline{U} \subset Y$. In particular, $g|V \colon V \to U$ is bijective and the map $\widetilde{g}_{U, v} \colon U \to V$ is correctly defined, where $V$ is an open Jordan region satisfying $\overline{V} \subset X$. 

Suppose that $(g_n)$ is a sequence of holomorphic covering maps converging locally uniformly on $X$ to $g$. Similarly to the previous case, we can define maps $\widetilde{g}_{n, U, v} \colon V_n \to U$ that are inverses of $g|V_n \colon U \to V_n$, where $V_n$ is the connected component of $v$ in $g_n^{-1}(U)$. Now we are going to investigate relations between the maps $\widetilde{g}_{n, U, v}$ and $\widetilde{g}_{U, v}$.

\begin{lemma}\label{lemm: convergence of inverses}
    The sequence $(\widetilde{g}_{n, U, v})$ converges uniformly on~$U$ to $\widetilde{g}_{U, v}$. 
\end{lemma}

\begin{proof}
    Choose an open Jordan region $U'$ such that $\overline{U} \subset U'$ and $\overline{U'} \subset Y$. Let $V'$ be a connected component of $v$ in $g^{-1}(U')$. Once again, $V'$ is an open Jordan region such that $\overline{V} \subset V'$ and $\overline{V'} \subset X$. Denote by $\gamma$ the simple closed curve that forms the boundary of $V'$. Since $g|\overline{V'}$ is injective, then $g(\gamma) \cap \overline{U} = \emptyset$. Therefore, for sufficiently large $n$, $g_n(\gamma) \cap U  = \emptyset$. It follows that $\gamma \cap V_n = \emptyset$. Since $v \in V_n$ and the set $V_n$ is connected, we must have $V_n \subset V'$.

    Now suppose that $(\widetilde{g}_{n, U, v})$ does not converge uniformly on $U$ to $\widetilde{g}_{U, v}$. Then we can find a real number $\varepsilon >0$ such that for infinitely many $n$,  there exists $u_n \in U$ with $|\widetilde{g}_{n, U, v}(u_n) -~\widetilde{g}_{U, v}(u_n)|~>~\varepsilon$. Without loss of generality, we can assume that $(u_n)$ converges to a point $u \in \overline{U}$, and due to the previous observation, that $(\widetilde{g}_{n, U, v}(u_n))$ converges to some point $z \in \overline{V'}$. Then the inequality $|z - \widetilde{g}_{U', v}(u)| \geq \varepsilon$ holds. Since $g$ is injective on $\overline{V'}$, we have $g(z) \neq g(\widetilde{g}_{U', v}(u))$, but at the same time, 
        \begin{align*}
            g(z) = \lim\limits_{n \to \infty} g_n(\widetilde{g}_{n, U, v}(u_n)) = \lim\limits_{n \to \infty} u_n = u = g(\widetilde{g}_{U', v}(u)),
        \end{align*}
        resulting in a contradiction.
\end{proof}

\begin{lemma}\label{lemm: neighbourhood}
    The map $g_n$ is injective on $V$ for sufficiently large $n$. Moreover, there exists an open set $W$ containing $g(v)$ such that $W \subset g_n(V)$ for $n$ large enough.

\end{lemma}

\begin{proof}
    As in the proof of Lemma \ref{lemm: convergence of inverses}, we choose an open Jordan region $U'$ such that $\overline{U} \subset U'$ and $\overline{U'} \subset Y$. Lemma \ref{lemm: convergence of inverses} states that the sequence of maps $\widetilde{g}_{n, U', v}\colon U' \to V_n', n \in \N$ converges uniformly on $U'$ to the map $\widetilde{g}_{U', v} \colon U' \to V'$, where $V_n'$ and $V'$ are the connected components of~$v$ in $g^{-1}(U')$ and $g_n^{-1}(U')$, respectively. Therefore, since $\overline{V} \subset V'$, then $V$ is a subset of $V_n'$ for sufficiently large $n$, i.e., $g_n|V$ is injective. Further take $W$ to be an arbitrary open set containing $g(v)$ such that $\overline{W} \subset U$. Then the desired statement follows from the locally uniform (on $V$) convergence of $(g_n)$ to~$g$.
    
\end{proof}


Lemmas \ref{lemm: convergence of inverses} and \ref{lemm: neighbourhood} allow us to obtain the following proposition that states convergence of lifts of the same loop under converging holomorphic covering maps.


\begin{proposition}\label{prop: convergence of coverings}

Let $(g_n)$ be a sequence of holomorphic covering maps converging locally uniformly on $X \subset \C$ to a holomorphic covering map $g \colon X \to Y$. Suppose that $g_n(z_0) = g(z_0) = w_0$ for some $z_0 \in X$ and $w_0 \in Y$. Let 
$\alpha$
be a loop in $Y$ based at $w_0$ and $\beta$ be its $g$-lift starting at $z_0$. Then $\alpha \subset \Rg(g_n)$ for all sufficiently large $n$. In particular, the $g_n$-lift $\beta_n$ of $\alpha$ starting at $z_0$ is well-defined and, moreover, the following conditions hold:
\begin{enumerate}
    \item \label{it: convergence} 
    the sequence $(\beta_n)$ converges uniformly on $\I$ to $\beta$, i.e., for every $\varepsilon > 0$ we have $|\beta(t) -~\beta_n(t)| < \varepsilon$ for all $t \in \I$ and all sufficiently large $n$;



    \item \label{it: topology} there exists $N = N(\alpha)$ such that for each $n\geq N$, 
    $\beta_n$ is a loop if and only if $\beta$ is a loop. 
\end{enumerate}
    
\end{proposition}

\begin{proof}

    Lemma \ref{lemm: neighbourhood} implies that for any $w \in Y$ there exists a neighbourhood of $w$ belonging to $\Rg(g_n)$ for all but finitely many $n \in \N$. In other words, $Y$ is a subset of the kernel of $(\Rg(g_n))$.
    In particular, there exists a neighbourhood of $\alpha$ belonging to $\Rg(g_n) \cap Y$ for all sufficiently large $n$. 

    We choose open Jordan regions $V_1, V_2, \dots, V_k$ covering $\beta$ and a strictly increasing finite sequence $t_0 := 0, t_1, t_2, \dots, t_{k - 1}, t_k := 1$ of points of $\I$ satisfying the following properties for each $j = 1, 2, \dots, k$:

    \begin{itemize}
        \item $\overline{V_j} \subset X$ and $g|\overline{V_j}$ is injective;

        \item $\beta([t_{j - 1}, t_j]) \subset V_j$.

    \end{itemize}

    
    Let $U_j := g(V_j)$ for each $j = 1, 2, \dots, k$. Assume that $n$ is large enough so that $\overline{V_j}~\subset~\Dom(g_n)$ and $\overline{U_j} \subset \Rg(g_n)$ for each $j = 1, 2, \dots, k$.

    Next, we prove by induction on $l$ that $(\beta_n)$ converges to $\beta$ uniformly on $[0, t_l]$. The base case $l = 0$ is obvious since $\beta_n(0) = \beta(0) = z_0$. Now we prove that if $(\beta_n)$ converges to $\beta$ uniformly on $[0, t_l]$, then it does so uniformly on $[0, t_{l + 1}]$. Indeed, since $\beta_n(t_l) \to \beta(t_l)$, due to Lemma \ref{lemm: neighbourhood}, $\varphi_{n, U_l, \beta_n(t_l)} = \varphi_{n, U_l, \beta(t_l)}$ for sufficiently large $n$. Thus,    
    \begin{align*}
        \beta|[t_l, t_{l + 1}] &= \varphi_{U_l, \beta(t_l)} (\alpha|[t_l, t_{l + 1}]),\\
        \beta_n|[t_l, t_{l + 1}] &= \varphi_{n, U_l, \beta_n(t_l)} (\alpha|[t_l, t_{l + 1}]) = \varphi_{n, U_l, \beta(t_l)} (\alpha|[t_l, t_{l + 1}]).
    \end{align*}
    Finally, it follows from Lemma \ref{lemm: convergence of inverses} that $\beta_n$ converges to $\beta$ uniformly on $[0, t_{l + 1}]$, and this completes the proof of item (\ref{it: convergence}).



Now assume that $\beta$ is a loop. By item (\ref{it: convergence}), the point $\beta_n(1)$ converge to $z_0$ as $n \to \infty$. But by Lemma \ref{lemm: neighbourhood}, there exists a neighbourhood $V$ of $z_0$ such that $g_n|V$ is injective for all sufficiently large $n$. Thus,
it is not possible that $\beta_n(0)$ and $\beta_n(1)$ are different for any sufficiently large $n$, i.e., $\beta_n$ is a loop. If $\beta$ is a non-closed path, then similarly to the previous case, we have $\beta_n(1) \to \beta(1) \neq z_0$. Thus, we have $z_0 = \beta_n(0) \neq \beta_n(1)$ for $n$ large enough, and the proof of item (\ref{it: topology}) follows.     \end{proof}

\begin{remark}\label{rem: topological version}
    Note that in Lemmas \ref{lemm: convergence of inverses}, \ref{lemm: neighbourhood}, and Proposition \ref{prop: convergence of coverings} we never used the fact that the maps $g_n$ and $g$ are holomorphic. This is why the corresponding results still hold when $g_n$ and~$g$ are covering maps, where $\Dom(g), \Rg(g), \Dom(g_n)$, and $\Rg(g_n)$ are open subsets of~$\R^2$.
\end{remark}

    \begin{corollary}\label{corr: convergence of different lifts}
        Let $g_n, n \in \N$ and $g$ be entire maps of finite type having the same number of singular values. Let $\gamma \subset \C \setminus S_{g}$ be a simple closed curve and $\widetilde{\gamma}$ be a connected component~of~$g^{-1}(\gamma)$.
        
        Suppose that $(g_n)$ converges locally uniformly to $g$. Then for any $z \in \widetilde{\gamma}$ there exists $\varepsilon_{z, \gamma} > 0$ such that for all $n$ sufficiently large, there is a unique connected component $\widetilde{\gamma}_n$ of $g_n^{-1}(\gamma)$ satisfying $d(z, \widetilde{\gamma}_n) < \varepsilon_{z, \gamma}$. Moreover, the following is true for all $n$ large enough:

        \begin{enumerate}
            \item if the degree of the covering $g|\widetilde{\gamma}$ is finite, then

            \begin{itemize}

                \item $\widetilde{\gamma}_n$ is a simple closed curve,

                \item $\deg(g_n|\widetilde{\gamma}_n) = \deg(g|\widetilde{\gamma})$, and

                \item there exists $N = N(\varepsilon)$ such that for $n \geq N$, we have $\widetilde{\gamma}_n \subset U_{\varepsilon}(\widetilde{\gamma})$, where $U_\varepsilon(X) := \{z \in \C: d(z, X) <~\varepsilon\}$ is the \textit{$\varepsilon$}-neighbourhood of the set $X \subset \C$.

            \end{itemize}
            

            \item if the degree of the covering $g|\widetilde{\gamma}$ is infinite (i.e., $\widetilde{\gamma}$ is an unbounded curve), then $\deg(g_n|\widetilde{\gamma}_n) \to \infty$ as $n \to \infty$.


                
        \end{enumerate}
    \end{corollary}

    \begin{proof}
        Since $g$ and $g_n$ have the same number of singular values, Proposition \ref{prop: kisaka} implies that $S_{g_n}$ converges to $S_g$ in the sense of Hausdorff distance. We define holomorphic covering maps $h, h_n, n\in \N$, as $h := g| \C \setminus g^{-1}(S_g)$ and $h_n~:=~g_n|\C \setminus g_n^{-1}(S_{g_n})$. Then it is clear that $(h_n)$  converges locally uniformly on $\C \setminus g^{-1}(S_g)$ to $h$. 

        Let $w := g(z)$ and $\alpha \colon \I \to \R^2$ be a parametrization of $\gamma$, i.e., a simple loop based at~$w$ such that $\alpha(\I) = \gamma$. By Lemma \ref{lemm: neighbourhood} applied to the maps $h, h_n, n \in \N$, there exists an open bounded neighbourhood $V$ of $z$ such that for sufficiently large $n$, both $h$ and $h_n$ are injective on $V$. In particular, there exists a unique point $z_n \in V \cap g_n^{-1}(w)$. For all $n$ large enough, we take $\widetilde{\gamma}_n$ to be the connected component of $g_n^{-1}(\gamma)$ intersecting $V$, which is unique due to Lemma~\ref{lemm: neighbourhood}. In particular, it implies the existence of $\varepsilon_{z, \gamma}$ with the desired properties. By Lemma \ref{lemm: convergence of inverses}, the sequence $(z_n)$ converges to $z$ as $n\to \infty$, and, for simplicity, after pre-composing $h_n$ with an appropriate translation, we can assume that $z = z_n$. 

        Let us consider the case when the degree $d := \deg(g|\widetilde{\gamma})$ is finite. It follows that in this case the loops $\alpha^j, 1 \leq j < d$ do not lift under $g$ to a loop based at $z$, but $\alpha^d$ does so. 
        Note that by Proposition \ref{prop: convergence of coverings}, $\alpha \subset \Rg(h_n)$ for all sufficiently large $n$. Therefore, we can define 
        $\widetilde{\alpha}$ and~$\widetilde{\alpha}_n$ as $g$- and $g_n$-lifts of $\alpha^d$ based at $z$, respectively. Due to Proposition \ref{prop: convergence of coverings}, $\widetilde{\alpha}_n$ is a simple loop based at $z$ such that $(\widetilde{\alpha}_n)$ converges uniformly on $\I$ to $\widetilde{\alpha}$. In particular, $\widetilde{\gamma}_n = \widetilde{\alpha}_n(\I)$ and, therefore, $\widetilde{\gamma}_n \subset U_\varepsilon (\gamma)$ for any $\varepsilon > 0$ and sufficiently large $n$. Finally, item~(\ref{it: topology}) of Proposition~\ref{prop: convergence of coverings} implies that for $n$ large enough, the loops $\alpha^j, 1 \leq j < d$ do not lift under $g_n$ to a closed loop based at $z$. Therefore, $\deg(g_n|\widetilde{\gamma}_n) = \deg(g|\widetilde{\gamma})$. 

        Now consider the case when the degree of the covering $g|\widetilde{\gamma}$ is infinite. In this case, the loops $\alpha^j, j \in \N$ do not lift under $g$ to a loop based at $z$. Therefore, by item (\ref{it: topology}) of Proposition~\ref{prop: convergence of coverings}, the loop $\alpha^j$ does not lift under $g_n$ to a loop based at $z$ for sufficiently large~$n$. In particular, $\deg(g_n|\widetilde{\gamma}_n) \to \infty$ as $n \to \infty$.
    \end{proof}   

\begin{remark}
    Similarly to Remark \ref{rem: topological version}, the result of Corollary \ref{corr: convergence of different lifts} still holds if we assume that $g_n$ and $g$ are topologically holomorphic (and not necessarily holomorphic) maps of finite type.
\end{remark}

Propositions \ref{prop: bargmann} and \ref{prop: convergence of coverings} allow us to obtain the following topological criteria for the locally uniform convergence of finite type entire maps, all with the same number of singular values (the authors of \cite[Theorem 1]{biswas_2} prove a result of the same nature in the setting of \textit{log-Riemann surfaces}).

\begin{theorem}\label{thm: nevanlinna}
    Let $g_n, n \in \N$ and $g$ be entire maps of finite type and $B_n \supset S_{g_n}$, $B \supset S_g$ be finite subsets all of the same cardinality such that  $(B_n)$ converges to $B$ in the sense of Hausdorff distance. Let $z_0, w_0 \in \C$ be points such that  $w_0 \not \in B \cup \bigcup_{n \in \N} B_n$ and $g(z_0) = g_n(z_0) = w_0 $ for all $n \in \N$. Then $(g_n)$ converges locally uniformly to $g$ if and only if the following conditions hold:
    \begin{enumerate}
        \item \label{it: basepoint} $\lim\limits_{n \to \infty} g_n'(z_0) = g'(z_0)$;


        \item \label{it: fund_groups} for every loop $\alpha \subset \C \setminus B$ based at $w_0$, for all sufficiently large $n$, $\alpha$ lifts under $g_n$ to a loop based at $z_0$ if and only if it lifts under $g$ to a loop based at $z_0$. 
        
    \end{enumerate}
\end{theorem}

\begin{proof}



If the sequence $(g_n)$ converges locally uniformly to $g$, then condition (\ref{it: basepoint}) is obvious and condition (\ref{it: fund_groups}) easily follows from Proposition \ref{prop: convergence of coverings} applied to the sequence of holomorphic covering maps $h_n := g_n|\C \setminus g_n^{-1}(B_n)$ converging locally uniformly on $\C \setminus g^{-1}(B)$ to the holomorphic covering map $g| \C\setminus g^{-1}(B)$.

Now suppose that conditions (\ref{it: basepoint}) and (\ref{it: fund_groups}) are satisfied. It suffices to show that an arbitrary subsequence of $(g_n)$ contains a further subsequence converging locally uniformly to $g$. For simplicity's sake, we will call the initial subsequence $(g_n)$. 

\begin{claim1}
    There exist an open neighborhood $W$ of $z_0$ and a subsequence $(g_{n_k})$ of $(g_n)$ converging uniformly on $W$ to a limiting function that is not locally constant at $z_0$.
\end{claim1}

\begin{subproof}[Proof of Claim 1]
    Let $\D(w_0, r)$ be a disk contained in $\C \setminus B_n$ for all sufficiently large $n$, and $V_n$ be the connected component of $z_0$ in $g_n^{-1}(\D(w_0, r))$. By Proposition~\ref{prop: preimages without infinity}, $g_n|V_n\colon V_n \to \D(w_0, r)$ is a biholomorphism and has an inverse $\varphi_n\colon \D(w_0, r) \to V_n$ satisfying $\varphi_n(w_0) = z_0$. Note that $\varphi_n'(w_0) = 1/g_n'(z_0)$ and, since $g_n'(z_0)$ converges to $g'(z_0)$, we can assume that there exists $\lambda >0$ such that $|\varphi_n'(w_0)|~>~\lambda$ for all sufficiently large $n$. Now the Koebe 1/4-theorem implies that the disk $W := \D(z_0, r\lambda/4)$ is contained in $V_n$ for all sufficiently large $n$. In particular, the map $g_n$ is injective on $W$ for all $n$ large enough.

    Since $g_n(W) \subset \D(w_0, r)$, thus, the sequence $(g_n|W)$ forms a normal family by Montel's theorem. Hence, we can extract a converging subsequence $(g_{n_k}|W)$ from $(g_n|W)$. Clearly, the limiting function cannot be locally constant at $z_0$ since $\lim\limits_{n \to \infty} g_n'(z_0) = g'(z_0) \neq 0$. 
\end{subproof}

    Once again, we relabel $(g_{n_k})$ as $(g_n)$ to assume that $(g_n)$ converges uniformly in a neighbourhood of $z_0$ to a function that is non-constant at $z_0$. Now we apply Proposition~\ref{prop: bargmann} to the sequence of holomorphic coverings $h_n = g_n|\C \setminus g_n^{-1}(B_n)$: letting $X \subset \C$ be the connected component of $z_0$ in the kernel of $(\C\setminus g_n^{-1}(B_n))$, we see that $(h_n)$ converges locally uniformly on $X$ to a holomorphic covering map $h\colon X \to \C\setminus B$. We will now show that $g$ and $h$ coincide~on~$X$.
    
    To begin with, we observe that $h_*\pi_1(X, z_0) = g_*\pi_1(\C\setminus g^{-1}(B), z_0) \subset \pi_1(\C\setminus B, w_0)$, which follows from condition (\ref{it: fund_groups}) and Proposition \ref{prop: convergence of coverings}. By the classical theory of covering maps and condition (\ref{it: basepoint}), there exists a biholomorphism $\varphi\colon \C\setminus g^{-1}(B) \to X$ such that $g = h \circ \varphi$ on $\C\setminus g^{-1}(B)$, $\varphi(z_0) = z_0$, and $\varphi'(z_0) = 1$.

\begin{claim2}\label{it: extension}
    The map $\varphi\colon \C \setminus g^{-1}(B) \to X$ extends to a M\"{o}bius transformation.
\end{claim2}

\begin{subproof}[Proof of Claim 2]
    As $g^{-1}(B)$ is discrete and $\varphi$ is injective, every point in $g^{-1}(B)$ is a removable singularity or pole of $\varphi$. Hence, we can extend $\varphi$ to a meromorphic map from $\C$ to $\widehat{\C}$.
    It is easy to see that $\varphi$ still has to be injective. 
    But the only domains in $\widehat{\C}$ that are biholomorphic to $\C$ are obtained from $\widehat{\C}$ by removing one point. Thus, by the classification of biholomorphisms of $\C$, we see that $\varphi$ has to be a M\"{o}bius transformation.
\end{subproof}
    

Note that Claim 2 implies that $X$ is obtained from $\widehat{\C}$ by removing countably many points which have at most one accumulation point.

\begin{claim3}\label{claim: phi is identity}
    The map $\varphi$ equals $\id_{\widehat{\C}}$.
\end{claim3}

\begin{subproof}[Proof of Claim 3]
    Let us first prove that $\varphi(\infty) = \infty$. Suppose that $\varphi(\infty) = z \neq \infty$. We choose a compact set $K \subset \C$ such that $z \in \inter(K)$ and $\partial K \subset X$. Note that $(g_n)$ converges to~$h$ uniformly on $\partial K$. Let $m \in (0, +\infty)$ be the maximum of $|h|$ on $\partial K$. Then by the maximum modulus principle, for any $\varepsilon > 0$, sufficiently large $n$, and any $z \in K$, we have $|g_n(z)| \leq m+\varepsilon$. Thus, $|h|$ is bounded by $m$ on $K\cap X$.

    On the other hand, one can find a sequence $(z_n) \subset \varphi^{-1}(K)$ such that $z_n \to \infty$ and $g(z_n) \to \infty$ as $n \to \infty$. At the same time, $(\varphi(z_n)) \subset X \cap K$ is a sequence converging to $z$. This means that $|g(z_n)| = |h(\varphi(z_n))|$ is bounded by $m$ for all $n \in \N$, leading to a contradiction.  Hence, we have shown that $\varphi(\infty) = \infty$.

    Thus, $\varphi$ is an affine map satisfying $\varphi(z_0) = z_0$ and $\varphi'(z_0) = 1$, and hence, $\varphi = \id_{\hat{\C}}$.
\end{subproof}

    Claim 3 implies that $g$ coincides with $h$ on $X = \C\setminus g^{-1}(B)$. In particular, $(g_n)$ converges locally uniformly on $\C\setminus g^{-1}(B)$ to $g$. This is enough to conclude that $(g_n)$ converges locally uniformly (on $\C$) to $g$.
\end{proof}
    

\begin{remark}
    Note that existence of the map $g$ in Theorem \ref{thm: nevanlinna} is essential. For instance, consider an arbitrary sequence of polynomials $(g_n)$ satisfying the following properties:
    \begin{enumerate}
        \item \label{it: base} $g_n(0) = 0$ and $g_n'(0) = 1$ for every $n \in \N$;

        \item \label{it: sing} $S_{g_n} = \{-1, 1\}$ for every $n \in \N$;

        \item \label{it: grps} $\lim\limits_{n \to \infty} (g_n)_* \pi_1(\C \setminus g_n^{-1}(S_{g_n}), 0)$ is trivial.
    \end{enumerate}
    
    Suppose that the sequence $(g_n)$ converges locally uniformly (at least up to a subsequence) on $\C$ to an entire map $g$. Proposition \ref{prop: kisaka} implies that $S_g \subset B := \{-1, 1\}$. Then one can easily observe that the sequence of holomorphic covering maps $(g_n|\C \setminus g_n^{-1}(B))$ converges to the holomorphic covering map $g|\C \setminus g^{-1}(B)$. Therefore, it follows from Proposition \ref{prop: convergence of coverings} that $g_*\pi_1(\C \setminus g^{-1}(B), 0)$ is trivial and $g|\C \setminus g^{-1}(B)$ is a universal covering of $\C \setminus B$. It leads to a contradiction since $\C \setminus g^{-1}(B)$ cannot be biholomorphic to $\D$.
    

    Now we outline the construction of such a sequence $(g_n)$ as above but leave technical details to the reader. We can start from a holomorphic universal covering map $h\colon \D \to \C \setminus\{-1, 1\}$. Using the framework of admissible quadruples (Sections \ref{subsec: admissible quadruples} and \ref{subsec: polynomial approximations}) and following the proof of Proposition \ref{prop: polynomial approximations} applied to $h$ (viewed as a topologically holomorphic map from $\R^2 $ to $\C$), one can construct a sequence of topological polynomials $f_n \colon \R^2 \to \C$ such that $S_{f_n} = \{-1, 1\}$ and $\lim\limits_{n \to \infty} (f_n)_* \pi_1(\R^2 \setminus f_n^{-1}(S_{f_n}), 0)$ is trivial. Finally, due to Proposition \ref{prop: pulling back complex structures} we can find a sequence of orientation-preserving homeomorphisms $\varphi_n \colon \R^2 \to \C$ so that $g_n := f_n \circ \varphi_n^{-1}$ is a complex polynomial satisfying the properties (\ref{it: base}), (\ref{it: sing}), and (\ref{it: grps}).




\end{remark}

\begin{example}\label{ex: applying nevanlinna}
    In this example, we apply Theorem \ref{thm: nevanlinna} to a particular sequence of entire maps of finite type.
    Let $g_n(z) = (1 + z/n)^n, n \in \N$ and $g(z) = \exp(z)$. Note that $S_{g_n} = S_g = \{0\}$, $g_n(0) = g(0) = 1$, and $g_n'(0) = g'(0) = 1$ for all $n \in \N$. Let $\alpha$ be a simple loop based at $w_0 = 1$ such that $0$ is contained in the bounded connected component of $\C \setminus \alpha$. Then $(g_n)_* \pi_1(\C \setminus \{-n\}, 0)$ is an infinite cyclic subgroup of $\pi_1(\C \setminus \{0\}, w_0)$ generated by $[\alpha^n]$ and $g_* \pi_1(\C, 0)$ is trivial. In other words, we can apply Theorem \ref{thm: nevanlinna} to obtain the classical locally uniform limit $\exp(z) = \lim\limits_{n \to \infty} (1 + z/n)^n$.
\end{example}


\subsection{Dynamical approximations} \label{subsec: dynamical approximations} Now we are ready to formulate and prove the main result of this section, which we will then use to prove Main Theorem~\ref{mainthm: main theorem A}.

\begin{theorem}\label{thm: dynamical approximations}
    Let $f_n \colon (\R^2, A) \righttoleftarrow$, $n \in \N$ be a sequence of Thurston map converging combinatorially to a Thurston map $f\colon (\R^2, A) \righttoleftarrow$. If $g \colon (\C, B) \righttoleftarrow$ is a postsingularly finite entire map Thurston equivalent to $f\colon (\R^2, A) \righttoleftarrow$, then there exists a sequence of postsingularly finite entire maps $g_n \colon (\R^2, B_n) \righttoleftarrow$ such that
    \begin{enumerate}
        \item $g_n$ converges locally uniformly to $g$;

        \item $B_n$ converges to $B$ in the sense of Hausdorff distance;

        \item $g_n \colon (\C, B_n) \righttoleftarrow$ is Thurston equivalent to $f_n \colon (\R^2, A) \righttoleftarrow$ for    sufficiently large $n$.
    \end{enumerate}

\end{theorem}

\begin{proof}

    We can assume that $(f_n)$ converges topologically to $f$ due to Remark \ref{rem: remark on sigma maps} and Proposition \ref{prop: equivalence of convergences}. In particular, without loss of generality we suppose that there are two points $b \in \R^2$ and $t \in \R^2 \setminus A$ such that $f(b) = f_n(b) = t$ for all $n \in \N$.

    Since $f$ is Thurston equivalent to $g$, it follows from Definition \ref{def: thurston equivalence} and Proposition \ref{prop: rigidity} that there exists a unique point $\tau \in \Teich(\R^2, A)$ such that $\tau = [\varphi] = [\psi]$, $g = \varphi \circ f \circ \psi^{-1}$,  and $B = \varphi(A) = \psi(A)$, where $\varphi\colon \R^2 \to \C$ and $\psi \colon \R^2 \to \C$ are orientation-preserving homeomorphisms isotopic rel.\ $A$ to each other.


    By Corollary \ref{corr: fixed points of approximating maps}, there exists $N\in \N$ such that for all $n\geq N$ the map $\sigma_{f_n}$ has a unique fixed point $\tau_n \in \Teich(\R^2, A)$ and, moreover, the sequence $(\tau_n)$ converges to $\tau$.
    Without loss of generality, we assume $N=1$. Let $\varphi_n\colon \R^2 \to \C$ and $\psi_n \colon \R^2 \to \C$ be orientation-preserving homeomorphisms such that $\tau_n = [\varphi_n] = [\psi_n]$. Then $h_n := \varphi_n \circ f_n \circ \psi_n^{-1}$ is an entire map and the set $B_n := \varphi_n(A)$ contains its singular set $S_{h_n}$. Due to the definition of Teichm\"{u}ller distance (see Section \ref{subsec: quasiconformal maps and teichmuller spaces}) and Proposition \ref{prop: def of sigma map}, we can choose $\varphi_n$ and $\psi_n$ so that they satisfy the following conditions for all $n \in \N$: 

    \begin{itemize}
        \item $\varphi_n \circ \varphi^{-1}\colon \C \to \C$ is quasiconformal, and the dilitations  $K(\varphi_n \circ \varphi^{-1}) $ converge to 1 as $n\to \infty$,
        
        
        \item $\varphi_n(t) = \varphi(t)$ and $\varphi_n(y) = \varphi(y)$ for some point $y \neq t$, 
        
        \item $\psi_n(b) = \psi(b)$ and $h_n'(\psi_n(b)) = g'(\psi(b))$.
    \end{itemize}




        
        

    We note that $h_n$ is not necessarily postsingularly finite under this choice of $\varphi_n$ and $\psi_n$.

    It is evident that for each $n$, the homeomorphism $\varphi_n \circ \varphi^{-1}$ fixes two distinct points $\varphi(t)$ and $\varphi(y)$. Therefore, by Proposition \ref{prop: qc homeos fixing three points} the sequence $(\varphi_n)$ converges locally uniformly to~$\varphi$ as $n$ tends to infinity.  

    \begin{claim1}
        The sequence $(h_n)$ converges locally uniformly to $g$.
    \end{claim1}

    \begin{subproof}[Proof of Claim 1]
        We prove that the sequence $(h_n)$ and the map $g$ satisfy all conditions of Theorem \ref{thm: nevanlinna} with respect to the points $z_0 := \psi(b) = \psi_n(b)$, $w_0 := \varphi(t) = \varphi_n(t)$, and the sets $B_n$ and $B$.

        Indeed, $B_n \supset S_{h_n}$ converges to $B \supset S_g$ in the Hausdorff topology of $\C$ since $\varphi_n \to \varphi$ as $n \to \infty$. Due to our choices of $\varphi_n$ and $\psi_n$, the equalities $h_n'(z_0) = g'(z_0)$ and $h_n(z_0) = g(z_0) =~w_0$ are satisfied, and we have $w_0 \not \in B \cup \bigcup_{n\in \N}B_n$.
        Hence, condition (\ref{it: basepoint}) of Theorem~\ref{thm: nevanlinna} is satisfied for the  maps $h_n$. The condition (\ref{it: topology}) required in Theorem~\ref{thm: nevanlinna} easily follows from the topological convergence of $(f_n)$ to $f$ and the locally uniform convergence of $(\varphi_n)$ to $\varphi$.
        \end{subproof}


    \begin{claim2}
        The sequence $(\psi_n)$ converges locally uniformly to $\psi$.
    \end{claim2}

    \begin{subproof}[Proof of Claim 2]

        We prove that the sequence $(\psi_n^{-1})$ converges locally uniformly to $\psi^{-1}$. Let $z \in \C \setminus g^{-1}(B)$ be a point such that $\psi_n^{-1}(z) \to \psi^{-1}(z)$ as $n \to \infty$ (for instance, the point $z_0 := \psi_n(b) = \psi(b)$ satisfies this property).
        
        Take an arbitrary bounded open neighbourhood $V \subset \C\setminus g^{-1}(B)$ of $z$ such that $g$ is injective on $V$. Due to the previous claim and Lemma \ref{lemm: neighbourhood} applied to the sequence $(h_n|\C\setminus h_n^{-1}(B_n))$ and the map $g|\C\setminus g^{-1}(B)$, we have that $V \subset \C\setminus h_n^{-1}(B_n)$ and $h_n|V$ is injective for $n$ large enough. Therefore, the maps $f|\psi^{-1}(V)\colon \psi^{-1}(V) \to \varphi^{-1}(g(V))$ and $f_n|\psi_n^{-1}(V) \colon \psi_n^{-1}(V) \to \varphi_n^{-1}(h_n(V))$ are homeomorphisms. Note that for sufficiently large $n$, sets $\psi^{-1}(V)$ and $\psi^{-1}_n(V)$ intersect, since $\psi^{-1}(V)$ contain $\psi_n^{-1}(z)$ for $n$ large enough.
            
        Thus,
        \begin{align*}
            \psi^{-1}|V = (f|\psi^{-1}(V))^{-1} \circ \varphi^{-1} \circ (g|V),\\
            \psi_n^{-1}|V = (f|\psi_n^{-1}(V))^{-1} \circ \varphi_n^{-1} \circ (h_n|V).
        \end{align*}
    
    Since $h_n \to g$ and $\varphi_n \to \varphi$ locally uniformly (on $\C$ and $\R^2$, respectively) and $f_n \to f$ topologically as $n \to \infty$, we have that $(\psi^{-1}_n)$ converges locally uniformly on $V$ to $\psi^{-1}$. Applying the same argument repeatedly (starting from $z = z_0$), we show that $(\psi_n^{-1})$ converges to $\psi^{-1}$ locally uniformly on $\C \setminus g^{-1}(B)$. It is enough to conclude that $\psi_n \to \psi$ locally uniformly on $\R^2$ as $n \to \infty$.
    \end{subproof}

    Since $[\varphi_n] = [\psi_n]$, there exists an affine map $M_n$ such that $\varphi_n|A$ and $M_n \circ \psi_n|A$ coincide. Then $g_n := h_n \circ M_n^{-1} \colon (\C, B_n) \righttoleftarrow$ is a postsingularly finite entire map. It suffices to show that $(M_n)$ converges locally uniformly to $\id_{\C}$ to prove that the sequence $(g_n)$ converges locally uniformly to $g$. 



    
    Further we can assume that $|A| \geq 2$; otherwise, by Proposition \ref{prop: small postsingular set} the maps $f_n$ and $f$ are Thurston equivalent to $z \mapsto z^d$ with the same $d \geq 2$ for all sufficiently large $n$. Note that $M_n$ is an affine map satisfying $M_n(\psi_n(a)) = \varphi_n(a)$ for every $n \in \N$ and $a \in A$. Now the desired statement follows from the fact that for every $a\in A$, $\psi_n(a) \to \psi(a) = \varphi(a)$ and $\varphi_n(a) \to \varphi(a) = \psi(a)$ as $n \to \infty$.
 \end{proof}   

The next corollary is an immediate consequence of the proof of Theorem \ref{thm: dynamical approximations}.

\begin{corollary}
    Suppose that we are in the setting of Theorem \ref{thm: dynamical approximations}. Assume that the sequence $(f_n)$ converges topologically to the map $f$ and $g = \varphi \circ f \circ \psi^{-1}$, where $\varphi\colon \R^2 \to \C$ and $\psi\colon \R^2 \to \C$ are orientation-preserving homeomorphisms isotopic rel.\ $A$ to each other. Then for sufficiently large $n$ there exist orientation-preserving homeomorphisms $\varphi_n \colon \R^2 \to \C$ and $\psi_n \colon \R^2 \to \C$ such that $g_n = \varphi_n \circ f_n \circ \psi_n^{-1}$, $\varphi_n$ is isotopic rel.\ $A$ to $\psi_n$, $\varphi_n \to \varphi$ and $\psi_n \to \psi$ locally uniformly as $n \to \infty$.
\end{corollary}

The following result establishes Main Theorem \ref{mainthm: main theorem A} from the introduction and easily follows from Proposition \ref{prop: polynomial approximations} and Theorem \ref{thm: dynamical approximations}.

\begin{corollary}\label{corr: main theorem A}
    Let $g$ be a postsingularly finite entire map. Then there exists a sequence of postcritically finite polynomials $(g_n)$ converging locally uniformly to $g$ such that $g_n$ has the same singular portrait as $g$ for every $n \in \N$.
\end{corollary}

\begin{remark}
    Note that usually there is no canonical choice for a sequence of polynomials $(g_n)$ in Corollary \ref{corr: main theorem A}. Constructing different ``combinatorial approximations'' by Proposition \ref{prop: polynomial approximations} and then applying Theorem \ref{thm: dynamical approximations} lead to different sequences of polynomials.
\end{remark}

\bibliographystyle{alpha}
\bibliography{ApproximationPaper_v0.bib}

\end{document}